\documentclass[microtype]{gtpart}
\agtart
\usepackage{pinlabel}  
\usepackage{graphicx}  
\usepackage[all]{xy}
\usepackage{amscd}
\usepackage{enumitem}
\usepackage{cleveref} 
\usepackage{xy, mathabx, sseq}
\usepackage[{a4paper, total={6.5in, 8in}}]{geometry}
\theoremstyle{plain}
\newtheorem{prop}{Proposition}[subsection]
\newtheorem{theorem}[prop]{Theorem}
\newtheorem{corollary}[prop]{Corollary}

\newtheorem{lemma}[prop]{Lemma}
\newtheorem{definition}[prop]{Definition}
\newtheorem{running notations}[prop]{Running notations}
\theoremstyle{definition}
\newtheorem{remark}[prop]{Remark}
\newtheorem{example}[prop]{Example}
\newtheorem{question}[prop]{Question}
\newtheorem{observation}[prop]{Observation}

\newtheorem{runningassumption}[prop]{Running Assumption}
\newtheorem{convention}[prop]{Conventions}

\usepackage{tikz}
\usetikzlibrary{matrix}
\usepackage{tikz-cd}
\usetikzlibrary{positioning}

\title{A May-type spectral sequence for higher \\ topological Hochschild homology}
\author{Gabe Angelini-Knoll} 
\givenname{Gabe}
\surname{Angelini-Knoll}
\address{Department of Mathematics, 
		Michigan State University, 
		619 Red Cedar Road,
		C207 Wells Hall,
		East Lansing, 
		MI 48824}
\email{angelini@math.msu.edu}
\urladdr{http://users.math.msu.edu/users/angelini/}

\author{Andrew Salch}
\givenname{Andrew}
\surname{Salch}
\address{Department of Mathematics,
	         	Wayne State University,
	         	1150 F/AB,
			656 W. Kirby,
	         	Detroit, MI, 48202,
	        		USA}
\email{asalch@math.wayne.edu}
\urladdr{http://math.wayne.edu/~asalch/}

\keyword{topological Hochschild homology}
\keyword{homotopy theory}
\subject{primary}{msc2017}{55P42}
\subject{secondary}{msc2017}{55T05}

\arxivreference{1612.00541}

\newcommand{\mathbbm}[1]{\text{\usefont{U}{bbm}{m}{n}#1}}

\DeclareMathOperator{\Comm}{{\rm Comm}}
\DeclareMathOperator{\Day}{\rm Day}
\DeclareMathOperator{\Ext}{\rm Ext}
\DeclareMathOperator{\sSet}{\rm sSet}
\DeclareMathOperator{\Sets}{\rm Sets}
\DeclareMathOperator{\id}{{\rm id}}
\DeclareMathOperator{\ob}{{\rm ob}}
\DeclareMathOperator{\holim}{{\rm holim}}
\DeclareMathOperator{\hocolim}{{\rm hocolim}}
\DeclareMathOperator{\colim}{{\rm colim}}
\DeclareMathOperator{\op}{{\rm op}}
\DeclareMathOperator{\Ho}{{\rm Ho}}
\DeclareMathOperator{\Ab}{{\rm Ab}}
\DeclareMathOperator{\Tor}{{\rm Tor}}
\DeclareMathOperator{\cof}{{\rm cof}}
\DeclareMathOperator{\bN}{\mathbb{N}}

\DeclareMathOperator{\Sp}{\rm Sp}
\DeclareMathOperator{\sr}{{\rm sr}}

\newcommand{\Smash}{\wedge}
\newcommand{\bigsmash}{\bigwedge}
\begin{document}
\begin{abstract}
Given a filtration of a commutative monoid $A$ in a symmetric monoidal stable model category $\mathcal{C}$, we construct a spectral sequence analogous to the May spectral sequence whose input is the higher order topological Hochschild homology of the associated graded commutative monoid of $A$, and whose output is the higher order topological Hochschild homology of $A$. We then construct examples of such filtrations and derive some consequences: for example, given a connective commutative graded ring $R$, we get an upper bound on the size of the $THH$-groups of $E_{\infty}$-ring spectra $A$ such that $\pi_*(A) \cong R$. 
\end{abstract} 
\maketitle
\tableofcontents
\section{Introduction.}
Suppose $A = F_0A \supseteq F_1A \supseteq F_2A\supseteq \dots$ is a filtered augmented $k$-algebra, where $k$ is a field.
In his 1964 Ph.D. thesis,~\cite{MR2614527}, J.~P.~May sets up a spectral sequence with input $\Ext^{*,*}_{E^*_0A}(k,k)$
and which converges to $\Ext^{*}_{A}(k,k)$. Here $E^*_0A = \oplus_{n\geq 0} F_nA/F_{n+1}A$ is the associated graded algebra of the filtration of $A$.

In the present paper, we construct an analogous spectral sequence for topological \linebreak
Hochschild homology and its ``higher order'' generalizations (as in~\cite{MR1755114}). 
Given a filtered $E_{\infty}$-ring spectrum $A$, we construct a spectral sequence
\begin{equation}\label{ss 2304} E^1_{*,*} \cong THH_{*,*}(E^*_0A) \Rightarrow THH_*(A).\end{equation}
Here $E^*_0A$ is the {\em associated graded $E_{\infty}$-ring spectrum of $A$}; part of our work in this paper is to define this ``associated graded $E_{\infty}$-ring spectrum,'' and prove that it has good formal properties and useful examples (eg Whitehead towers; see~\eqref{whitehead tower 3408}, below).

More generally: given any connective generalized homology theory $E_*$ (see Definition \ref{def of gen hom thy}), and given any simplicial finite set $X_{\bullet}$, we construct a spectral sequence
\begin{equation}\label{ss 2305} E^1_{*,*} \cong E_{*,*}(X_{\bullet}\otimes E^*_0A) \Rightarrow E_*(X_{\bullet}\otimes A).\end{equation}
We recover spectral sequence~\eqref{ss 2304} as a special case of~\eqref{ss 2305} by letting $E_* = \pi_*$ and letting $X_{\bullet}$ be a simplicial model for the circle $S^1$. 

We formulate a definition (see Definition~\ref{def of dec filt obj}) of a ``filtered $E_{\infty}$-ring spectrum'' which is sufficiently well-behaved that we can actually construct a spectral sequence of the form \eqref{ss 2305}, identify its $E^1$- and $E^{\infty}$-terms and prove its multiplicativity and convergence properties. Actually our constructions and results work in a somewhat wider level of generality than commutative ring spectra: we fix a symmetric monoidal stable model category $\mathcal{C}$ satisfying some reasonable hypotheses (spelled out in Running Assumptions~\ref{ra:1} and~\ref{ra:2}), and we work with filtered commutative monoid objects in $\mathcal{C}$. In the special case where $\mathcal{C}$ is the category of symmetric spectra in pointed simplicial sets, in the sense of~\cite{MR1695653} and~\cite{schwedebook}, the commutative monoid objects are equivalent to $E_{\infty}$-ring spectra. Our hope is that our framework is sufficiently general that an interested reader could 
also apply it to monoidal model categories of equivariant, motivic, and/or parametrized spectra.

The main difficulty in constructing the spectral sequence \eqref{ss 2305} at this level of generality is identifying the $E^1_{*,*}$-page. We refer to the theorem identifying the $E^1_{*,*}$-page as the \emph{fundamental theorem of the May filtration}, and briefly it can be described in words using the slogan: ``higher order Hochschild homology commutes with passage to the associated graded commutative ring spectrum." This theorem does not follow easily by categorical properties, and in fact the bulk of Section \ref{fund thm section} consists of a proof of this theorem. 

We observe in Appendix \ref{THH-May w coeff}  that with the right adjustments, one can construct a version of spectral sequence~\eqref{ss 2305} with coefficients in a filtered symmetric $A$-bimodule $M$:
\begin{equation}\label{ss 2305a} E^1_{*,*} \cong E_{*,*}(X_{\bullet}\otimes (E^*_0A, E^*_0M)) \Rightarrow E_*(X_{\bullet}\otimes (A,M)),\end{equation}
and as a special case,
\begin{equation}\label{ss 2304a} E^1_{*,*} \cong E_{*,*}THH(E^*_0A, E^*_0M) \Rightarrow E_*THH(A,M).\end{equation}

Some of the most important cases of filtered commutative ring spectra, or filtered commutative monoid objects in general, are those which arise from Whitehead towers: given a cofibrant connective commutative monoid in symmetric spectra, we construct a filtered commutative monoid
\begin{equation}\label{whitehead tower 3408} A = \tau_{\geq 0}A \leftarrow  \tau_{\geq 1} A  \leftarrow \tau_{\geq 2} A  \leftarrow \dots \end{equation}
where each map is a cofibration in $\mathcal{C}$ and the induced map $\pi_n(\tau_{\geq m} A) \rightarrow \pi_n(\tau_{\geq m-1} A)$ is an isomorphism if $n\geq m$, and $\pi_n(\tau_{\geq m}A) \cong 0$ if $n<m$. 
While the homotopy type of $\tau_{\geq m}A$ is very easy to construct, 
it takes us some work to construct a sufficiently rigid \emph{multiplicative} model for the Whitehead tower~\eqref{whitehead tower 3408}; this is the content of Theorem~\ref{post filt}.

If $\mathcal{C}$ is the category of symmetric spectra in pointed simplicial sets, then the associated graded ring spectrum of the Whitehead tower~\eqref{whitehead tower 3408} is the generalized Eilenberg-Mac Lane ring spectrum $H\pi_*(A)$ of the graded ring $\pi_*(A)$. 
Consequently we get a spectral sequence
\begin{equation}\label{ss 2305b} E^1_{*,*} \cong E_{*,*}(X_{\bullet}\otimes H\pi_*A) \Rightarrow E_*(X_{\bullet}\otimes A),\end{equation}
and as a special case,
\begin{equation}\label{ss 2304b} E^1_{*,*} \cong THH_{*,*}(H\pi_*A) \Rightarrow THH_*(A).\end{equation}

Many explicit computations are possible using spectral sequence~\eqref{ss 2304b} and its generalizations with coefficients in a bimodule (see Appendix \ref{THH-May w coeff}). For example, in~\cite{161200548}, the first author uses these spectral sequences to compute topological Hochschild homology of the algebraic $K$-theory spectra of a large class of finite fields. 

In the present paper, in lieu of explicit computations using our new spectral sequences, we point out that the mere existence of these spectral sequences implies an upper bound on the size of the topological Hochschild homology groups of a ring spectrum: namely, if $R$ is a graded-commutative ring and $X_{\bullet}$ is a simplicial finite set and $E_*$ is a generalized homology theory, then for any $E_{\infty}$-ring spectrum $A$ such that $\pi_*(A) \cong R$,  $E_*(X_{\bullet}\otimes A)$ is a subquotient of $E_*(X_{\bullet} \otimes HR)$. Here we write $HR$ for the generalized Eilenberg-Mac Lane spectrum with $\pi_*(HR)\cong R$ as graded rings. 

Consequently, in Theorem~\ref{upper bound thm} we arrive at the slogan: {\em 
the topological Hochschild homology of $A$ is bounded above by the topological Hochschild homology of $H\pi_*(A)$.} 
This lets us extract information about the topological Hochschild homology of $E_{\infty}$-ring spectra $A$ from information depending {\em only on the ring $\pi_*(A)$ of homotopy groups of $A$.}
We demonstrate how to apply this idea in Theorem~\ref{polynomial case 120} and its corollaries, by working out the special case where $R = \mathbb{Z}_{(p)}[x]$ for some prime $p$, with $x$ in positive grading degree $2n$. We get, for example, that for any $p$-local finite-type $E_{\infty}$-ring spectrum $A$ such that $\pi_*(A) \cong \mathbb{Z}_{(p)}[x]$,
the Poincar\'{e} series of the mod $p$ topological Hochschild homology $(S/p)_*(THH(A))$ satisfies the inequality
\[ \sum_{i\geq 0} \left( \dim_{\mathbb{F}_p} (S/p)_*(THH(A))\right) t^i \leq 
 \frac{(1 + (2p-1)t)(1 + (2n+1)t)}{(1 - 2nt)(1 - 2pt)}, \]
where we interpret $\le$ as in Definition \ref{partial order on power series}. Furthermore, if $p$ does not divide $n$, then $THH_{2i}(A) \cong 0$ for all $i$ congruent to $-p$ modulo $n$ such that $i\leq pn-p-n$, and
$THH_{2i}(A) \cong 0$ for all $i$ congruent to $-n$ modulo $p$ such that $i\leq pn-p-n$. 
In particular, $THH_{2(pn-p-n)}(A) \cong 0$. If $p$ divides $n$, then $THH_{i}(A)\cong 0$, unless $i$ is congruent to $-1,0,$ or $1$ modulo $2p$. 

As a specific example, consider the $p$-local Adams summand $\ell$, which satisfies the property that $p$ does not divide $n=p-1$ for any prime $p$. The theorem then states that $THH_{2i}(\ell)\cong 0$ for all $i$ congruent to $-1$ modulo $p-1$ and all $i$ congruent to $1$ modulo $p$ such that $i\leq p^2-3p+1$, which agrees with the computation of $THH_{*}(\hat{\ell}_p)$ in these degrees due to Angeltveit-Hill-Lawson \cite[Theorem  2.6]{AHL10}. 

In Section \ref{convenient model cat for filtered objects}, we further the development of filtered objects in a model category satisfying Running Assumptions \ref{ra:1} and \ref{ra:2} along with Running Assumption \ref{ra:3}. This theory is used to construct multiplicative Whitehead towers as cofibrant decreasingly filtered commutative monoids in symmetric spectra in pointed simplicial sets (Theorem \ref{post filt}). In Appendix B, we discuss the Bousfield-Kan spectral sequence  (Theorem \ref{bkss existence thm}) in order to address a technical lemma (Lemma \ref{main lemma}) needed to construct the multiplicative Whitehead tower. 

There is some precedent for spectral sequence~\eqref{ss 2304}: when $A$ is a filtered commutative ring (rather than a filtered commutative ring spectrum), M.~Brun constructed a spectral sequence of the form~\eqref{ss 2304} in the paper~\cite{MR1750729}. In~\cite{MR1823499}, Brun also studies $THH$ of filtered ring spectra, but using the technology of FSPs rather than our general approach, and stops short of producing the THH-May spectral sequence or the multiplicative Whitehead filtration. In Theorem~2.9 of the preprint~\cite{angeltveitpreprint}, V.~Angeltveit remarks that a version of spectral sequence~\eqref{ss 2304} exists for commutative ring spectra by virtue of a lemma in~\cite{MR1750729} on associated graded FSPs of filtered FSPs; filling in the details to make this spectral sequence have the correct $E^1$-term, $E^{\infty}$-term, convergence properties, and multiplicativity properties takes a lot of work, and even aside from the substantially greater level of generality of the results in the present paper (allowing $X_{\bullet}\otimes A$ and not just $S^1\otimes A$, working with commutative monoids in symmetric monoidal model categories rather than any particular model for ring spectra, working with coefficient bimodules as in Appendix \ref{THH-May w coeff}), we think it is valuable to add these very nontrivial details to the literature. 

\subsubsection{Acknowledgements.}
We are grateful to C.~Ogle and Ohio State University for their hospitality in hosting us during a visit to talk about this project and A.~Blumberg for a timely and useful observation. We are indebted to A.~Lindenstrauss and B.~Richter for their careful reading and insightful questions which led to improvements in this paper.
We also thank J.~Greenlees for his editorial help, and J.~Rognes as well as the anonymous referee for their careful reading and very perceptive observations and suggestions which helped us to improve this paper. We also thank the anonymous referee for their generosity and patience in reviewing such a long paper. 
The first author would like to thank A.~Lindenstrauss, T.~Gerhardt, and C.~Malkiewich for helpful conversations on the material in this paper. 

\section{Conventions and running assumptions.} \label{conv and ra} 
\begin{convention} \label{cof assump}
By convention, the ``cofiber of $f\co X\longrightarrow Y$" will mean that $f$ is a cofibration and we are forming the pushout $Y\coprod_X 0$ in the given pointed model category. Also, by convention, we will write $Y/X$ as shorthand for the cofiber of $f\co X\longrightarrow Y$.
\end{convention} 
\begin{convention} 
By convention, given a coproduct $\coprod_{i\in I} X_i$ of objects $X_i$ in a cocomplete category $\mathcal{A}$, we will refer to the map $X_j\hookrightarrow \coprod_{i\in I} X_i$ for $j\in I$, given by definition of the coproduct, as the inclusion map. When $\mathcal{A}$ is a subcategory of $\mathcal{B}$, we will also refer to the evident functor $\mathcal{A}\hookrightarrow \mathcal{B}$ as an inclusion. Since one use of inclusion refers to map between objects in a category $\mathcal{A}$ and the second use of inclusion refers to a functor between categories, no confusion should arise. 
\end{convention}
We will write $\Comm(\mathcal{C})$ for the category of commutative monoid objects in a symmetric monoidal category $\mathcal{C}$, we will write $s\mathcal{C}$ for the category of simplicial objects in $\mathcal{C}$, and we will write $\Smash$ for the symmetric monoidal product in a symmetric monoidal category $\mathcal{C}$, since the main example we have in mind is the category of symmetric spectra, in which the symmetric monoidal product is the smash product. We will write $S^1_{\bullet}$ throughout for the minimal simplicial model for the circle $\Delta[1]/\delta \Delta[1]$, which is important for defining topological Hochschild homology.  
\begin{runningassumption} \label{ra:1}
Throughout, let $\mathcal{C}$ be a 
left proper stable model category
equipped with the structure of a symmetric monoidal model category in the sense of
\cite{MR1734325}, satisfying the following axioms:
A model structure (necessarily unique) on $\Comm(\mathcal{C})$ exists in which 
weak equivalences and fibrations are created by the forgetful functor \linebreak
$\Comm(\mathcal{C})\rightarrow\mathcal{C}$. 
The forgetful functor $\Comm(\mathcal{C})\rightarrow \mathcal{C}$ commutes with geometric realization of simplicial objects and sends cofibrant objects to cofibrant objects. 
Geometric realization of simplicial 
cofibrant objects in $\mathcal{C}$ commutes with
the monoidal product, ie, if $X_{\bullet},Y_{\bullet}$ are simplicial 
cofibrant objects of 
$\mathcal{C}$, then the canonical comparison map
\[ \left| X_{\bullet}\Smash Y_{\bullet}\right| \rightarrow
\left| X_{\bullet}\right| \Smash \left| Y_{\bullet}\right|\]
is a weak equivalence in $\mathcal{C}$. 
We will say a class of morphisms $\mathcal{M}$ in $\mathcal{C}$ is called \emph{retractile} if whenever a composite $A\overset{f}{\rightarrow}B\overset{g}{\rightarrow} C$ is in $\mathcal{M}$ then $A\overset{f}{\rightarrow}B$ is in $\mathcal{M}$. We additionally assume that $\mathcal{C}$ comes equipped with a class of morphisms, which we will refer to as \emph{level-wise cofibrations}, which are closed under composition, and such that the cofibrations in $\mathcal{C}$ are contained in the class of level-wise cofibrations, and such that the class of level-wise cofibrations is retractile. We also assume that if $X$ is a cofibrant object in $\mathcal{C}$ and $f$ is a level-wise cofibration in $\mathcal{C}$, then $X\wedge f$ is a level-wise cofibration in $\mathcal{C}$. 
\end{runningassumption} 

Here are a few immediate consequences of these assumptions about $\mathcal{C}$:
Since being cofibrantly generated is part of the definition of a monoidal model category in \cite{MR1734325}, $\mathcal{C}$ is cofibrantly generated and hence can be equipped with functorial factorization systems.
We assume that a choice of functorial factorization has been made and we will use it implicitly whenever a cofibration-acyclic-fibration or
acyclic-cofibration-fibration factorization is necessary.
Smashing with any given object is a left adjoint, hence preserves colimits.
Smashing with any given cofibrant object is a left Quillen functor, hence preserves cofibrations and acyclic cofibrations, and by Ken Brown's Lemma, preserves weak equivalences between cofibrant objects. 

Since $\mathcal{C}$ is assumed left proper, a homotopy cofiber of any
map $f: X \rightarrow Y$ between cofibrant objects in $\mathcal{C}$ can be computed
by factoring $f$ as $f = f_2\circ f_1$ with $f_1: X \rightarrow \tilde{Y}$ a cofibration and $f_2: \tilde{Y}\rightarrow Y$ an acyclic fibration, and then taking the
pushout of the square
\[\xymatrix{ X\ar[r]^{f_1} \ar[d] & \tilde{Y} \\ 0. &  }\]
In particular, if $f$ is already a cofibration, the pushout map
$Y \rightarrow Y\coprod_X 0$ is a homotopy cofiber of $f$.

The added assumption that $\mathcal{C}$ comes equipped with a retractile class of morphisms called the level-wise cofibrations was necessary to resolve a technical issue pointed out to the authors by Birgit Richter and Ayelet Lindenstrauss. 
We recognize that it complicates the running assumptions, however it is satisfied in the main example of interest (as discussed below) and provides the THH-May spectral sequence with some desired properties. The level-wise cofibrations are used only in two places in the paper, in Observation~\ref{observation on cofibrancy conditions} and in Remark~\ref{gabes remark on cofibrancy of degree 0 quotient}.

\begin{runningassumption} \label{ra:2}
In addition to Running Assumption \ref{ra:1}, we assume our model category $\mathcal{C}$ satisfies the following condition: a map $X_{\bullet}\rightarrow Y_{\bullet}$ in the category of simplicial objects in $\mathcal{C}$ is a Reedy cofibration between Reedy cofibrant objects whenever the following all hold:
\begin{enumerate}
\item{} \label{it1}The object $X_n$ in $\mathcal{C}$ is cofibrant for each $n$. 
\item{} \label{it2} Each degeneracy map $s_i\co X_n\rightarrow X_{n+1}$ and $s_i\co Y_n\rightarrow Y_{n+1}$ is a level-wise cofibration in $\mathcal{C}$. 
\item{} \label{it3} Each induced map $X_n \rightarrow Y_n$ is a cofibration in $\mathcal{C}$. 
\end{enumerate}
A consequence of this assumption is that the geometric realization of a map of simplicial objects in $\mathcal{C}$ satisfying Item \eqref{it1}, Item \eqref{it2}, and Item \eqref{it3} is a cofibration.  
\end{runningassumption}
The main motivating example of a category $\mathcal{C}$ satisfying Running Assumption \ref{ra:1} is the category of symmetric spectra in pointed simplicial sets  $\sSet_*$, 
denoted $\Sp_{\sSet_*}$, equipped with the positive flat stable model structure. In this case, then $\Comm(\mathcal{C})$ is the category of commutative ring spectra and it is known to be equivalent to the category of $E_{\infty}$-ring spectra, see Corollary 4.8 in \cite{MR3666740}. 
The existence of the desired model structure on $\Comm(\mathcal{C})$ is proven in Theorem 4.1 of \cite{MR1734325}. 
The fact that the forgetful functor $\Comm(\mathcal{C})\rightarrow\mathcal{C}$ commutes with geometric realization in the positive flat stable model structure on $\mathcal{C}$ is a consequence of~\cite[Thm. 1.6]{MR2580430}. 
The fact that there exists a model structure on $\Comm(\mathcal{C})$ created by the forgetful functor and that the forgetful functor preserves cofibrations with cofibrant source is a consequence of Theorem 5.7 in \cite{MR3666740}.  
The category $\Sp_{\sSet_*}$ with the positive flat stable model structure satisfies Running Assumption \ref{ra:2}, as the authors prove in \cite{161106215}. 
To see that the level-wise cofibrations in symmetric spectra of pointed simplicial sets are retractile, note that the level-wise cofibrations are simply the levelwise cofibrations of pointed simplicial sets, hence levelwise monomorphisms of pointed simplicial sets, and monomorphisms are retractile in any category. If $f:A\rightarrow B$ is a flat cofibration in symmetric spectra then the pushout product $f\square -$ preserves level-wise cofibrations, so in particular all flat cofibrations are level-wise cofibrations and for any flat cofibrant object $X$, the functor $X\wedge -$ preserves level-wise cofibrations (by Proposition 2.8 of~\cite{schwedebook}, for example). 

\section{Construction of the spectral sequence.} \label{construction of ss section}
\subsection{Filtered commutative monoids and associated graded commutative monoids.}
\begin{definition}\label{def of dec filt obj}
By a {\em cofibrant decreasingly filtered object in $\mathcal{C}$} we mean
 a sequence of cofibrations in $\mathcal{C}$
\[ \dots \stackrel{f_3}{\longrightarrow} I_2 \stackrel{f_2}{\longrightarrow} I_1 \stackrel{f_1}{\longrightarrow} I_0, \]
such that each object $I_i$ is cofibrant.
\end{definition}

\begin{definition}\label{def of dec filt comm mon}
By a {\em cofibrant decreasingly filtered commutative monoid in $\mathcal{C}$} we mean: a cofibrant decreasingly filtered object
$ \dots \stackrel{f_3}{\longrightarrow} I_2 \stackrel{f_2}{\longrightarrow} I_1 \stackrel{f_1}{\longrightarrow} I_0 $
in $\mathcal{C}$,
and for every pair of natural numbers $i,j\in\mathbb{N}$, a map in $\mathcal{C}$
\[ \rho_{i,j}\co  I_i\Smash I_j \rightarrow I_{i+j},\]
and a map $\eta\co  \mathbbm{1}\rightarrow I_0$,
satisfying the axioms listed below. For the sake of listing the axioms concisely, it will be useful to have the following notation:
if $i^{\prime}\leq i$, we will write $f_i^{i^{\prime}}\co  I_i\rightarrow I_{i^{\prime}}$ for the composite
\begin{equation}\label{composite} f_i^{i^{\prime}} = f_{i^{\prime}+1} \circ f_{i^{\prime}+2} \circ \dots \circ f_{i-1}\circ f_i.\end{equation}

Here are the axioms we require:
\begin{description}[before={\renewcommand\makelabel[1]{\bfseries ##1.}}]
\item[{\bf Compatibility}] For all $i,j,i^{\prime},j^{\prime}\in \mathbb{N}$ with $i^{\prime}\leq i$ and $j^{\prime}\leq j$, the following diagram commutes:
\[\xymatrix{
I_i \Smash I_j \ar[d]_{f_i^{i^{\prime}}\Smash f_j^{j^{\prime}}} \ar[r]^{\rho_{i,j}} & I_{i+j}\ar[d]^{f_{i+j}^{i^{\prime}+j^{\prime}}} \\
I_{i^{\prime}}\Smash I_{j^{\prime}} \ar[r]^{\rho_{i^{\prime} , j^{\prime}}} & I_{i^{\prime} + j^{\prime}}.}\]

\item[{\bf Commutativity}] For all $i,j\in \mathbb{N}$, the diagram
\[\xymatrix{ I_i\Smash I_j \ar[d]_{\chi_{I_i,I_j}}\ar[rd]^{\rho_{i,j}} & \\
I_j\Smash I_i \ar[r]_{\rho_{j,i}} & I_{i+j} }\]
commutes, where $\chi_{I_i,I_j}\co  I_i\Smash I_j \stackrel{\cong}{\longrightarrow} I_j\Smash I_i$
is the symmetry isomorphism in $\mathcal{C}$.
\item[{\bf Associativity}] For all $i,j,k\in\mathbb{N}$, the following diagram commutes:
\[\xymatrix{ I_i\Smash I_j\Smash I_k \ar[r]^{\id_{I_i}\Smash \rho_{j,k}} \ar[d]_{\rho_{i,j}\Smash \id_{I_k}} & I_i\Smash I_{j+k}\ar[d]^{\rho_{i,j+k}} \\
I_{i+j}\Smash I_k\ar[r]_{\rho_{i+j,k}} & I_{i+j+k} .}\]

\item [{\bf Unitality}] For all $i\in \mathbb{N}$, the diagram
\[\xymatrix{ \mathbbm{1}\Smash I_i\ar[rd]^{\cong} \ar[d]_{\eta\Smash \id_{I_i}} & \\ I_0\Smash I_i\ar[r]_{\rho_{0,i}} & I_i }\]
commutes, where the map marked $\cong$ is the (left-)unitality isomorphism in $\mathcal{C}$.
 \item [{\bf Cofibrant in degree 0}] The commutative monoid $I_0$ is cofibrant in $\Comm(\mathcal{C})$.
\item [{\bf Cofibrancy of degree 0 quotient}] The composite map $S\rightarrow I_0\to I_0/I_1$ is a level-wise cofibration in $\mathcal{C}$.
\end{description}
\end{definition}
Note that, in the ``Cofibrancy of degree 0 quotient'' condition, we do not require that $I_0/I_1$ be cofibrant in $\Comm(\mathcal{C})$, but only that the map $S\rightarrow I_0\rightarrow I_0/I_1$ is a level-wise cofibration in $\mathcal{C}$. We hope that Observation~\ref{observation on cofibrancy conditions} and Remark~\ref{remark on cofibrancy conditions}  will be helpful to the reader who is wondering about the role of the ``Cofibrancy of degree 0 quotient'' condition.

\begin{remark}\label{smith ideals remark}
Note that, if $I_{\bullet}$ is a cofibrant decreasingly filtered commutative monoid in $\mathcal{C}$, then 
$I_0$ is a cofibrant commutative monoid in $\mathcal{C}$, with multiplication map
$\rho_{0,0}\co  I_0\Smash I_0\rightarrow I_0$ and unit map $\eta\co  \mathbbm{1}\rightarrow I_0$.
The objects $I_i$ for $i>0$ do not receive commutative monoid structures from the structure
of $I_{\bullet}$, but instead play a role analogous 
to that of the nested sequence of powers of an ideal in a commutative ring. If we neglect the commutativity axiom from Definition~\ref{def of dec filt comm mon}, then the special case $\dots \stackrel{\id}{\longrightarrow} I \stackrel{\id}{\longrightarrow} I \rightarrow R$ of Definition~\ref{def of dec filt comm mon} recovers the definition of a Smith ideal. Leaving the commutativity axiom in Definition~\ref{def of dec filt comm mon} intact, we recover the notion of a commutative Smith ideal, as studied in \cite{MR3666740}.
\end{remark}

\begin{definition} \label{def of Hausdorff filt}
Suppose $I_{\bullet}$ is a cofibrant decreasingly filtered commutative monoid in $\mathcal{C}$. 
We shall say that $I_{\bullet}$ is {\em Hausdorff} if 
$\holim_n I_n\simeq 0$.
We shall say that $I_{\bullet}$ is {\em finite} if there exists
some $n\in\mathbb{N}$ such that $f_{m}\co  I_m\rightarrow I_{m-1}$ is a
weak equivalence for all $m>n$. 
\end{definition}

\begin{remark} \label{rem filt comm mon}
Definition~\ref{def of dec filt comm mon} has the advantage of concreteness, but if we are willing to temporarily neglect the cofibrancy assumption in Definition~\ref{def of dec filt comm mon}, 
then there is an equivalent, more concise definition of a decreasingly filtered commutative monoid. Observe that the the data of a decreasingly filtered commutative monoid is the same as the data of a lax symmetric monoidal functor
$ I_{\bullet}\co\mathbb{N}^{\op} \longrightarrow \mathcal{C},$
where $\mathbb{N}^{\op}$ is the opposite category of $\mathbb{N}$, viewed as a partially ordered set, and equipped with a symmetric monoidal structure with addition as the symmetric monoidal product and $0$ as the unit. 

Recall that due to Day~\cite{Day}, the full subcategory of lax symmetric monoidal functors in $\mathcal{C}^{\mathbb{N}^{\op}}$ is equivalent to the category $\Comm \mathcal{C}^{\mathbb{N}^{\op}}$ of commutative monoid objects in the symmetric monoidal category $(\mathcal{C}^{\mathbb{N}^{\op}}, \otimes_{\text{Day}}, \mathbbm{1}_{\text{Day}})$ where $\otimes_{\text{Day}}$ is the (unenriched) Day convolution symmetric monoidal product and $\mathbbm{1}_{\text{Day}}$ is the unit of this symmetric monoidal product (see Day \cite{Day} for these constructions). In sum, specifying the data of a decreasingly filtered commutative monoid, without the cofibrancy condition, is the same as specifying an object in $\Comm \mathcal{C}^{\mathbb{N}^{\op}}$. 
\end{remark}
Remark \ref{rem filt comm mon} does not address the cofibrancy conditions needed for an object in the category $\Comm \mathcal{C}^{\mathbb{N}^{\op}}$ to be a cofibrant decreasingly filtered commutative monoid in $\mathcal{C}$ in the sense of Definition \ref{def of dec filt obj}. We will discuss this in detail in Section \ref{convenient model cat for filtered objects}. 
\begin{definition} {\bf (The associated graded monoid.)} \label{assoc graded monoid}
Let $I_{\bullet}$ be a cofibrant decreasingly filtered commutative monoid in $\mathcal{C}$.
By $E_0^*I_{\bullet}$, the {\em associated graded commutative monoid of $I_{\bullet}$}, we mean 
the graded commutative monoid object in $\mathcal{C}$ defined as follows: as an object of $\mathcal{C}$,  
\[ E_0^*I_{\bullet} =\coprod_{n\in\mathbb{N}} I_n/I_{n+1}. \]
The unit map $\mathbbm{1}\rightarrow E_0^*I_{\bullet}$ is the composite
$ \mathbbm{1}\stackrel{\eta}{\longrightarrow} I_0 \rightarrow I_0/I_1\hookrightarrow E_0^*I_{\bullet} .$

The multiplication on $E_0^*I_{\bullet}$ is given as follows. 
Since the smash product commutes with colimits, hence with coproducts, to specify a map
$ E_0^*I_{\bullet} \Smash E_0^*I_{\bullet}\rightarrow E_0^*I_{\bullet}$
it suffices to specify a component map 
\[ \nabla_{i,j}\co  I_i/I_{i+1}\Smash I_j/I_{j+1}\rightarrow E_0^*I_{\bullet}\]
for every $i,j\in\mathbb{N}$.
We define such a map $\nabla_{i,j}$ as follows:
first, we have the commutative square
\[\xymatrix{ I_{i+1}\Smash I_j \ar[r]^{\rho_{i+1,j}} \ar[d]_{f_{i+1}\Smash \id_{I_j}} & I_{i+j+1}\ar[d]^{f_{i+j+1}} \\
I_i\Smash I_j \ar[r]^{\rho_{i,j}} & I_{i+j} }\]
so, since the vertical maps are cofibrations by Definition \ref{def of dec filt comm mon}, 
we can take vertical cofibers to get a map
\[ \tilde{\nabla}_{i,j}\co  I_i/I_{i+1}\Smash I_j\rightarrow I_{i+j}/I_{i+j+1}, \]
which is well-defined by Running Assumption \ref{ra:1}.

Now we have the commutative diagram
\[\xymatrix{ I_{i+1}\Smash I_{j+1}\ar[rd]^{\id_{I_{i+1}}\Smash f_{j+1}}\ar[dd]^{f_{i+1}\Smash \id_{I_{j+1}}}\ar[rr]^{\rho_{i+1,j+1}} & & I_{i+j+2}\ar[rd]^{f_{i+j+2}}\ar[dd]^(.7){f_{i+j+2}} & \\
 & I_{i+1}\Smash I_j \ar[dd]^(.7){f_{i+1}\Smash \id_{I_j}} \ar[rr]^(.35){\rho_{i+1,j}} & & I_{i+j+1}\ar[dd]^{f_{i+j+1}} \\
I_i\Smash I_{j+1}\ar[rd]^{\id_{I_i}\Smash f_{j+1}} \ar[rr]^(.4){\rho_{i,j+1}} \ar[dd] & & I_{i+j+1}\ar[rd]^{f_{i+j+1}}\ar[dd] & \\
 & I_i\Smash I_j\ar[rr]^(.4){\rho_{i,j}} \ar[dd] & & I_{i+j} \ar[dd] \\
I_i/I_{i+1}\Smash I_{j+1} \ar[rd]_{\id_{I_i/I_{i+1}}\Smash f_{j+1}} \ar[rr]^(.4){\tilde{\nabla}_{i,j+1}} & & I_{i+j+1}/I_{i+j+2}\ar[rd]^{0} & \\
 & I_i/I_{i+1}\Smash I_j \ar[rr]^{\tilde{\nabla}_{i,j}} & & I_{i+j}/I_{i+j+1} }\]
in which the columns are cofiber sequences.
So we have a factorization of the composite map $\tilde{\nabla}_{i,j}\circ \left( \id_{I_i/I_{i+1}}\Smash f_{j+1}\right)$
through the zero object by Running Assumption \ref{ra:1} 
Thus, we have the commutative square
\[\xymatrix{ I_i/I_{i+1}\Smash I_{j+1}\ar[d]^{\id_{I_i/I_{i+1}}\Smash f_{j+1}} \ar[r] & 0 \ar[d] \\
I_i/I_{i+1}\Smash I_j \ar[r]^{\tilde{\nabla}_{i,j}} & I_{i+j}/I_{i+j+1} }\]
and, taking vertical cofibers, a map
\[ I_i/I_{i+1}\Smash I_j/I_{j+1}\rightarrow I_{i+j}/I_{i+j+1},\]
which we compose with the inclusion map
$I_{i+j}/I_{i+j+1}\hookrightarrow E_0^*I_{\bullet}$ 
to produce
our desired map $\nabla_{i,j}\co  I_i/I_{i+1}\Smash I_j/I_{j+1}\rightarrow E_0^*I_{\bullet}$. (Note that all these maps are defined in the model category $\mathcal{C}$, not just in Ho($\mathcal{C}$).)

This produces a multiplication map 
$E_0^*I_{\bullet} \Smash E_0^*I_{\bullet}\rightarrow E_0^*I_{\bullet}$
that, together with the unit map  $\mathbbm{1}\rightarrow E_0^*I_{\bullet}$, satisfies the necessary commutativity, associativity, and unitality conditions to make $E_0^*I_{\bullet}$ a commutative monoid object in $\mathcal{C}$, by construction. 
\end{definition} 

\subsection{Tensoring and pretensoring over simplicial sets.}
We will write $f\Sets$ for the category of finite sets.
First we introduce the {\em pretensor product}, which is merely a convenient notation for the well-known ``Loday construction'' of~\cite{MR981743}:
\begin{definition}\label{def of tensoring}
We define a functor
\[ - \tilde{\otimes} -\co sf\Sets \times \Comm\mathcal{C}\rightarrow 
s\Comm\mathcal{C},\]
which we call the {\em pretensor product}, as follows.
If $X_{\bullet}$ is a simplicial finite set and
$A$ a commutative monoid in $\mathcal{C}$,
the simplicial commutative monoid $X_{\bullet}\tilde{\otimes} A$
is given by: 

For all $n\in\mathbb{N}$, the $n$-simplex object
$ (X_{\bullet} \tilde{\otimes} A)_n  =
 \coprod_{x\in X_n} A $
is a coproduct, {\em taken in $\Comm(\mathcal{C})$},
of copies of $A$, with one copy for each $n$-simplex $x\in X_n$. Recall that the categorical coproduct in $\Comm(\mathcal{C})$ is the smash product $\Smash$. 

For all positive $n\in\mathbb{N}$ and all $0\leq i\leq n$, the $i$-th face map
\[ d_i\co(X_{\bullet} \tilde{\otimes} A)_n \rightarrow 
(X_{\bullet} \tilde{\otimes} A)_{n-1}\]
is given on the component corresponding to an $n$-simplex $x\in X_n$ by the map
\[ A \rightarrow \coprod_{y\in X_{n-1}} A\]
which is inclusion of the coproduct factor corresponding to the $(n-1)$-simplex
$\delta_i(x)$.

For all positive $n\in\mathbb{N}$ and all $0\leq i\leq n$, the $i$-th degeneracy map
\[ s_i\co(X_{\bullet} \tilde{\otimes} A)_n \rightarrow 
(X_{\bullet} \tilde{\otimes} A)_{n+1}\]
is given on the component corresponding to an $n$-simplex $x\in X_n$ by the map
\[ A \rightarrow \coprod_{y\in X_{n+1}} A\]
which is inclusion of the coproduct factor corresponding to the $(n+1)$-simplex
$\sigma_i(x)$.

We have defined the pretensor product on objects; it is then 
defined on morphisms in the evident way.

We define the {\em tensor product}
\[ - \otimes -\co sf\Sets \times \Comm\mathcal{C}\rightarrow 
\Comm\mathcal{C}\]
as the geometric realization of the pretensor product:
\[ X_{\bullet}\otimes A = \left| X_{\bullet}\tilde{\otimes} A\right| .\]
\end{definition}
It is easy to check that $X_{\bullet}\tilde{\otimes} A$ is indeed a simplicial
object in $\Comm(\mathcal{C})$.

When $\mathcal{C}$ is the category of symmetric spectra,
the tensor product $X_{\bullet}\otimes A$ agrees with the tensoring
of commutative ring spectra over simplicial sets. (This is proven in~\cite{MR1473888}, although using (an early incarnation of) $S$-modules \cite{MR1417719}, rather than symmetric spectra; the symmetric monoidal Quillen equivalence of $S$-modules and symmetric spectra, as in~\cite{MR1819881}, then gives us the same result in symmetric spectra.)
The same is true when $E$ is a commutative $S$-algebra and
$\mathcal{C}$ is the category of $E$-modules.
In fact, the tensor product defined in Definition~\ref{def of tensoring} 
agrees
with the tensoring over simplicial sets in every case 
of a symmetric monoidal model category whose category of
commutative monoids is tensored over simplicial sets
that is known to the authors.

In particular, when $X_{\bullet}$ is the minimal simplicial model for the circle $S^1_{\bullet}$,
then $S^1_{\bullet}\tilde{\otimes} A$ is the cyclic bar construction on $A$,
and hence (by the main result of \cite{MR1473888}) 
$S^1_{\bullet}\otimes A$ agrees with the topological Hochschild homology
ring spectrum $THH(A,A)$.

For other simplicial sets, $X_{\bullet}\otimes A$ is regarded as a generalization
of topological Hochschild homology, eg 
as ``higher order Hochschild homology'' in \cite{MR1755114}. 

\begin{observation}\label{observation on cofibrancy conditions}
Suppose $I_{\bullet}$ is a cofibrant decreasingly filtered commutative monoid, in the sense of Definition~\ref{def of dec filt comm mon}, and suppose that $X_{\bullet}$ is a simplicial finite set. 
The assumptions made in Definition~\ref{def of dec filt comm mon}, particularly the ``Cofibrancy of degree 0 quotient'' assumption, together with Running Assumption~\ref{ra:2}, ensures that $X_{\bullet}\tilde{\otimes} E_0^*I_{\bullet}$ is Reedy-cofibrant as a simplicial object of $\mathcal{C}$ (but {\em not} as a simplicial object of $\Comm(\mathcal{C})$). The argument is as follows: each $I_{n+1}\rightarrow I_n$ is a cofibration in $\mathcal{C}$, so the pushout map $0\rightarrow  I_n/I_{n+1}$ is a cofibration, so each $I_n/I_{n+1}$ is cofibrant. So for each $m$, $X_{m}\tilde{\otimes} E_0^*I_{\bullet}$ is a coproduct of cofibrant objects, hence $X_{m}\tilde{\otimes} E_0^*I_{\bullet}$ is cofibrant for each $m$. So, if we know that each degeneracy map in $X_{\bullet}\tilde{\otimes} E_0^*I_{\bullet}$ is a level-wise cofibration (in the sense of Running Assumption \ref{ra:1}), then $X_{\bullet}\tilde{\otimes} E_0^*I_{\bullet}$ is Reedy-cofibrant in $\mathcal{C}^{\Delta^{\op}}$ by Running Assumption \ref{ra:2}. The degeneracy maps in $X_{\bullet}\tilde{\otimes} E_0^*I_{\bullet}$ are smash products of coproducts of copies of the map $0 \rightarrow I_n/I_{n+1}$ for $n>0$ and copies of the composite map $S \rightarrow I_0 \rightarrow I_0/I_1$, so the ``Cofibrancy of degree 0 quotient'' condition from Definition~\ref{def of dec filt comm mon} is exactly what is necessary to ensure that the degeneracy maps are level-wise cofibrations and hence that $X_{\bullet}\tilde{\otimes} E_0^*I_{\bullet}$ is Reedy-cofibrant in $\mathcal{C}^{\Delta^{\op}}$.
\end{observation}

\begin{remark} \label{remark on cofibrancy conditions}
Because of Observation~\ref{observation on cofibrancy conditions}, we claim that when $I_{\bullet}$ is a cofibrant decreasingly filtered commutative monoid, the spectral sequence of Definition~\ref{def of thh-may ss} that we will construct using the pretensor product has good homotopical properties, {\em even though} the assumptions in Definition~\ref{def of dec filt comm mon} are not enough to guarantee that $E_0^*I_{\bullet}$ is cofibrant in $\Comm(\mathcal{C})$.
For the sake of the spectral sequence of Definition~\ref{def of thh-may ss} (ie, the central motivating construction in this paper), it is enough to know that $X_{\bullet}\otimes E_0^*I_{\bullet}$ has the correct homotopy type, ie, that generalized homologies of $X_{\bullet}\otimes E_0^*I_{\bullet}$ are computable from those of $\pi_*\left( E_0^*I_{\bullet}\right)$ by the usual methods (eg the B\"{o}kstedt spectral sequence) one uses in order to compute $THH$ or its higher-order variants.

For that purpose, it is enough to know that $X_{\bullet}\tilde{\otimes} A$ is Reedy-cofibrant {\em as a simplicial object of $\mathcal{C}$}, not necessarily as a simplicial object of $\Comm(\mathcal{C})$. This is because Reedy-cofibrancy of $X_{\bullet}\tilde{\otimes} A$ as a simplicial object of $\mathcal{C}$ is enough to give us a Bousfield-Kan-type spectral sequence with $E_2$-term the homology of the alternating sign chain complex obtained by applying a generalized homology theory $E_*$ to $X_{\bullet}\tilde{\otimes} A$, and which converges to $E_*\left| X_{\bullet}\tilde{\otimes} A\right| = E_*\left( X_{\bullet}\otimes A\right)$ (the B\"{o}kstedt spectral sequence is the special case $E = H\mathbb{F}_p$); and it is enough to tell us that the geometric realization $\left| X_{\bullet}\tilde{\otimes} A\right| = X_{\bullet}\otimes A$ is a model for the homotopy colimit of the functor $\Delta^{\op}\rightarrow \mathcal{C}$ given by $X_{\bullet}\tilde{\otimes} A$. See Chapter~XII of~\cite{MR0365573} for a classical account of these ideas; we also provide some details and discussion in the present paper in Theorem~\ref{homological bkss existence thm}. 

More is true, however: suppose that $I_{\bullet}$ is a cofibrant decreasingly filtered commutative monoid, and suppose that we are not satisfied by $E_0^*I_{\bullet}$'s lack of cofibrancy in $\Comm(\mathcal{C})$. Let $cE_0^*I_{\bullet}$ be a cofibrant commutative monoid in $\mathcal{C}$ and let $cE_0^*I_{\bullet} \rightarrow E_0^*I_{\bullet}$ be a weak equivalence in $\Comm(\mathcal{C})$. Then \begin{equation}\label{map of simplicial things} X_{\bullet} \tilde{\otimes} cE_0^*I_{\bullet} \rightarrow X_{\bullet} \tilde{\otimes} E_0^*I_{\bullet}\end{equation} is a Reedy weak equivalence (in $\Comm(\mathcal{C})^{\Delta^{\op}}$ as well as in $\mathcal{C}^{\Delta^{\op}}$) whose domain and codomain are both Reedy-cofibrant in $\mathcal{C}^{\Delta^{\op}}$, so~\eqref{map of simplicial things} induces a weak equivalence (in $\mathcal{C}$) on geometric realizations, by the famous Theorem~D in Reedy's thesis~\cite{reedy}. So, as long as $I_{\bullet}$ is a cofibrant decreasingly filtered commutative monoid, the homotopy type of $X_{\bullet}\otimes E_0^*I_{\bullet}$ in $\mathcal{C}$ is not affected by the failure of $E_0^*I_{\bullet}$ to be cofibrant in $\Comm(\mathcal{C})$.
\end{remark}

\subsection{The fundamental theorem of the May filtration.}\label{fund thm section}
The fundamental theorem of the May filtration relies on the following lemma. 
\begin{lemma}\label{n=2 case of hard lemma}
Suppose $I,J$ are objects of $\mathcal{C}$ and
$f\co I^{\prime}\rightarrow I$ and $g\co J^{\prime}\rightarrow J$ are cofibrations.
Suppose $I,J,I^{\prime},J^{\prime}$ are cofibrant.
Let $P=I\Smash J^{\prime}\coprod_{I^{\prime}\Smash J^{\prime}} I^{\prime}\Smash J$ denote the pushout (which, by Running Assumption \ref{ra:1}, is a homotopy pushout).
Let 
$f\square g\co P\rightarrow I\Smash J$
denote the canonical map given by the universal property of the pushout, known as the pushout product. Then $f\square g$ is a cofibration by the pushout product axiom in the definition of a monoidal model category, as in ~\cite{MR1734325} and the cofiber of $f\square g$ is isomorphic to $(I/I^{\prime})\Smash (J/J^{\prime})$.
So the following sequence is a cofiber sequence:
\begin{equation}\label{useful cof seq} P \stackrel{f}{\longrightarrow} I\Smash J \rightarrow (I/I^{\prime})\Smash (J/J^{\prime}). \end{equation}
\end{lemma}
\begin{proof}
This lemma occurs as Lemma~4.7 in May~\cite{MR1867203} and its proof is easily generalized to a general model category $\mathcal{C}$ satisfying Running Assumption \ref{ra:1}. 
\end{proof}

We now define some categories and functors that will be important for Definition \ref{def of may filt 0 4}.  If $S$ is a finite set, we will equip the set, $\mathbb{N}^S$, of functions from $S$ to $\mathbb{N}$
with the $L^1$-norm, that is, $\left| x\right| = \sum_{s\in S}x(s)$, and 
with the strict direct product order, that is, 
$x \leq y$ in $\mathbb{N}^S$ if and only if $x(s)\leq y(s)$ for all $s\in S$. If $T \stackrel{f}{\longrightarrow} S$ is a function between finite sets,
let $\mathbb{N}^T \stackrel{\mathbb{N}^f}{\longrightarrow} \mathbb{N}^S$
be the function of partially-ordered sets defined by
\[ \left(\mathbb{N}^f(x)\right)(s) = \sum_{\{ t\in T\co f(t) = s\}} x(t) .\]

If $S$ is a finite set, for each $n\in\mathbb{N}$ we will let $\mathcal{D}^S_n$ be the 
sub-poset of $\mathbb{N}^S$ consisting of all functions $x\in\mathbb{N}^S$
such that $\left| x\right|  \geq n$. If $T \stackrel{f}{\longrightarrow} S$ is a function between finite sets,
let $\mathcal{D}^T_n \stackrel{\mathcal{D}^f_n}{\longrightarrow} \mathcal{D}^S_n$
be the function of partially-ordered sets defined by
restricting $\mathbb{N}^f$ to $\mathcal{D}^T_n$.

For each $x\in \mathbb{N}^S$ and each $n\in \mathbb{N}$, let $\mathcal{D}^S_{n; x}$ denote
the following sub-poset of $\mathbb{N}^S$:
\begin{equation}\label{sub-poset}  \mathcal{D}^S_{n; x} = \left\{ y \in \mathbb{N}^S\co y \geq x, \mbox{\ and\ } \left| y\right| \geq n + \left| x\right| \right\} .\end{equation}
So, for example, $\mathcal{D}^S_{n; \vec{0}} = \mathcal{D}^S_n$,
where $\vec{0}$ is the constant zero function.
If $T \stackrel{f}{\longrightarrow} S$ is a function between finite sets
and $x\in \mathbb{N}^T$ and $n\in\mathbb{N}$,
let $\mathcal{D}^T_{n;x} \stackrel{\mathcal{D}^f_{n;x}}{\longrightarrow} \mathcal{D}^S_{n;\mathcal{D}^f_n(x)}$
be the function of partially-ordered sets defined by
restricting $\mathbb{N}^f$ to $\mathcal{D}^T_{n;x}$.

Let $S$ be a finite set and let $n$ be a nonnegative integer.
We write $\mathcal{E}^S_{n,k}$ for the set 
\begin{equation}\label{def of may filt 0 3} \mathcal{E}^S_{n,k} = \left\{ x\in \{ 0,1, \dots , n\}^{S}\co \left| x\right| \geq k \right\}. \end{equation}

When $n=k$, we simply write $\mathcal{E}^S_{n}$ for this partially ordered set. The definition of $\mathcal{E}^S_n$ is natural in $S$ in the following sense:
if $T\stackrel{f}{\longrightarrow} S$ is a injective map of finite sets, then $\mathbb{N}^f$ naturally restricts to a function $\mathcal{E}^T_n\rightarrow \mathcal{E}^S_n$. 

\begin{definition} \label{def of may filt 0 4}
Suppose $I_{\bullet}$ is a cofibrant decreasingly filtered object in $\mathcal{C}$ and suppose
$S$ is a finite set. We will let 
$ \mathcal{F}^S(I_{\bullet})\co \left(\mathbb{N}^S\right)^{\op}\rightarrow \mathcal{C}$
be the functor sending $x$ to the smash product
$\Smash_{s\in S} I_{x(s)},$
and defined on morphisms in the apparent way, and we will let 
$ \mathcal{F}^S_n(I_{\bullet})\co \left(\mathcal{D}^S_n\right)^{\op} \hookrightarrow
\left(\mathbb{N}^S\right)^{\op}\stackrel{\mathcal{F}^S(I_{\bullet})}{\longrightarrow} 
\mathcal{C}$
be the functor which is the composite of $\mathcal{F}^S(I_{\bullet})$ with the inclusion
of $(\mathcal{D}^S_n)^{\op}$ into $(\mathbb{N}^S)^{\op}$ as a subcategory. 

If $x\in \mathcal{D}^S_n$, we will write $\mathcal{F}^S_{n; x}(I_{\bullet})$ for the restriction of the diagram
$\mathcal{F}^S(I_{\bullet})$ to $\mathcal{D}^S_{n; x}$, that is, $\mathcal{F}^S_{n; x}(I_{\bullet})$ is the composite
$\mathcal{F}^S_{n; x}(I_{\bullet})\co(\mathcal{D}^S_{n; x})^{\op} \hookrightarrow \left(\mathbb{N}^S\right)^{\op}\stackrel{\mathcal{F}^S(I_{\bullet})}{\longrightarrow} 
\mathcal{C}.$
Finally, let $\mathcal{M}^S_n(I_{\bullet})$ denote the colimit
$ \mathcal{M}^S_n(I_{\bullet}) = \colim \left(\mathcal{F}^S_n(I_{\bullet})\right)$
in $\mathcal{C}$. By the natural inclusion of
$\mathcal{D}^S_n$ into $\mathcal{D}^S_{n-1}$ as a subcategory, we now have a sequence in $\mathcal{C}$:
\begin{equation}\label{may filt diagram} \dots \rightarrow \mathcal{M}^S_3(I_{\bullet}) \rightarrow \mathcal{M}^S_2(I_{\bullet})
 \rightarrow \mathcal{M}^S_1(I_{\bullet}) \rightarrow \mathcal{M}^S_0(I_{\bullet}) \cong \Smash_{s\in S} I_0.\end{equation}
We refer to the functor $\mathbb{N}^{\op}\rightarrow \mathcal{C}$ given by
sending $n$ to $\mathcal{M}^S_n(I_{\bullet})$ as the {\em May filtration on $\Smash_{s\in S}I_0$}.

The May filtration is functorial in $S$ in the following sense:
if $T\stackrel{f}{\longrightarrow} S$ is a map of finite sets, 
we have a functor 
\begin{eqnarray*} 
\mathcal{D}^f_n\co \mathcal{D}^{T}_n & \rightarrow & \mathcal{D}^S_n \\
\left(\mathcal{D}^f_n(x)\right)(s) & \mapsto & \sum_{\left\{ t\in T\co f(t) = s\right\}} x(t)
\end{eqnarray*}
and a map of diagrams from
$\mathcal{F}^T_n(I_{\bullet})$ to $\mathcal{F}^S_n(I_{\bullet})$ given by 
sending the object 
$ \mathcal{F}^T_n(I_{\bullet})(x) = \Smash_{t\in T}I_{x(t)}$
to the object 
$ \mathcal{F}^S_n(I_{\bullet})(\mathcal{D}^f_n(x)) = \Smash_{s\in S} I_{\Sigma_{\{ t\in T\co f(t) = s\}} x(t)}$
by the map
\[ \Smash_{t\in T}I_{x(t)} \rightarrow \Smash_{s\in S} I_{\Sigma_{\{ t\in T\co f(t) = s\}} x(t)}\]
given as the smash product, across all $s\in S$, of the maps
\[ \Smash_{\{ t\in T\co f(t) = s\}} I_{x(t)} \rightarrow I_{\Sigma_{\{ t\in T\co f(t) = s\}} x(t)}\]
given by multiplication via the maps $\rho$ of Definition~\ref{def of dec filt comm mon}.
\end{definition}
To really make Definition~\ref{def of may filt 0 4} precise, we should say in which order we multiply the factors $I_{x(t)}$ using the maps
$\rho$; but the purpose of the associativity and commutativity axioms in Definition~\ref{def of dec filt comm mon} is that any two such choices
commute, so any choice of order of multiplication will do.

\begin{definition}\label{def of may filt}
{\bf (Definition of the May filtration.)}
If $I_{\bullet}$ is a cofibrant decreasingly filtered commutative monoid in $\mathcal{C}$ 
and $X_{\bullet}$ a simplicial finite set, by the {\em May filtration on $X_{\bullet}\tilde{\otimes} I_0$} we mean the 
functor $\mathcal{M}^{X_{\bullet}}(I_{\bullet})\co \mathbb{N}^{\op} \rightarrow \mathcal{C}^{\Delta^{\op}}$
given by sending a natural number $n$ to the 
simplicial object $\mathcal{M}^{X_{\bullet}}_n(I_{\bullet})$ of $\mathcal{C}$, 
with $\mathcal{M}^{X_i}_n(I_{\bullet})$ defined as in Definition~\ref{def of may filt 0 4},
and with face and degeneracy maps defined as follows: The face map
$ d_i\co \mathcal{M}^{X_j}_n(I_{\bullet}) \rightarrow \mathcal{M}^{X_{j-1}}_n(I_{\bullet})$
is the colimit of the map of diagrams 
$ \mathcal{F}^{X_j}_n(I_{\bullet}) \rightarrow \mathcal{F}^{X_{j-1}}_n(I_{\bullet})$
induced, as in Definition~\ref{def of may filt 0 4}, by $\delta_i\co X_j\rightarrow X_{j-1}$. The degeneracy map
$ s_i\co \mathcal{M}^{X_j}_n(I_{\bullet}) \rightarrow \mathcal{M}^{X_{j+1}}_n(I_{\bullet})$
is the colimit of the map of diagrams 
$ \mathcal{F}^{X_j}_n(I_{\bullet}) \rightarrow \mathcal{F}^{X_{j+1}}_n(I_{\bullet})$
induced, as in Definition~\ref{def of may filt 0 4}, by $\sigma_i\co X_j\rightarrow X_{j+1}$.
\end{definition}

\begin{remark} \label{rem on structure}
Note that the associative, commutative, and unital multiplications on the objects $I_i$, via the maps $\rho$ of
Definition~\ref{def of dec filt comm mon}, also yield (by taking smash products of the maps
$\rho$) associative, commutative, and unital multiplication natural transformations 
\begin{equation}\label{action maps on F} \mathcal{F}^S_m(I_{\bullet})\Smash \mathcal{F}^S_n(I_{\bullet})\rightarrow\mathcal{F}^S_{m+n}(I_{\bullet}),\end{equation} hence, on taking colimits,
associative, commutative, and unital multiplication maps
\[\mathcal{M}^S_m(I_{\bullet})\Smash \mathcal{M}^S_n(I_{\bullet})\rightarrow\mathcal{M}^S_{m+n}(I_{\bullet}),\] ie,
the functor $\mathbb{N}^{\op} \rightarrow \mathcal{C}$ sending $n$ to $\mathcal{M}^S_n(I_{\bullet})$
is a cofibrant decreasingly filtered commutative monoid, in the sense of
Definition~\ref{def of dec filt comm mon}.
Note furthermore that, if $f\co T\rightarrow S$ is a map of finite sets, then
the induced maps
$\mathcal{F}^T_m(I_{\bullet}) \rightarrow \mathcal{F}^S_m(I_{\bullet})$
commute with the multiplication maps~\eqref{action maps on F}, and so
$\mathcal{M}^T_{\bullet}(I_{\bullet})\rightarrow\mathcal{M}^S_{\bullet}(I_{\bullet})$ is a map of 
cofibrant decreasingly filtered commutative monoids.

Consequently, for any simplicial finite set $X_{\bullet}$, we have that
$\mathcal{M}_{\bullet}^{X_{\bullet}}(I_{\bullet})$ is a simplicial object in the category of cofibrant decreasingly
filtered commutative monoids in $\mathcal{C}$.
Since geometric realization commutes with the monoidal product in $\mathcal{C}$
by our running assumptions on $\mathcal{C}$,
this in turn implies that the geometric realization
$\left| \mathcal{M}_{\bullet}^{X_{\bullet}}(I_{\bullet})\right|$ of $\mathcal{M}_{\bullet}^{X_{\bullet}}(I_{\bullet})$ is
a cofibrant decreasingly filtered commutative monoid in $\mathcal{C}$ by Running Assumption \ref{ra:2}. It can easily be shown that $\mathcal{M}_n^{X_{\bullet}}(I_{\bullet})$ satisfies Running Assumption \ref{ra:2}'s Item \eqref{it2} for each $n$ since $I_0$ is assumed cofibrant as an object in $\mathcal{C}$. Running Assumption \ref{ra:2}'s Item \eqref{it1} and Item \eqref{it3} are satisfied by definition of $\mathcal{M}_n^{X_{\bullet}}(I_{\bullet})$ and by definition of the maps $ \mathcal{M}_n^{X_{\bullet}}(I_{\bullet})\rightarrow \mathcal{M}_{n-1}^{X_{\bullet}}(I_{\bullet}). $
Therefore, the 
commutative monoid $\left| \mathcal{M}_{\bullet}^{X_{\bullet}}(I_{\bullet})\right|$ in $\mathcal{C}^{\mathbb{N}^{\op}}$ is a cofibrant decreasingly filtered commutative monoid in $\mathcal{C}$. (We remind the reader that, by the main theorem of the authors' paper \cite{161106215}, an example of a setting in which Running Assumption \ref{ra:2} holds is the category of symmetric spectra equipped with the positive flat stable model structure.)
\end{remark}
\begin{definition}\label{Kan ext}
Suppose $n\in \mathbb{N}$.
We have the canonical inclusion of categories $\iota\co \mathcal{D}^S_{n+1} \hookrightarrow  \mathcal{D}^S_n$.
We will write $\tilde{\mathcal{F}}_{n+1}^S(I_{\bullet})$ for the left Kan extension of 
$\mathcal{F}_{n+1}^S(I_{\bullet})\co (\mathcal{D}^S_{n+1})^{\op}\rightarrow \mathcal{C}$ along $\iota^{\op}$, ie,
if we write 
$ Kan\co \mathcal{C}^{(\mathcal{D}_{n+1}^S)^{\op}} \rightarrow \mathcal{C}^{(\mathcal{D}_{n}^S)^{\op}}$
for the left adjoint of the functor
$\mathcal{C}^{(\mathcal{D}_{n}^S)^{\op}} \rightarrow \mathcal{C}^{(\mathcal{D}_{n+1}^S)^{\op}}$
induced by $\iota$, then 
$ \tilde{\mathcal{F}}_{n+1}^S(I_{\bullet}) = Kan(\mathcal{F}_{n+1}^S(I_{\bullet})).$
By the universal property of this Kan extension, we have a canonical map $c\co \tilde{\mathcal{F}}_{n+1}^S(I_{\bullet}) \rightarrow \mathcal{F}_n^S(I_{\bullet})$.
\end{definition}
\begin{remark}
It is a elementary exercise in combinatorics to show that there are $\binom{n+\#(S)}{n}$ elements in the set $\{ x\in \mathbb{N}^S\co \left| x\right| = n\}$, where $\#(S)$ is the cardinality 
of $S$. This set indexes a coproduct in the following lemma. 
\end{remark}
\begin{lemma}\label{cofiber computation lemma}
Let $I_{\bullet}$ be a cofibrant decreasingly filtered object in $\mathcal{C}$ and let 
$S$ be a finite set. Then the cofiber of the map
\[ \colim \left( \tilde{\mathcal{F}}_{n+1}^S(I_{\bullet})\right) \stackrel{\colim c}{\longrightarrow} \colim \left(\mathcal{F}_n^S(I_{\bullet})\right), \]
computed in $\mathcal{C}$, is isomorphic to the coproduct in $\mathcal{C}$
\[ \coprod_{\{ x\in \mathbb{N}^S\co \left| x\right| = n\}} \left( \left( \Smash_{s\in S} I_{x(s)}\right)/\left( \colim \mathcal{F}_{1; x}^S(I_{\bullet})\right) \right) \]
(see Definition \ref{Kan ext} for the definition of the map $c$). This isomorphism is natural in the variable $S$.
\end{lemma}
\begin{proof}
One knows that the left Kan extension of $\mathcal{F}_{n+1}^S(I_{\bullet})$ agrees with $\mathcal{F}_{n+1}^S(I_{\bullet})$ wherever both are defined: given $x\in \mathcal{D}_{n+1}^S$, the pointwise formula for a Kan extension gives us that 
$\tilde{\mathcal{F}}^S_{n+1}(I_{\bullet})(x)$ is the colimit, over all $y\rightarrow x$ in $\mathcal{D}_{n+1}^S$, of $\mathcal{F}^S_{n+1}(I_{\bullet})(y)$. Since the identity map on $x$ is already in $\mathcal{D}_{n+1}^S$, we get that $\tilde{\mathcal{F}}_{n+1}^S(I_{\bullet})(x) \cong \mathcal{F}_{n+1}^S(I_{\bullet})(x)$. Hence, 
\begin{eqnarray*} \tilde{\mathcal{F}}_{n+1}^S\left(I_{\bullet}\right)\left( x\right) & \cong & \mathcal{F}_{n+1}^S\left(I_{\bullet}\right)\left( x\right) \\
 & = & \Smash_{s\in S} I_{x(s)}\end{eqnarray*}
for all $x\in \mathcal{D}_{n+1}^S \subseteq \mathcal{D}_{n}^S$.
The elements of $\mathcal{D}_{n}^S$ which are {\em not} in $\mathcal{D}_{n+1}^S$ are those
$x$ such that $\left| x\right| = n$, and by the pointwise formula for a Kan extension (see eg~\cite{MR1712872}) 
one knows that, for all $x$ such that $\left| x\right| = n$, we have an isomorphism of
$\tilde{\mathcal{F}}_{n+1}^S(I_{\bullet})\left( x\right)$ with the colimit of 
the values of $\mathcal{F}_{n+1}^S(I_{\bullet})$ over those elements of $(\mathcal{D}_{n+1}^S)^{\op}$ which map to 
$x$, ie, the colimit of the values of $\mathcal{F}_{n+1}^S(I_{\bullet})$ over 
$(\mathcal{D}_{1, x}^S)^{\op}\subseteq (\mathcal{D}_{n+1}^S)^{\op}$, ie,
$\colim \left(\mathcal{F}_{1, x}^S(I_{\bullet})\right)$.

For each $x\in \mathcal{D}^S_n$, the map $c(x)$ can be shown to be a cofibration by iterated use of the pushout product axiom, so the cofiber of $c(x)$ is a homotopy cofiber. By the previous paragraph the levelwise cofiber $\cof c\co  (\mathcal{D}_{n}^S)^{\op} \rightarrow\mathcal{C}$ of the natural transformation $c$ is given as follows:
\[ (\cof c)\left( x \right) \cong \left\{ \begin{array}{lll} 0 & \mbox{\ if\ } & \left| x\right| > n \\
 \left(\mathcal{F}_n^S(I_{\bullet})\right)/\left(\colim \left(\mathcal{F}_{1, x}^S(I_{\bullet})\right)\right) & \mbox{\ if\ } &
    \left| x\right| = n.\end{array}\right. \]
Hence, on taking colimits, we have
\begin{eqnarray*} \cof \colim c & \cong & \colim \cof c \\
 & = &  \coprod_{\{ x\in\mathbb{N}^S\co \left| x\right| = n\}} \left( \left( \Smash_{s\in S} I_{x(s)}\right)/\left( \colim \mathcal{F}_{1; x}^S(I_{\bullet})\right) \right) ,\end{eqnarray*} 
as claimed.
\end{proof}

\begin{lemma} \label{rather hard lemma} 
Suppose $S$ is a finite set
and suppose $Z_{s,1} \rightarrow Z_{s,0}$ is a cofibration for each $s\in S$.
Suppose the objects $Z_{s,1},Z_{s,0}$ are all cofibrant. 
Let $\mathcal{G}_S\co (\mathcal{E}^S_1)^{\op}\rightarrow \mathcal{C}$ be the functor given on objects by
$ \mathcal{G}_S(x) = \Smash_{s\in S} Z_{s,x(s)},$
and given on morphisms in the obvious way.

Then the smash product 
$ \Smash_{s\in S} Z_{s,0} \rightarrow \Smash_{s\in S} \left( Z_{s,0}/Z_{s,1}\right) $
of the cofiber projections $Z_{s,0} \rightarrow Z_{s,0}/Z_{s,1}$ fits into a cofiber sequence:
\[ \colim \mathcal{G}_S \rightarrow \Smash_{s\in S} Z_{s,0} \rightarrow \Smash_{s\in S} \left( Z_{s,0}/Z_{s,1}\right).\]
\end{lemma}
\begin{proof}
If the cardinality of $S$ is one, the statement of the lemma is true by the definition of a cofiber.

The case in which the cardinality of $S$ is two is precisely Lemma~\ref{n=2 case of hard lemma}, already proven.

For the case in which the cardinality of $S$ is greater than two, we introduce a notation we will need to use: let $\mathcal{PO}$ denote the category indexing pushout diagrams, ie,
\begin{equation}\label{def of PO} \mathcal{PO} = \left( \vcenter{\xymatrix{ & [1^{\prime}] \ar[ld]\ar[rd] & \\ [1] & & [0] }}\right) ,\end{equation}
the symbols $[1^{\prime}],[1],[0]$ each representing an object, and the arrows each representing a morphism.
We observe that $\mathcal{PO}$ and $\mathcal{E}^S_1$ are not arbitrary small categories but are
in fact partially-ordered sets; this simplifies some of the arguments we will give
in the rest of the proof.

Suppose the cardinality of $S$ is greater than two. Choose an element $s_0\in S$. We will write $S^{\prime}$ for the 
complement, 
$ S^{\prime} = \{ s\in S\co s\neq s_0\}$,
of $s_0$ in $S$.
Define objects $X_1^{\prime},X_2^{\prime},Y_1^{\prime},Y_2^{\prime}$ in $\mathcal{C}$ as follows: 
$Y^{\prime}_1= \colim \mathcal{G}_{S^{\prime}}$, $X^{\prime}_1=\Smash_{s\in S^{\prime}} Z_{s,0}$, $Y^{\prime}_2=Z_{s_0,1}$ and $X^{\prime}_2=Z_{s_0,0}$.
Now we apply the statement of the lemma, in the (already proven, above) case $S = \{ 1,2\}$ and using
$X_1^{\prime},X_2^{\prime},Y_1^{\prime},Y_2^{\prime}$ in place of $Z_{1,0},Z_{2,0},Z_{1,1},Z_{2,1}$ to obtain a cofiber sequence
\begin{equation}\label{preliminary cof seq} \colim \mathcal{B} \rightarrow \Smash_{s\in S} Z_{s,0} \rightarrow \Smash_{s\in S} \left( Z_{s,0}/Z_{s,1}\right) ,\end{equation}
where $\mathcal{B}$ is the functor $\mathcal{PO}\rightarrow \mathcal{C}$ given by: 
\begin{eqnarray*}
\mathcal{B}\left( [1^{\prime}] \right) & = & \left( \colim \mathcal{G}_{S^{\prime}}\right) \Smash Z_{s_0,1} \\
\mathcal{B}\left( [1] \right) & = & \left( \Smash_{s\in S^{\prime}} Z_{s_0,0} \right) \Smash Z_{s_0,1} \\
\mathcal{B}\left( [0] \right) & = & \left( \colim \mathcal{G}_{S^{\prime}} \right) \Smash Z_{s_0,0} .
\end{eqnarray*}

By Lemma~\ref{n=2 case of hard lemma}, we know that the map $\colim \mathcal{G}_{S^{\prime}}\longrightarrow \Smash_{s\in S'} Z_{s,0}$ is a cofibration in the case $S=\{1,2\}$. Since $\colim \mathcal{B}$ is constructed as a pushout, the pushout product axiom ensures that the map $\colim \mathcal{B}\longrightarrow \Smash_{s\in S} Z_{s,0}$ is also a cofibration as long as $\colim \mathcal{G}_{S^{\prime}}\longrightarrow \Smash_{s\in S'} Z_{s,0}$ is a cofibration. It suffices to  show that $\colim\mathcal{B}\cong \colim \mathcal{G}_S$. This will show that the map $\colim \mathcal{G}_{S}\longrightarrow \Smash_{s\in S}Z_{s,0}$ is a cofibration and allow us to identify the cofiber, thus completing the induction on the cardinality of the set $S$. 
We can reindex, to describe $\colim\mathcal{B}$ as the colimit of a larger diagram $\mathcal{H}$:
\begin{eqnarray*} \mathcal{H}\co (\mathcal{E}_1^{S^{\prime}})^{\op}\times \mathcal{PO} & \rightarrow & \mathcal{C} \\
 \left( x,[1^{\prime}] \right) & \mapsto & \left( \Smash_{s\in S^{\prime}} Z_{s,x(s)}\right) \Smash Z_{s_0,1} \\
 \left( x,[1] \right) & \mapsto & \left( \Smash_{s\in S^{\prime}} Z_{s,0}\right) \Smash Z_{s_0,1} \\
 \left( x,[0] \right) & \mapsto & \left( \Smash_{s\in S^{\prime}} Z_{s,x(s)}\right) \Smash Z_{s_0,0} \end{eqnarray*}

We have a functor 
\begin{eqnarray*} P\co (\mathcal{E}_1^{S^{\prime}})^{\op}\times \mathcal{PO} & \rightarrow & (\mathcal{E}_1^S)^{\op} \\
 \left( x,[1^{\prime}] \right)(s) & = & \left\{ \begin{array}{lll} x(s) & \mbox{\ if\ } & s\neq s_0 \\ 1 & \mbox{\ if\ } & s = s_0 \end{array}\right. \\
 \left( x,[1] \right)(s) & = & \left\{ \begin{array}{lll} 0 & \mbox{\ if\ } & s\neq s_0 \\ 1 & \mbox{\ if\ } & s = s_0 \end{array}\right. \\
 \left( x,[0] \right)(s) & = & \left\{ \begin{array}{lll} x(s) & \mbox{\ if\ } & s\neq s_0 \\ 0 & \mbox{\ if\ } & s = s_0 \end{array}\right. .\end{eqnarray*}

Now we claim that the canonical map $\colim\mathcal{H}\rightarrow\colim\mathcal{G}_S$ given by $P$ is an isomorphism in $\mathcal{C}$. We define a functor 
\begin{eqnarray*} 
I\co (\mathcal{E}_1^S)^{\op} & \rightarrow & (\mathcal{E}_1^{S^{\prime}})^{\op}\times \mathcal{PO} \\
x & \mapsto & \left\{ \begin{array}{lll}
   \left(I^{\prime}(x),[1^{\prime}]\right) & \mbox{\ if\ } & x(s_0)= 1 \\
   \left(I^{\prime}(x),[0]\right) & \mbox{\ if\ } & x(s_0)= 0 \end{array}\right. \end{eqnarray*}
where $I^{\prime}\co (\mathcal{E}_1^S)^{\op}\rightarrow (\mathcal{E}_1^{S^{\prime}})^{\op}$ is the functor given by restriction, ie,
$I^{\prime}(x)(s) = x(s)$ for $s\in S'$.
Now we observe some convenient identities:
\begin{eqnarray*}
 (\mathcal{G}_S\circ P)(x,j) & = & \left\{ \begin{array}{lll} 
    \left( \Smash_{s\in S^{\prime}} Z_{s,x(s)}\right) \Smash Z_{s_0,1} & \mbox{\ if\ } & j = [1^{\prime}] \\
    \left( \Smash_{s\in S^{\prime}} Z_{s,x(s)}\right) \Smash Z_{s_0,0} & \mbox{\ if\ } & j = [0] \\
    \left( \Smash_{s\in S^{\prime}} Z_{s,0}\right) \Smash Z_{s_0,1} & \mbox{\ if\ } & j = [1] \end{array}\right. \\
  & = & \mathcal{H}(x,j), \mbox{\ and} \\
 (\mathcal{H}\circ I)(x) & = & \left\{ \begin{array}{lll}
    \left( \Smash_{s^{\prime}} Z_{s,x(s)}\right)  \Smash Z_{s_0,1} & \mbox{\ if\ } & x_{s_0} = 1 \\
    \left( \Smash_{s^{\prime}} Z_{s,x(s)}\right)  \Smash Z_{s_0,0} & \mbox{\ if\ } & x_{s_0} = 0\end{array}\right. \\
  & = & \mathcal{G}_S(x).\end{eqnarray*}
We conclude that $P,I$ give mutually inverse maps between $\colim \mathcal{G}_S$
and $\colim \mathcal{H}$, ie, 
$\colim \mathcal{G}_S\cong \colim\mathcal{H}$ and hence
$\colim \mathcal{G}_S\cong \colim\mathcal{B}$, as desired. So 
from cofiber sequence~\eqref{preliminary cof seq}, we have
a cofiber sequence
\[ \colim \mathcal{G}_S \rightarrow \Smash_{s\in S} Z_{s,0} \rightarrow \Smash_{s\in S} \left( Z_{s,0}/Z_{s,1}\right), \]
as desired.

From inspection of the colimit diagrams one sees that the cofiber sequence~\eqref{preliminary cof seq} does not depend on the choice of $s_0\in S$, and naturality in
$S$ follows.
\end{proof}

\begin{lemma}\label{cofinality lemma}
Let $S$ be a finite set, let $n$ be a positive integer, and let $x\in \mathbb{N}^S$. Let $\mathcal{E}_n^S$ and 
$\mathcal{D}_{n; x}^S$ be the partially-ordered sets defined in equations~\eqref{sub-poset} and \eqref{def of may filt 0 3}. Let $J_{n;x}^S$ be the functor (ie, morphism of partially-ordered sets) defined by
\begin{eqnarray*} J_{n;x}^S\co \mathcal{E}^S_n & \rightarrow & \mathcal{D}_{n; x}^S \\
(J_{n;x}(y))(s) & = & x(s) + y(s).\end{eqnarray*}
Then $J_{n;x}$ has a right adjoint. Consequently $J_{n;x}$ is a cofinal functor, ie, for any functor $F$ defined on $\mathcal{D}_{n; x}^S$ 
such that the limit $\lim F$ exists, the limit $\lim (F\circ J_{n;x}^S)$ also exists, and the canonical map
$\lim (F\circ J_{n;x}^S) \rightarrow \lim F$ is an isomorphism.
\end{lemma}
\begin{proof}
We construct the right adjoint explicitly. Let $K_{n;x}^S$ be the functor defined by
\begin{eqnarray*} K_{n;x}^S\co \mathcal{D}_{n; x}^S & \rightarrow & \mathcal{E}^S_n \\
 (K_{n;x}^S(y))(s) & = & \min\{ n, y(s) - x(s)\}. \end{eqnarray*}
(We remind the reader that every element $y\in \mathcal{D}_{n; x}^S$ has the property that $y(s) \geq x(s)$ for all $s\in S$, so
$y(s) - x(s)$ will always be nonnegative.)

Now suppose $z\in \mathcal{E}_n^S$ and $y\in \mathcal{D}_{n; x}^S$. 
Then: 
$z\leq K_{n;x}^S(y)$ if and only if $z(s) \leq K_{n;x}^S(y)(s)$ for all $s\in S$, 
 ie, $z\leq K_{n;x}^S(y)$ if and only if $z(s)\leq \min\{ n, y(s) - x(s)\}$ for all $s\in S$.
 By the definition of $\mathcal{E}_n^S$, $z(s)\leq n$ for all $s\in S$. Hence $z\leq K_{n;x}^S(y)$ if and only if $z(s) \leq y(s) - x(s)$ for all $s\in S$,
 ie, $z\leq K_{n;x}^S(y)$ if and only if $x(s) + z(s) \leq y(s)$ for all $s\in S$, 
 ie, $z\leq K_{n;x}^S(y)$ if and only if $J_{n;x}^S(z) \leq y$.
Hence $\hom_{\mathcal{E}_n^S}(z, K_{n;x}^S(y))$ is nonempty if and only if $\hom_{\mathcal{D}_{n;x}^S}(J_{n;x}^S(z),y)$ is nonempty.
Since $\mathcal{E}_n^S$ and $\mathcal{D}_{n;x}^S$ are partially-ordered sets and hence their hom-sets are either nonempty or have only a single element,
we now have a (natural) bijection
\[ \hom_{\mathcal{E}_n^S}(z, K_{n;x}^S(y)) \cong \hom_{\mathcal{D}_{n;x}^S}(J_{n;x}^S(z),y) \]
which is exactly what we are looking for: $J_{n;x}^S$ is left adjoint to $K_{n;x}^S$.

For the fact that having a right adjoint implies cofinality, see section IX.3 of Mac Lane's \cite{MR1712872}. (Mac Lane handles the equivalent dual case.)
\end{proof}

\begin{theorem}{\bf (Fundamental theorem of the May filtration.)} \label{thm on fund theorem}
Let $I_{\bullet}$ be a cofibrant decreasingly filtered commutative monoid in $\mathcal{C}$,
and let $X_{\bullet}$ be a simplicial finite set. 
Then the associated graded commutative monoid 
$E_0^*\left| \mathcal{M}^{X_{\bullet}}(I_{\bullet})\right|$
of the geometric realization of the May filtration
is weakly equivalent, as a commutative graded monoid, to the tensoring
$X_{\bullet}\otimes E_0^* I_{\bullet}$ 
of $X_{\bullet}$ with the associated graded commutative monoid of $I_{\bullet}$:
\[ E_0^*\left| \mathcal{M}^{X_{\bullet}}(I_{\bullet})\right| \simeq
 X_{\bullet}\otimes E_0^* I_{\bullet}.\]
\end{theorem}
\begin{proof}
We must compute the filtration quotients
\begin{equation}\label{iso 2309} \left| \mathcal{M}^{X_{\bullet}}_n(I_{\bullet})\right|/\left| \mathcal{M}^{X_{\bullet}}_{n+1}(I_{\bullet})\right|
\cong
\left| \mathcal{M}^{X_{\bullet}}_n(I_{\bullet})/\mathcal{M}^{X_{\bullet}}_{n+1}(I_{\bullet})\right| .
\end{equation}
(We have isomorphism~\eqref{iso 2309} because $\left|\mathcal{M}^{X_{\bullet}}_{n+1}(I_{\bullet})\right| \rightarrow \left|\mathcal{M}^{X_{\bullet}}_n(I_{\bullet})\right|$ is a cofibration, by Remark~\ref{rem on structure}.)

We compute the filtration quotients as follows. First, we claim that there exists, for any finite set $S$ and for all $n\in \mathbb{N}$, a cofiber sequence
\begin{equation}\label{putative cof seq 1} \colim \left(\mathcal{F}^S_{n+1}(I_{\bullet})\right)\rightarrow \colim \left(\mathcal{F}^S_{n}(I_{\bullet})\right) \rightarrow \coprod_{x\in \mathbb{N}^S\co \left| x\right| = n} 
\left( \Smash_{s\in S}\left( I_{x(s)}/I_{1+x(s)} \right)\right) \end{equation}
in $\mathcal{C}$, natural in $S$. We have already defined (in Definition~\ref{def of may filt 0 4})
how $\mathcal{F}^S_n$ is natural, ie, functorial in $S$; by taking the obvious coproduct
of quotients, this naturality in $S$ induces a naturality in $S$ on the terms 
$\coprod_{x\in \mathbb{N}^S\co \left| x\right| = n} 
\left( \Smash_{s\in S}\left( I_{x(s)}/I_{1+x(s)} \right)\right)$ appearing in~\eqref{putative cof seq 1}. The claim that ~\eqref{putative cof seq 1} is a cofiber sequence implies that 
\begin{equation} \label{important cofiber} 
|\mathcal{M}_n^{X_k}(I_{\bullet})| / |\mathcal{M}_{n+1}^{X_k}(I_{\bullet})|\cong \coprod_{x\in \mathbb{N}^{X_k}\co \left| x\right| = n} 
\left( \Smash_{s\in X_k}\left( I_{x(s)}/I_{1+x(s)} \right)\right),
\end{equation} 
and naturally implies the necessary naturality with respect to the face and degeneracy maps. 

We now show that the cofiber sequence~\eqref{putative cof seq 1} exists.
First, by the universal property of the Kan extension 
from Lemma~\ref{cofiber computation lemma}, 
the cofiber of the map
$\colim  (\mathcal{F}^S_{n+1}(I_{\bullet}) )\rightarrow \colim  (\mathcal{F}^S_{n}(I_{\bullet}) )$
agrees with the cofiber of the map
$\colim (\tilde{\mathcal{F}}^S_{n+1}(I_{\bullet}))\rightarrow \colim (\mathcal{F}^S_{n}(I_{\bullet}) )$.
By Lemma~\ref{cofiber computation lemma}, this cofiber is the coproduct
\[ \coprod_{\{ x\in \mathbb{N}^S\co \left| x \right| = n\}} \left( \left( \Smash_{s\in S} I_{x(s)}\right)/\left( \colim \mathcal{F}_{1; x}^S(I_{\bullet})\right) \right) .\]
In Lemma~\ref{cofinality lemma}, we showed that the functor $J_{1;x}$ is cofinal, hence that the comparison map of colimits
\begin{equation*}\label{consequence of cofinality lemma} \colim \left( \mathcal{F}_{1; x}(I_{\bullet})\circ J_{1;x}\right) \rightarrow \colim \left(\mathcal{F}_{1; x}(I_{\bullet})\right)\end{equation*}
is an isomorphism. (We here have a colimit, not a limit as in the statement of Lemma~\ref{cofinality lemma}, since $\mathcal{F}_{1; x}(I_{\bullet})$ is a
{\em contravariant} functor on $\mathcal{D}_{1; x}^S$. Of course Lemma~\ref{cofinality lemma} still holds in this dual form.)

Now Lemma~\ref{rather hard lemma} identifies the cofiber
\[ \left(\Smash_{s\in S} I_{x(s)}\right)/\left( \colim \left(\mathcal{F}_{1; x}(I_{\bullet})\circ J_{1;x}\right) \right) \]
with $\Smash_{s\in S}\left( I_{x(s)}/I_{1+x(s)} \right)$, as desired. 
So we have our cofiber sequence of the form~\eqref{putative cof seq 1}.

All isomorphisms in the lemmas we have invoked in this proof are natural in $S$,
with the exception of the isomorphisms from Lemma~\ref{cofinality lemma}
and Lemma~\ref{rather hard lemma} which directly involve $\mathcal{E}_1^S$,
only because we did not specify in Lemma~\ref{rather hard lemma} how $\mathcal{G}_S$ is functorial in $S$. In the present proof, $\mathcal{G}_S$ is $\mathcal{F}_{1; x}(I_{\bullet})\circ J_{1;x}$,
and
the cofinality of $J_{1;x}$ together with the fact that $K_{1;x} \circ J_{1;x} = \id_{\mathcal{E}_1^S}$
implies, on inspection of the colimit diagrams, that the isomorphism
\begin{eqnarray*} 
\colim \mathcal{G}_S & = & \colim \left( \mathcal{F}_{1; x}(I_{\bullet})\circ J_{x}\right) \\
 & \cong & \colim \left(\mathcal{F}_{1; x}(I_{\bullet})\right)\end{eqnarray*}
is natural in $S$; details are routine and left to the reader.
We conclude that the cofiber sequence~\eqref{putative cof seq 1} is indeed natural in $S$.

Now we have the sequence of simplicial commutative monoids in $\mathcal{C}$:
\[\xymatrix{  \mathcal{M}^{X_{\bullet}}_0(I_{\bullet}) & \ar[l] \mathcal{M}^{X_{\bullet}}_1(I_{\bullet}) & \ar[l]  \mathcal{M}^{X_{\bullet}}_2(I_{\bullet}) & \ar[l] \dots }\]
and geometric realization commuting with cofibers together with the isomorphism \eqref{important cofiber} implies that the comparison map
\begin{equation}\label{comparison map 1}  X_{\bullet}\otimes E_0^* I_{\bullet}\rightarrow E_0^*\left| \mathcal{M}^{X_{\bullet}}(I_{\bullet})\right| \end{equation}
of objects in $\mathcal{C}$ is a weak equivalence. Hence the comparison map~\eqref{comparison map 1} in $\Comm(\mathcal{C})$ must also be a weak equivalence, since
the weak equivalences in $\Comm(\mathcal{C})$ are created by the forgetful functor
$\Comm(\mathcal{C})\rightarrow\mathcal{C}$, by Running Assumption~\ref{ra:1}.
\end{proof}
\subsection{Construction of the topological Hochschild-May spectral sequence.}\label{construction of ss}
\begin{definition}\label{def of gen hom thy}
By a {\em connective generalized homology theory on $\mathcal{C}$} we shall mean the choice, for each integer $n$, of a functor $E_n\co \Ho(\mathcal{C}) \rightarrow \Ab$
satisfying the axioms:
\begin{description}
\item[Exactness] 
For each integer $n$ and each distinguished triangle 
$X \rightarrow Y \rightarrow Z \rightarrow \Sigma X $
in $\Ho(\mathcal{C})$,
the following sequence of abelian groups is exact:
\[  \xymatrix{\dots \ar[r] &  E_{n}(Y) \ar[r]^{} & E_{n}(Z) \ar[r]^{} & E_n(\Sigma X) \ar[r]^{} &   E_n(\Sigma Y)  \ar[r]^{} & E_n(\Sigma Z) \ar[r] & \dots }\]
\item[Additivity]
For each integer $n$ and each collection of objects $\{ X_i\}_{i\in I}$ in $\Ho(\mathcal{C})$,
the following canonical map of abelian groups is an isomorphism:
\[ \coprod_{i\in I} E_n(X_i) \rightarrow E_n(\coprod_{i\in I} X_i).\]


\item[Connectivity of the unit object]  We have $E_n(\mathbbm{1}) \cong 0$ for all $n<0$.
\item[Connectivity of smash products] Suppose that $X,Y$ are objects of $\mathcal{C}$, and that $A,B$ are nonnegative integers such that
$E_n(X) \cong 0$ for all $n<A$, and $E_n(Y) \cong 0$ for all $n<B$.
Then $E_n(X\Smash Y) \cong 0$ for all $n<A+B$.
\end{description}
\end{definition}
Clearly, Definition~\ref{def of gen hom thy} is just a formulation, in a general pointed model category, of (the triangulated category form of) the Eilenberg-Steenrod axioms, from \cite{MR0012228}, for a generalized homology theory with connective 
coefficients. The ``connectivity of smash products'' axiom is easily proven anytime one has an $E$-homology K\"{u}nneth spectral sequence in $\mathcal{C}$, which is the case in 
any of the usual models for the stable homotopy category.

\begin{definition}\label{def of thh-may ss}
If $I_{\bullet}$ is a cofibrant decreasingly filtered commutative monoid in $\mathcal{C}$, 
$X_{\bullet}$ is a simplicial finite set, and $E_*$ is a connective generalized homology theory on $\mathcal{C}$, 
then by the {\em topological Hochschild-May spectral sequence for $X_{\bullet}\tilde{\otimes} I_{\bullet}$}
we mean the spectral sequence obtained by applying $E_*$ to the tower of cofiber sequences in $\mathcal{C}$
\begin{equation}\label{tower of sequences} \xymatrix{ 
\dots  \ar[r] & \left| \mathcal{M}_1^{X_{\bullet}}(I_{\bullet})\right| \ar[r] \ar[d] & \left| \mathcal{M}_0^{X_{\bullet}}(I_{\bullet})\right|  \ar[d]  \\
  &\left| \mathcal{M}_1^{X^{\bullet}}(I_{\bullet})\right|/\left| \mathcal{M}_2^{X^{\bullet}}(I_{\bullet})\right| &\left| \mathcal{M}_0^{X^{\bullet}}(I_{\bullet})\right|/\left| \mathcal{M}_1^{X^{\bullet}}(I_{\bullet})\right|. 
}
\end{equation}

That is, it is the spectral sequence of the exact couple 
\[\xymatrix{  
D^1_{*,*}  \ar[rr] && D^1_{*,*} \ar[dl]\\
& E^1_{*,*} \ar[ul] & }\]
where $E^1_{*,*}= \bigoplus_{i,j} H_i \left| \mathcal{M}_j^{X_{\bullet}}(I_{\bullet})\right|/\left| \mathcal{M}_{j+1}^{X_{\bullet}}(I_{\bullet})\right|$ and $D^1_{*,*}=\bigoplus_{i,j} H_i \left| \mathcal{M}_j^{X_{\bullet}}(I_{\bullet})\right|$. 
\end{definition}
 
\begin{lemma}{\bf (Connectivity conditions.)}\label{connectivity conditions lemma}
Let $E_*$ be a connective generalized homology theory on $\mathcal{C}$. Suppose that there exist objects $Z,E$ of $\mathcal{C}$ such that
$E_*(-)$ is naturally isomorphic
to $[\Sigma^* Z, E \Smash -]$. Let
 \begin{equation}\label{sequence 13083} \dots \rightarrow Y_2 \rightarrow Y_1 \rightarrow Y_0\end{equation}
 be a sequence
in $\mathcal{C}$, and suppose that $E_n(Y_i) \cong 0$ for all $n<i$.
Then 
\[ [\Sigma^n Z, \holim_i (E \Smash Y_i)] \cong 0\] 
for all $n$. If we instead suppose that $A$ is a nonnegative integer and that $X_{\bullet}$ is a simplicial object of $\mathcal{C}$ such that $E_n(X_i) \cong 0$ for all $n<A$ and all $i$,
then $E_n\left(\left| X_{\bullet}\right|\right) \cong 0$ for all $n<A$.
\end{lemma}
\begin{proof} 
Since $\mathcal{C}$ is assumed stable, the homotopy limit $\holim_i Y_i$ is the homotopy fiber of the map
\[ \prod_{n\in \mathbb{N}} Y_n \stackrel{\id - T}{\longrightarrow} \prod_{n\in \mathbb{N}} Y_n \]
in $\Ho(\mathcal{C})$, where $T$ is the product of the maps $Y_n \rightarrow Y_{n-1}$ occuring in the sequence~\eqref{sequence 13083}.
For each object $Z$ of $\mathcal{C}$, we then have the long exact sequence obtained by applying the functor $[\Sigma^* Z, E \Smash -]$ to the fiber sequence 
\[ \holim_i Y_i \rightarrow \prod_{n\in \mathbb{N}} Y_n \stackrel{\id - T}{\longrightarrow} \prod_{n\in \mathbb{N}} Y_n. \]
hence the Milnor exact sequence
\[ 0 \rightarrow R^1\lim_i [\Sigma^{j+1} Z, E\Smash Y_i] \rightarrow [\Sigma^j Z, \holim_i E\Smash Y_i] \rightarrow \lim_i [\Sigma^j Z, E\Smash Y_i] \rightarrow 0 .\]
The assumption that $[\Sigma^j Z, E\Smash Y_i] \cong 0$ for $j<i$ guarantees that the sequence
\[ \dots \rightarrow  [\Sigma^j Z, E\Smash Y_2] \rightarrow [\Sigma^j Z, E\Smash Y_1] \rightarrow [\Sigma^j Z, E\Smash Y_0] \]
is eventually constant and zero for all $j$, hence both its limit and $R^1\lim$ vanish for all $j$, hence $[\Sigma^j Z, \holim_i E\Smash Y_i]\cong 0$ for all $j$.

The Bousfield-Kan spectral sequence, ie, the $E$-homology spectral sequence of the simplicial object $X_{\bullet}$,
has input $E^1_{s,t} \cong \pi_s(E\Smash X_t)$ and converges to $E_{s+t}\left(\left| X_{\bullet}\right|\right)$, since 
$E_n(X_i) \cong 0$ for all $n<A$ and all $i$. 
The differential in this spectral sequence is of the form $d^r\co E^r_{s,t} \rightarrow E^r_{s-r,t+r-1}$, hence this spectral sequences has a nondecreasing upper vanishing curve at $E^1$, hence converges strongly.
Triviality of $E^1_{s,t}$ for $s<A$ and $t<0$ then gives us that $E_s\left(\left| X_{\bullet}\right|\right)$ vanishes for $s<A$.
\end{proof}

\begin{definition}\label{Reedy cube}
If $S$ is a finite set, recall that $\mathcal{D}^S_{i}$ is the partially-ordered set of functions $f: S \rightarrow \mathbb{N}$ such that $\left|f\right| \geq i$. The category $(\mathcal{D}^S_{i})^{\op}$ has a cofinal subcategory $(\mathcal{E}^S_i)^{\op}$ defined in \eqref{def of may filt 0 3}. We give this cofinal subcategory a Reedy category structure by letting the degree function be $-\left|s\right|$, letting $((\mathcal{E}^S_i)^{\op})_{+}$ contain all morphisms and $((\mathcal{E}^S_i)^{\op})_{-}$ contain only identity morphisms.
\end{definition}
\begin{remark}\label{projective coincides with reedy}
Since the category $(\mathcal{E}^S_i)^{\op}$ is a direct category by Definition \ref{Reedy cube}, the Reedy model structure on the category of functors of the form $(\mathcal{E}^S_i)^{\op}\rightarrow \mathcal{C}$, where $\mathcal{C}$ is a model category, agrees with the projective model structure. 
\end{remark}
 Since $\left|x\right|<i\cdot(\#S)$, the target of the degree map on $(\mathcal{E}^S_i)^{\op}$ is a (finite) ordinal, and this satisfies the requirement for a Reedy category. 

\begin{lemma}\label{connectivity of truncated cube colims}
Let $S$ be a finite set, and for each integer $i$, let $\mathcal{D}^S_{i}$ and $(\mathcal{E}^S_i)^{\op}$ be the categories defined immediately preceding Definition~\eqref{def of may filt 0 4}. For each $s\in S$ and each $n\in\mathbb{N}$, let $\mathcal{J}_{n+1}(s) \rightarrow \mathcal{J}_n(s)$ be a cofibration between cofibrant objects in $\mathcal{C}$, and 
let $\mathcal{J}\co (\mathcal{D}_{i}^S)^{\op} \longrightarrow \mathcal{C}$ be the functor given by $\mathcal{J}(f) = \Smash_{s\in S} \mathcal{J}_{f(s)}(s)$.
equip $(\mathcal{E}^S_i)^{\op}$ with the Reedy structure of Definition \ref{Reedy cube}.
Then the restriction of $\mathcal{J}$ to $(\mathcal{E}^S_i)^{\op}$ is Reedy-cofibrant.

Furthermore, if $N$ is an integer and we additionally assume that 
$E_*$ is a connective generalized homology theory (as defined in Definition \ref{def of gen hom thy}) such
that $E_m(\mathcal{J}(f))\cong 0$ for all $m<N$ and $f\in (\mathcal{D}_{i}^S)^{\op}$, then $E_m\left(\colim\mathcal{J}\right)\cong 0$ for all $m<N$ and all integers $i$.
\end{lemma}

\begin{proof}
Choose an object $f\in(\mathcal{E}^S_{i})^{\op}$. Then let $(\mathcal{E}^S_{1;f})^{\op}$ be the punctured overcategory of $f$, ie, $(\mathcal{E}^S_{1;f})^{\op}$ is the full subcategory of the overcategory $(\mathcal{E}^S_{i})^{\op}\downarrow f$ generated by all objects other than $\id_f$. 
The latching object $L_f(\tilde{\mathcal{J}})$ is simply the colimit of $\tilde{\mathcal{J}}$ restricted to $((\mathcal{E}^S_{1; f})^{\op})_{+}=(\mathcal{E}^S_{1; f})^{\op}$, and in $(\mathcal{E}^S_{1; f})^{\op}$  we have the cofinal subcategory consisting of all functions $g\co S \rightarrow\mathbb{N}$ such that $g\neq f$ and such that $f(s) \leq g(s) \leq f(s) + 1$ for all $s\in S$. 

The composite map 
\[ \colim\left( \tilde{\mathcal{J}}\right)^{\op} \stackrel{\cong}{\longrightarrow} 
L_f(\tilde{\mathcal{J}}) \rightarrow \tilde{\mathcal{J}}(f)\] 
is precisely the pushout product of the maps $\{\mathcal{J}_{f(s)+1}(s)\rightarrow \mathcal{J}_{f(s)}(s)\}_{s\in S}$ and therefore it is a cofibration by the pushout product axiom. Hence, $\tilde{\mathcal{J}}$ is Reedy cofibrant.

Since $\tilde{\mathcal{J}}$ is also projectively cofibrant by Remark~\ref{projective coincides with reedy}, $\colim\tilde{\mathcal{J}}$ is weakly equivalent to $\left| \sr\left(\tilde{\mathcal{J}}\right)\right|$ by Theorem~\ref{thm from gambino}. Since the category $(\mathcal{E}^S_{i})^{\op}$ is cofinal in $(\mathcal{D}_{i}^S)^{\op}$,
\[ E_*\left(\colim\mathcal{J} \right) \cong 
E_*\left(\colim\tilde{\mathcal{J}}\right). \] 
Together, these two facts imply that $E_*\left(\colim\left(\mathcal{J} \right)\right)$ is concentrated in degrees $\geq N$ by Theorem~\ref{homological bkss existence thm}. 

\end{proof}

\begin{lemma} \label{conn lem 2}
Suppose $E_*$ is a connective generalized homology theory as defined in Definition \ref{def of gen hom thy}, and $\mathcal{M}_i^{S}(I_{\bullet})$ is the $i$-th degree of the May filtration for a finite set $S$ and a cofibrant decreasingly filtered commutative monoid $I_{\bullet}$ as defined in Definition \ref{def of dec filt comm mon}. Then, if 
$E_m(I_i)\cong 0$ for all $m,i\in \mathbb{N}$ such that $m<i$, then 
$ E_m(\mathcal{M}_i^{S}(I_{\bullet}))\cong 0 $ 
for all $m,i\in \mathbb{N}$ such that $m< i$. 
\end{lemma}
\begin{proof}
Immediate from Lemma~\ref{connectivity of truncated cube colims} and the connectivity hypotheses in Definition~\ref{def of gen hom thy}.

\end{proof}
\begin{theorem}\label{subprop on thh-may ss} 
Suppose $I_{\bullet}$ is a 
Hausdorff cofibrant decreasingly filtered commutative monoid in $\mathcal{C}$, 
$X_{\bullet}$ is a simplicial finite set, and $E_*$ is a connective generalized homology theory on $\mathcal{C}$. Suppose $E_*(-)\cong [\Sigma^*Z,-\wedge E]$ for some objects $Z$ and $E$ in $\mathcal{C}$. 
Suppose the following connectivity axiom: $E_m(I_n)\cong 0$ for all $m<n$.

Then the topological Hochschild-May spectral sequence is strongly convergent, its differential satisfies the graded Leibniz rule, and its input and output and differential are as follows:
\begin{eqnarray*}
E^1_{s,t}\cong E_{s,t}(X_{\bullet}\otimes E_0^*I_{\bullet}) & \Rightarrow & E_{s}(X_{\bullet}\otimes I_{0}) \\
d^r\co E^r_{s,t} & \rightarrow & E^r_{s-1,t+r} 
\end{eqnarray*}
\end{theorem}
\begin{proof}
It is standard (see eg the section on Adams spectral sequences in~\cite{MR1718076}) that
 the $E$-homology spectral sequence of a tower of cofiber sequences of the form~\eqref{tower of sequences} converges to $E_*\left(\left| \mathcal{M}_0^{X_{\bullet}}(I_{\bullet})\right|\right)$ as long as 
$\left[\Sigma^* Z,  \holim_i \left( E\Smash\left| \mathcal{M}_i^{X_{\bullet}}(I_{\bullet})\right| \right) \right]$ is trivial.
By Lemma~\ref{conn lem 2} 
$E_m\left( \left| \mathcal{M}_i^{X_{\bullet}}(I_{\bullet})\right| \right) \cong 0$ for all $m<i$,
so by Lemma~\ref{connectivity conditions lemma}, 
\[ \left[\Sigma^n Z, \holim_i \left(E \Smash \left| \mathcal{M}_i^{X_{\bullet}}(I_{\bullet})\right|\right)\right] \cong 0 \]
 for all $n$, as desired.
Hence the spectral sequence converges to 
$E_*\left(\left| \mathcal{M}_0^{X_{\bullet}}(I_{\bullet})\right|\right)$ and by Theorem \ref{thm on fund theorem} this is isomorphic to $E_*(X_{\bullet}\otimes I_0).$ 

That the differential has the stated bidegree is a routine and easy 
computation in the spectral sequence of a tower of cofiber sequences. The sequence 
\[ \xymatrix{ \dots \ar[r] & \left | \mathcal{M}_2^{X_{\bullet}}(I_{\bullet}) \right |  \ar[r] & \left | \mathcal{M}_1^{X_{\bullet}}(I_{\bullet}) \right | \ar[r] &  \left | \mathcal{M}_0^{X_{\bullet}}(I_{\bullet}) \right | } \]
is a cofibrant decreasingly filtered commutative monoid in $\mathcal{C}$ as observed in Remark \ref{rem on structure} and therefore, in particular, it produces a ``pairing of towers'' in the sense of ~\cite{2003math......5173D} and therefore by Proposition 5.1 of ~\cite{2003math......5173D} the differentials in the spectral sequence satisfy a graded Leibniz rule.

Strong convergence is also standard: the connectivity axiom, the ``connectivity of smash products'' axiom from Definition~\ref{def of gen hom thy}, and Lemma~\ref{conn lem 2} together imply that our spectral sequence has a nondecreasing upper vanishing curve already at the $E^1$-term, so the spectral sequence converges strongly.
\end{proof} 

\begin{remark}\label{remark on day conv}
Another construction of our $THH$-May spectral sequence
\begin{equation}\label{our ss 4308} E^1_{*,*} \cong E_{*,*}(X_{\bullet} \otimes E_0^{*}I_{\bullet})  \Rightarrow E_{*}(X_{\bullet} \otimes I_0)\end{equation}
is possible using the Day convolution product. This construction is conceptually cleaner, but it does not, to our knowledge, simplify the process of proving that the resulting spectral sequence has the correct input term, output term and convergence properties.

Recall from Remark \ref{rem filt comm mon} that a cofibrant decreasingly filtered commutative monoid in $\mathcal{C}$ (without the cofibrancy assumption) is the same data as an object in $\Comm \mathcal{C}^{\mathbb{N}^{\op}}$. A cofibrant object in $\Comm \mathcal{C}^{\mathbb{N}^{\op}}$ with the projective model structure is, by Lemma \ref{lem cof proj}, a cofibrant decreasingly filtered commutative monoid in $\mathcal{C}$ (including that cofibrancy assumption).

Now fix a simplicial finite set $X_{\bullet}$ and a cofibrant commutative monoid object $I$ in $\mathcal{C}^{\mathbb{N}^{\op}}$. Hence, we can form the pretensor product $X_{\bullet}\tilde{\otimes} I_{\bullet}$, which is a simplicial object in 
$\Comm \mathcal{C}^{\mathbb{N}^{\op}}$. For example, if $X_{\bullet}$ is the usual minimal simplicial model for the circle $S_{\bullet}^1$, then 
$X_{\bullet}\tilde{\otimes} I_{\bullet}$ is the cyclic bar construction using the Day convolution as the tensor product:
\begin{equation*}
S^1_{\bullet} \tilde{\otimes} I = 
\left( 
\vcenter{ \xymatrix{I \ar[r] &   I\otimes_{\Day} I \ar@<1ex>[l]\ar@<-1ex>[l]\ar@<1ex>[r]\ar@<-1ex>[r] &  I\otimes_{\Day} I\otimes_{\Day} I \ar@<2ex>[l]\ar@<-2ex>[l]\ar[l]                            \ar@<2ex>[r]\ar@<-2ex>[r]\ar[r] & \dots \ar@<3ex>[l]\ar@<-3ex>[l]\ar@<1ex>[l]\ar@<-1ex>[l]   }   } \right )\end{equation*}

Since $I$ is a functor $\mathbb{N}^{\op} \rightarrow \mathcal{C}$, we will write $I(n)$ for the evaluation of this functor at a nonnegative integer $n$. (If we instead think of $I$ as a decreasingly filtered commutative monoid, as in most of the rest of this paper, we would write $I_n$ instead of $I(n)$.)
We write $\left( S^1_{\bullet}\tilde{\otimes} I \right ) (i)$ for the the simplicial object in $\mathcal{C}$ 
\begin{equation*}
\left (S^1_{\bullet} \tilde{\otimes} I \right) (i) = 
\left( 
\vcenter{ \xymatrix{I (i) \ar[r] &   (I\otimes_{\Day} I) (i) \ar@<1ex>[l]\ar@<-1ex>[l]\ar@<1ex>[r]\ar@<-1ex>[r] &  (I\otimes_{\Day} I\otimes_{\Day} I )(i)\ar@<2ex>[l]\ar@<-2ex>[l]\ar[l]                            \ar@<2ex>[r]\ar@<-2ex>[r]\ar[r] & \dots \ar@<3ex>[l]\ar@<-3ex>[l]\ar@<1ex>[l]\ar@<-1ex>[l]   }   } \right )\end{equation*}
Applying geometric realization to $\left (S^1_{\bullet}\tilde{\otimes} I\right )(i)$, we get a cofibrant decreasingly filtered object in in $\mathcal{C}$ (assuming Running Assumption \ref{ra:2})
\[ \left | \left (S^1_{\bullet}\tilde{\otimes} I \right ) (0) \right | \leftarrow \left | \left( S^1_{\bullet}\tilde{\otimes} I\right ) (1) \right |  \leftarrow \left| \left( S^1_{\bullet}\tilde{\otimes} I\right) (2) \right| \leftarrow \dots\]
and the spectral sequence obtained by applying a generalized homology theory $E_*$ to this cofibrant decreasingly filtered object in $\mathcal{C}$ is precisely the spectral sequence~\eqref{our ss 4308}, the spectral sequence constructed and considered throughout this paper. (It is an easy exercise in unwinding definitions to check that this spectral sequence agrees with the one constructed in Definition~\ref{def of thh-may ss}, but to verify that the resulting spectral sequence has the expected input term, output term, and convergence properties amounts to exactly the same proofs already found in this paper which aren't expressed in terms of Day convolution.)
\end{remark}

\section{Decreasingly filtered commutative ring spectra.} \label{section 2}
\subsection{Convenient model structures on functor categories.} \label{convenient model cat for filtered objects}
For this section we will need some additional assumptions on our model category $\mathcal{C}$ satisfying \ref{ra:1}. These are not used to construct the spectral sequence, but they are used to construct a large class of examples of decreasingly filtered commutative ring spectra in Section \ref{whitehead}. Recall that the symbol $\square$ was defined in Lemma~\ref{n=2 case of hard lemma}.
\begin{definition}[Definition 3.4 of \cite{MR3666740}] \label{def of scma}
We say a model category $\mathcal{A}$ satisfies the strong commutative monoid axiom if, for any cofibration $h$ (respectively,  any acyclic cofibration $h$) the map $h^{\square n}/\Sigma_n$ is a cofibration (respectively,  an acyclic cofibration) for each integer $n\ge 1$. 
\end{definition}
\begin{runningassumption}\label{ra:3}
We assume that $\mathcal{C}$ satisfies the strong commutative monoid axiom. We also assume that $\mathcal{C}$ is locally presentable as a category. Additionally, we assume $\mathcal{C}$ is a right proper model category.
\end{runningassumption}
For example, symmetric spectra in the positive flat stable model structure satisfy Running Assumption~\ref{ra:3}; see~\cref{whitehead}.
\begin{remark}
If $\mathcal{C}$ is a model category satisfying Running Assumption \ref{ra:1} and Running Assumption \ref{ra:2}, then $\mathcal{C}$ is cofibrantly generated and locally presentable and therefore it is combinatorial in the sense of Jeff Smith (see Definition 2.1 in Dugger \cite{MR1870516}). 
\end{remark} 
We will also make use of categories enriched in a model category $\mathcal{C}$ satisfying Running Assumption \ref{ra:1}. See Kelly \cite{MR2177301}, for a good treatment of enriched category theory at the level of generality needed for this paper. 

Given a category $\mathcal{B}$ enriched over $\mathcal{C}$, we will also discuss the category of $\mathcal{C}$-functors on $\mathcal{B}$, written $\mathcal{C}^{\mathcal{B}}$. It can be shown (see Kelly \cite{MR2177301} for example) that $\mathcal{C}^{\mathcal{B}}$ can again be equipped with a $\mathcal{C}$-enriched category structure, but we do not use that property here and we will write $\mathcal{C}^{\mathcal{B}}$ for the underlying category of enriched functors and natural transformations. 

\begin{definition}\label{natural numbers}
Suppose $\mathcal{C}$ satisfies Running Assumption \ref{ra:1}. Let  $\mathcal{N}^{\op}$ be the $\mathcal{C}$-enriched category with objects $\ob\mathcal{N}^{\op}=\mathbb{N}$ and morphism objects 
\[ \mathcal{N}^{\op}(n,m)= \left \{\begin{array}{c} \mathbbm{1} \text{ if } n\ge m \\ 0 \text{ if } n<m \end{array} \right. \]
where $\mathbbm{1}$ is the unit in $\mathcal{C}$ and $0$ is the zero object of $\mathcal{C}$. We equip $\mathcal{N}^{\op}$ with a symmetric monoidal product using the usual addition $+$ so that $0$ is the unit of the symmetric monoidal structure on $\mathcal{N}^{\op}.$ This makes $\mathcal{N}^{\op}$ a symmetric monoidal $\mathcal{C}$-category. 
\end{definition}
\begin{remark}\label{equiv of cats Day}
In Remark~\ref{rem filt comm mon} we discussed the theorem, due to Day in \cite{Day}, that the category of lax symmetric monoidal functors in $\mathcal{C}^{\bN^{\op}}$ is equivalent to the category $\Comm \mathcal{C}^{\bN^{\op}}$. In the setting of categories enriched in $\mathcal{Top}_*,$ based compactly-generated weak-Hausdorff spaces, this is proven in Proposition 22.1 of~\cite{MR1806878}. In the setting of categories enriched in a model category $\mathcal{C}$ satisfying Running Assumption \ref{ra:1} the same proof yields the desired result that the category of lax symmetric monoidal functors in $\mathcal{C}^{\mathcal{N}^{\op}}$ is equivalent to the category of commutative monoids in $\mathcal{C}^{\mathcal{N}^{\op}}$ with the enriched Day convolution symmetric monoidal product
\end{remark}
The following truncated versions of the category $\mathcal{N}^{\op}$ will also be useful for constructing decreasingly filtered commutative ring spectra. 
\begin{definition}\label{def J_n}
We will define a category enriched in a category $\mathcal{C}$ satisfying Running Assumption \ref{ra:1} and \ref{ra:3} called $J_n^{\op}$. The objects in $J_n$ are the subset $J_n \subset \mathbb{N}$ of all natural numbers consisting of all $i\in \mathbb{N}$ such that $i\le n$. We then define $J_n^{\op}$ as a sub-$\mathcal{C}$-category of $\mathcal{N}^{\op}$ so that 
$J_n^{\op}(i,j)= \mathcal{N}^{\op}(i,j) $
whenever $i,j\le n$. We give this category the structure of a symmetric monoidal $\mathcal{C}$-category $(J_n, \dot{+}, 0)$ by letting 
$ i \dot{+} j = \text{min} \{ i+j, n\}. $
\end{definition}
We now discuss the cofibrancy conditions needed on an object in $\Comm \mathcal{C}^{\mathbb{N}^{\op}}$ in order to produce a cofibrant decreasingly filtered commutative monoid in $\mathcal{C}$ in the sense of Definition \ref{def of dec filt comm mon}. This will be used to construct a certain class of filtered commutative ring spectra in Theorem \ref{post filt}. In Proposition~\ref{model structure on filtered objects}, we use a definition from~\cite{MR2713397}: an object $X$ in a pointed monoidal model category is {\em virtually cofibrant} if $(0\rightarrow X)\square -$ preserves cofibrations and preserves acyclic cofibrations. 
\begin{prop}\label{model structure on filtered objects}
Suppose $\mathcal{C}$ is a model category satisfying Running Assumption \ref{ra:1} and Running Assumption \ref{ra:2} and $\mathcal{B}$ is a small symmetric monoidal category enriched in $\mathcal{C}$ with virtually cofibrant function spaces \cite[Definition 2.2.12]{MR2713397}. Then there exists a model structure on the category of $\mathcal{C}$-functors $\mathcal{C}^{\mathcal{B}}$ called the \emph{projective model structure} where a fibration (respectively,  weak equivalence) is a natural transformation $\eta\co F\rightarrow G$ between functors $F,G\in \mathcal{C}^{\mathcal{B}}$ such that $\eta_X\co F(X)\rightarrow G(X)$ is a fibration (respectively,  weak equivalence) for each object $X$ in $\mathcal{B}$. In addition, because $\mathcal{C}$ satisfies Running Assumption \ref{ra:1}, $\mathcal{C}^{\mathcal{B}}$ is a symmetric monoidal model category under Day convolution and it satisfies the Schwede-Shipley monoid axiom~\cite{MR1734325}. 
\end{prop}
\begin{proof}
Since $\mathcal{C}$ is assumed to be combinatorial, the projective model structure exists on $\mathcal{C}^{\mathcal{B}}$ and it is also cofibrantly generated by Proposition 2.2.13 of \cite{MR2713397}. Since $\mathcal{C}$ satisfies Running Assumption \ref{ra:1} and Running Assumption \ref{ra:3}, Proposition 2.2.15 and Proposition 2.2.16 in \cite{MR2713397} imply the second part of the proposition. 
\end{proof}
\begin{remark}
We will only apply Proposition \ref{model structure on filtered objects} in the case where $\mathcal{B}$ is either the category $\mathcal{N}^{\op}$ of Definition \ref{natural numbers} or $J_n^{\op}$ of Definition \ref{def J_n}. It is easy to check that the categories $\mathcal{N}^{\op}$ and $J_n^{\op}$ have virtually cofibrant function spaces because 
the unit object $\mathbbm{1}$ is always virtually cofibrant in a symmetric monoidal model category by the unit axiom.
\end{remark}
\begin{lemma} \label{lem cof proj}
Let $\mathcal{C}$ satisfy Running Assumption \ref{ra:1} and Running Assumption \ref{ra:3} and let $\mathcal{N}^{\op}$ be the category defined Definition \ref{natural numbers}. Let $\mathcal{C}^{\mathcal{N}^{\op}}$ be the category of $\mathcal{C}$-functors $\mathcal{N}^{\op}\rightarrow \mathcal{C}$, equipped with the projective model structure. Let $P$ be a cofibrant object in $\mathcal{C}^{\mathcal{N}^{\op}}$. Then, for all $n\in\mathbb{N}$, the object $P(n)$ of $\mathcal{C}$ is cofibrant, and the morphism $P(n+1) \rightarrow P(n)$ is a cofibration in $\mathcal{C}$.
\end{lemma} 
\begin{proof} 
This is a consequence of Lemma 3.1 in \cite{160201515}. 
\end{proof}
The following definition is from Proposition~2.2.13 of \cite{MR2713397}, 
and it will be useful for describing the generating cofibrations in the projective model structure on $\mathcal{C}^{\mathcal{N}^{\op}}$. 
\begin{definition}\label{def for notation}
Suppose $\mathcal{C}$ is a symmetric monoidal model category satisfying Running Assumption \ref{ra:1} and Running Assumption \ref{ra:3}. Each integer $i$ corepresents a functor $\mathcal{N}^{\op}\rightarrow\mathcal{C}$, and we write $\mathcal{N}^{\op}(i, -)$ for this functor. Given a set of morphisms $\mathcal{M}$ in $\mathcal{C}$, we write $\mathcal{M}\otimes \mathcal{N}^{\op}$ for the collection of morphisms in $\mathcal{C}^{\mathcal{N}^{\op}}$ which are of the form $\overline{f}\Smash \mathcal{N}^{\op}(i,-)$ for some integer $i$ and some map $\overline{f}$ in $\mathcal{M}$.
\end{definition}
\begin{theorem}\label{lem scma}
Suppose $\mathcal{C}$ satisfies Running Assumption \ref{ra:1} and Running Assumption \ref{ra:3} and $\mathcal{N}^{\op}$ is the category defined in Definition \ref{natural numbers}. Then the projective model structure on $\mathcal{C}^{\mathcal{N}^{\op}}$ satisfies the strong commutative monoid axiom. 
\end{theorem}
\begin{proof}
By White \cite[Lem. A.1]{MR3666740}, it suffices to check the strong commutative monoid axiom on the generating cofibrations. Due to Proposition 4.5 of \cite{MR3380069}, a set of generating cofibrations in the projective model structure on $\mathcal{C}^{\mathcal{N}^{\op}}$ is the class of maps $\mathcal{J}\otimes \mathcal{N}^{\op}$, where $\mathcal{J}$ is a set of generating cofibrations of $\mathcal{C}$, and a set of generating acyclic cofibrations in the projective model structure is the set of maps $\mathbb{J}\otimes \mathcal{N}^{\op}$, where $\mathbb{J}$ is a set of generating acyclic cofibrations of $\mathcal{C}$ (see Definition \ref{def for notation} for the definition of the notation). 

Since 
\[ \mathcal{N}^{\op}(i,-)\simeq \left \{ \begin{array}{c} \mathbbm{1} \text{ if } j\le  i \\ 0 \text{ if} j>i \end{array} \right. \]
a set of generating cofibrations for $\mathcal{C}^{\mathcal{N}^{\op}}$ consists of the set of the natural transformations $f\co I\rightarrow J$ such that, for some $\bar{f}\co A\rightarrow B$ in $\mathcal{J}$, $f_j=\bar{f}\Smash \id_{\mathbbm{1}}\co A\Smash \mathbbm{1}\rightarrow B\Smash \mathbbm{1}$ for $j\le i$ and 
$f_j=\bar{f}\Smash \id_0\co A\Smash 0 \rightarrow B\Smash 0$ for $j> i$. 
Similarly, a set of generating acyclic cofibrations consists of the natural transformations $f\co I\rightarrow J$ such that there is some map $\bar{f}\co A\rightarrow B$ in $\mathbb{J}$ and $f_j=\bar{f}\Smash id_\mathbbm{1}\co A\Smash \mathbbm{1}\rightarrow B\Smash \mathbbm{1}$ for $j\le i$ and $f_j=\bar{f}\Smash id_0\co A\Smash 0 \rightarrow B\Smash 0$ for $j>i$. 

Let $h\co I\rightarrow J$ be an map in $\mathcal{J}\otimes \mathcal{N}^{\op}$ (respectively,  a map in $\mathbb{J}\otimes \mathcal{N}^{\op}$) with $\bar{h}:A\rightarrow B$ the corresponding map in $\mathcal{J}$ (respectively,  $\mathbb{J}$). Then we need to prove that $h^{\square n}/\Sigma_n$ is a cofibration (respectively,  an acyclic cofibration) in $\mathcal{C}^{\bN^{\op}}$ with the projective model structure. 
The case $n=1$ is vacuous and therefore we omit it. 

In the case $n=2$, we need to show that the map
\[ h^{\square 2}/\Sigma_2\co(I\otimes_{Day} J \coprod_{I\otimes_{Day} I} J\otimes_{Day} I )/\Sigma_2 \rightarrow (J\otimes_{Day}J)/\Sigma_2 \]
is an (acyclic) cofibration in the projective model structure. To see this, note that 
$I\otimes_{Day}I\cong (A\Smash A)\Smash \mathcal{N}^{\op}(2i,-),$ 
$ J\otimes_{Day}I\cong (B\Smash A)\Smash \mathcal{N}^{\op}(2i,-),$
$I\otimes_{Day}J\cong (A\Smash B)\Smash \mathcal{N}^{\op}(2i,-),$ and 
$J\otimes_{Day}J\cong (B\Smash B)\Smash \mathcal{N}^{\op}(2i,-).$
Therefore, the map $(h^{\square 2})/\Sigma_2$ is the map $(\bar{h}^{\square 2} \otimes \mathcal{N}^{\op}(2i, -))/\Sigma_2$ up to isomorphism. 
By definition of $\mathcal{N}^{\op}$, this is equivalent, up to isomorphism, to $((\bar{h}^{\square 2} )/\Sigma_2)  \Smash \mathcal{N}^{\op}(2i,-)$, that is, $(h^{\square 2})/\Sigma_2$ is a composite of isomorphisms and the map
$((\bar{h}^{\square 2} )/\Sigma_2)  \Smash \mathcal{N}^{\op}(2i,-)$. 
Since $\mathcal{C}$ satisfies the strong commutative monoid axiom, $(\bar{h}^{\square 2} )/\Sigma_2$ is a cofibration (respectively,  acyclic cofibration). We claim that the map $\left((\bar{h}^{\square 2} )/\Sigma_2\right )\Smash \mathcal{N}^{\op}(2i,-)$ is also a cofibration (respectively,  acyclic cofibration). 
We will give the argument that $\left((\bar{h}^{\square 2} )/\Sigma_2\right )\Smash \mathcal{N}^{\op}(2i,-)$ is a cofibration in the projective model structure when $h$ is a cofibration; the argument that $\left((\bar{h}^{\square 2} )/\Sigma_2\right )\Smash \mathcal{N}^{\op}(2i,-)$ is an acyclic cofibration when $h$ is an acyclic cofibration will be omitted because it is essentially the same. 
We need to show that for any acyclic fibration $X\rightarrow Y$ in $\mathcal{C}^{\mathcal{N}^{\op}}$ fitting into the diagram
\[ 
\xymatrix{
W \ar[r] \ar[d] & X \ar[d] \\
Z \ar[r] \ar@{-->}[ur]^{z} & Y }
\] 
there exists a lift $z\co Z\rightarrow X$, where $W=(I\otimes_{Day} J \coprod_{I\otimes_{Day} I} J\otimes_{Day} I )/\Sigma_2$ and $Z=(J\otimes_{Day}J)/\Sigma_2$. Since $\left((\bar{h}^{\square 2} )/\Sigma_2\right )\Smash\mathcal{N}^{\op}(2i,m): W_m\rightarrow Z_m$ is, up to isomorphism, the cofibration
\[ \left(A\Smash B \coprod_{B\Smash B} B\Smash A\right) /\Sigma_2\rightarrow \left(B\Smash B\right)/\Sigma_2 \]
when $m\le 2i$ and it is the map $0\rightarrow 0$ when $m>2i$ up to isomorphism, we can define the map 
$z\co Z\rightarrow X$ to be the map $0\rightarrow X_m$ when $m>2i$, we can define $z_{2i}\co Z_{2i}\rightarrow X_{2i}$ to be the lift in the diagram
\[
\xymatrix{
W_{2i}\ar[r] \ar[d] & X_{2i} \ar[d]\\
Z_{2i} \ar[r] \ar@{-->}[ur]^{z_{2i}} & Y_{2i}, }
\]
which we know exists because $W_{2i}\rightarrow Z_{2i}$ is a cofibration and $X_j\rightarrow Y_j$ is an acyclic fibration for all $j$ (because $X\rightarrow Y$ is an acyclic fibration in the projective model structure). For all $j<2i$, we can then define $z_j$ as the composite
\[ Z_j = Z_{2i} \stackrel{z_{2i}}{\longrightarrow} X_{2i} \rightarrow X_{2i-1} \rightarrow \dots \rightarrow X_j.\] 
There is no difficulty with the need for our lift maps to be compatible for various choices of $j$, since in the range $j\leq 2i$, the sequences $W_{\bullet}$ and $Z_{\bullet}$ are constant, and above this range, these sequences are zero! We have therefore defined a map $z_m\co Z_m\rightarrow X_m$ for each $m\ge 0$ and the diagrams
\[ 
\xymatrix{
Z_m \ar[r]\ar[d] & X_m \ar[d] \\
Z_{m-1} \ar[r] & X_{m-1} }\]
clearly commute for all $m\ge 0$ by the definition of $z_m$. 

The same type of argument works for $n>2$ giving 
\[ (h^{\square n})/\Sigma_n \cong ((\bar{h}^{\square n} )/\Sigma_n)  \Smash \mathcal{N}^{\op}(ni,-) \]
which is a cofibration by a proof essentially the same as the one above. Here the isomorphism is in the arrow category, but this is the same as saying $(h^{\square n})/\Sigma_n$ is a composite of isomorphisms and the map $((\bar{h}^{\square n} )/\Sigma_n)  \Smash \mathcal{N}^{\op}(ni,-)$, which is a cofibration and therefore $(h^{\square n})/\Sigma_n$ is also a cofibration. 
\end{proof}

\begin{lemma} 
Let $\mathcal{C}$ be a model category satisfying Running Assumption \ref{ra:1} and Running Assumption \ref{ra:3} and let $J_n^{\op}$ be the category defined in \ref{def J_n}. Then the projective model structure on $\mathcal{C}^{J_n^{\op}}$ satisfies the strong commutative monoid axiom. The cofibrant objects in the projective model structure are \emph{exactly} the objects $P\in \mathcal{C}^{J_n^{\op}}$ such that $P(i)$ is cofibrant in $\mathcal{C}$ for all $0\le i\le n$ and $P(i)\rightarrow P(i-1)$ is a cofibration in $\mathcal{C}$. 
\end{lemma}
\begin{proof}
The proof that $\mathcal{C}^{J_n^{\op}}$ satisfies the strong commutative monoid axiom is the same as the proof of Theorem \ref{lem scma} and therefore we omit it. The second part of the lemma is again a consequence of Lemma 3.1 in \cite{160201515}, as the authors remark just after the lemma. 
\end{proof}
\begin{remark}
There are four different model structures that we use here, which are all commonly referred to as the projective model structure. 
The model category structure on $\Sp^{\mathcal{N}^{\op}}$ discussed in Proposition \ref{model structure on filtered objects}, 
the model structure created by the forgetful functor $U:\Comm \Sp^{\mathcal{N}^{\op}}\rightarrow \Sp^{\mathcal{N}^{\op}}$, 
the model structure on algebras over a commutative monoid $I$ in $\Comm \Sp^{\mathcal{N}^{\op}}$ created by the forgetful functor to $\Comm \Sp^{\mathcal{N}^{\op}}$, and the model structure on modules over a commutative monoid $I$ in 
$\Comm \Sp^{\mathcal{N}^{\op}}$ created by the forgetful functor from $I$-modules (or equivalently symmetric $I$-bimodules) to $\Sp^{\mathcal{N}^{\op}}.$ Each of these will be referred to as the projective model structure, and it should be clear from the context which of the four is meant. Each of these model structures also make sense when $\mathcal{N}^{\op}$ is replaced by $J_n^{\op}$. 
Since the model structure on $\Sp$ is cofibrantly generated, each these projective model structures is cofibrantly generated by Theorem \ref{model structure on filtered objects}, \cite[Theorem 3.2]{MR3666740}, and  \cite[Theorem 4.1]{MR1734325}. 
Consequently, as remarked in Hovey \cite{MR1650134} we have functorial factorization in each of these model categories. 
\end{remark}

\subsection{Whitehead towers.}\label{whitehead}
For this section, we will abbreviate and write $\Sp$ for the category of symmetric spectra of pointed simplicial sets with the positive flat stable model structure. This category satisfies Running Assumptions \ref{ra:1} and \ref{ra:2}, as discussed in Section \ref{conv and ra}. This model category also satisfies Running Assumption \ref{ra:3}: since it is combinatorial by Hovey-Shipley-Smith \cite{MR1695653}, it satisfies the strong commutative monoid axiom by Theorem 5.7 in \cite{MR3666740}, and it is also right proper by Theorem 5.4.2 in \cite{MR1695653}. 
The goal of this section is to produce a cofibrant decreasingly filtered commutative monoid in $\Sp$ as a specific multiplicative model for the Whitehead tower of a connective commutative monoid in $\Sp$. 

As a consequence of Lemma \ref{lem cof proj}, a cofibrant object in the category $\Comm \Sp^{\mathcal{N}^{\op}}$ equipped with the projective model structure is, in particular, a cofibrant decreasingly filtered commutative monoid in $\Sp$ (also see Remark \ref{rem filt comm mon}). 

\begin{theorem} \label{post filt} Let $R$ be a cofibrant connective commutative monoid in $\Sp$. Then there exists a cofibrant decreasingly filtered commutative monoid in $\Sp$
\[ \xymatrix{ \dots \ar[r]  & \tau_{\ge 2} R \ar[r] & \tau_{\ge 1} R \ar[r] & \tau_{\ge 0}R  } \] 
with structure maps 
\[ \rho_{i,j}\co \tau_{\ge i } R \Smash \tau_{\ge j} R \longrightarrow \tau_{\ge i+j}R \]
such that $R\simeq \tau_{\ge 0}R$ and the maps 
$\tau_{\ge n}R\rightarrow \tau_{\ge 0}R$ induce an isomorphism on $\pi_k$ for $k\ge n$ and $\pi_k(\tau_{\ge n}R)\cong 0$ for $k<n$. This cofibrant decreasingly filtered commutative monoid in $\Sp$ is denoted $\tau_{\ge \bullet}R$. 
\end{theorem}
\begin{proof}
We will prove the theorem by induction. First, we will consider $R$ as a cofibrant object $\Comm\Sp^{J_0^{\op}}=\Comm\Sp$ and we will produce a cofibrant object in $\Comm\Sp^{J_1^{\op}}$ in projective model structure. To construct $\tilde{\tau}_{\ge 1} R$, we consider the map of commutative $R$-algebras $R\to H\pi_0R$, constructed as in \cite[Thm. 8.1]{MR1732625}. This map is not usually a fibration (in fact it is a cofibration as discussed in Remark \ref{gabes remark on cofibrancy of degree 0 quotient}), 
so we factor it as a composite $R\to  \tilde{\tau}_{\ge 0} R  \to  H\pi_0R $ of an acyclic cofibration and a fibration in the category of commutative $R$-algebras in $\Sp$. The resulting object $ \tilde{\tau}_{\ge 0} R$ is a cofibrant commutative monoid in $\Sp$ equipped with a fibration $\tilde{\tau}_{\ge 0} R\to H\pi_0R$, which is a map of $\tilde{\tau}_{\ge 0} R$-algebras. 

We then define $\tilde{\tau}_{\ge 1} R$ to be the fiber of the map $\tilde{\tau}_{\ge 0} R \rightarrow H\pi_0R$ in the category of $\tilde{\tau}_{\ge 0} R$-modules (equivalently symmetric $\tilde{\tau}_{\ge 0} R$-bimodules). The symmetric $\tilde{\tau}_{\ge 0} R$-bimodule structure produces an object $f_1:\tilde{\tau}_{\ge 1} R \rightarrow \tilde{\tau}_{\ge 0} R$ in $\Comm\Sp^{J_1^{\op}}$ with action maps $\rho_{i,j}$ for $i,j\in \ob J_1^{\op}$ defined as follows: the map $\rho_{0,0}\co \tilde{\tau}_{\ge 0} R\Smash \tilde{\tau}_{\ge 0} R\rightarrow \tilde{\tau}_{\ge 0} R$ is the multiplication map, the maps $\rho_{1,0}$ and $\rho_{0,1}$ are the right and left $\tilde{\tau}_{\ge 0} R$-module structure maps, and the map $\rho_{1,1}$ is the composite 
\[ \xymatrix{ \rho_{1,1}\co \tilde{\tau}_{\ge 1} R \Smash \tilde{\tau}_{\ge 1} R \ar[rr]^{\id_{\tilde{\tau}_{\ge 1} R}\Smash f_1} && \tilde{\tau}_{\ge 1} R\Smash\tilde{\tau}_{\ge 0} R \ar[r]^(.6){\rho_{1,0}}& \tilde{\tau}_{\ge 1} R} \]
or equivalently, because $\tilde{\tau}_{\ge 1} R$ is a symmetric $\tilde{\tau}_{\ge 0}R$-bimodule, 
\[ \xymatrix{ \rho_{1,1}\co \tilde{\tau}_{\ge 1} R \Smash \tilde{\tau}_{\ge 1} R \ar[rr]^{f_1\Smash \id_{\tilde{\tau}_{\ge 1} R}} &&\tilde{\tau}_{\ge 0} R \Smash \tilde{\tau}_{\ge 1} R \ar[r]^(.6){\rho_{0,1}} & \tilde{\tau}_{\ge 1} R. }\]
These maps are easily seen to satisfy the necessary associativity, commutativity, and compatibility axioms. We then cofibrantly replace $\tilde{\tau}_{\ge 1}R \rightarrow \tilde{\tau}_{\ge 0} R$ in the projective model structure on $\Comm\Sp^{J^{\op}_1}$ to produce an object $\tau_{\ge \bullet}^{\le 1}R$ with the property that $R\simeq \tau_{\ge 0}^{\le 1}R$ and the map $\tau_{\ge 1}^{\le 1}R\to \tau_{\ge 0}^{\le 1}R$ induces an isomorphism on $\pi_k$ for $k\ge 1$ and $\pi_k(\tau_{\ge 1}^{\le 1}R)\cong 0$ for $k<1$. Recall that the effect of cofibrantly replacing in this model structure is that the map is replaced by a cofibration and the objects are replaced by cofibrant objects without changing their homotopy type. This completes the base step in the induction. Note that since we have functorial factorization in $\Sp$ and $\Comm\Sp^{J^{\op}_1}$, this construction is entirely functorial. 

Now, for the inductive step: suppose we can functorially construct a cofibrant object 
$\tau^{\le n-1}_{\ge \bullet}(R)\in\ob \Comm\Sp^{J^{\op}_{n-1}}$
for an arbitrary $n\ge 1$ with the property that $R\simeq \tau^{\le n-1}_{\ge 0}R$ and the map $\tau^{\le n-1}_{\ge j}R \to \tau^{\le n-1}_{\ge 0}R$ induces an isomorphism on $\pi_k$ for $n-1\ge k\ge j$ and $\pi_k(\tau^{\le n-1}_{\ge j}R)\cong 0$ for $k<j$, and there is an acyclic cofibration $R\to \tau^{\le n-1}_{\ge 0}R$. 
First, note that, due to Basterra \cite[Thm. 8.1]{MR1732625}, we can construct a map of commutative $R$-algebras $R\rightarrow  \gamma_{\le n-1} R$ where $\gamma_{\le n-1}R$ is a commutative $R$-algebra with the property that the map $R \to \gamma_{\le n-1} R$ induces an isomorphism on $\pi_k$ for $k\le n-1$ and $\pi_k(\gamma_{\le n-1}R)\cong 0$ for $k\ge n$. By the assumed functoriality of the construction of $\tau^{\le n-1}_{\ge \bullet}(R)$, we get a map $\tau_{\ge \bullet}^{\le n-1}(R)\rightarrow \tau_{\ge \bullet}^{\le n-1}(\gamma_{\le n-1} R)$ of commutative $\tau_{\ge \bullet}^{\le n-1}(R)$-algebras in $\Comm\Sp^{{J_{n-1}}^{\op}}$. We first fibrantly replace $\tau_{\ge \bullet}^{\le n-1}(\gamma_{\le n-1} R)$ in the category of commutative $\tau_{\ge \bullet}^{\le n-1}(R)$-algebras, to produce a fibrant object $\bar{\tau}_{\ge \bullet}^{\le n-1}(\gamma_{\le n-1} R)$, which still receives a map from $\tau_{\ge \bullet}^{\le n-1}(R)$. We then
functorially factor the map $\tau_{\ge \bullet}^{\le n-1}(R)\rightarrow \bar{\tau}_{\ge \bullet}^{\le n-1}(\gamma_{\le n-1} R)$ as an acyclic cofibration followed by a fibration $\tau_{\ge \bullet}^{\le n-1}(R)\rightarrow  \tilde{\tau}^{\le n-1}_{\ge \bullet}(R)\rightarrow \bar{\tau}_{\ge \bullet}^{\le n-1}(\gamma_{\le n-1} R)$ in the category of commutative $\tau_{\ge \bullet}^{\le n-1}(R)$-algebras.

The following argument combines two ideas. As in the base step, we work in the category of symmetric $\tilde{\tau}_{\ge \bullet}^{\le n-1}(R)$ bimodules throughout. 
This builds in the commutativity, associativity, unitality and compatibility of most of the structure maps $\rho_{i,j}:\tilde{\tau}_{\ge i}^{\le n}(R)\wedge \tilde{\tau}_{\ge j}^{\le n}(R) \to \tilde{\tau}_{\ge i\dot{+}j }^{\le n}(R)$, where $\tilde{\tau}_{\ge i }^{\le n}(R):=\tilde{\tau}_{\ge i }^{\le n-1}(R)$ for $i<n$, as we discuss below. 
Unfortunately, this alone does not build in the associativity of the maps $\rho_{i,j}:\tilde{\tau}_{\ge i}^{\le n}(R)\wedge \tilde{\tau}_{\ge j}^{\le n}(R) \to \tilde{\tau}_{\ge n }^{\le n}(R)$ when $i,j<n$ and $i+j\ge n$. 
We therefore build this into our definition of $\tilde{\tau}_{\ge n}^{\le n}R$ as well. To combine these two ideas we must describe how to construct the associativity diagrams that we use in the category of $\tilde{\tau}_{\ge \bullet}^{\le n-1}(R)$ bimodules. 

Recall that a symmetric $\tilde{\tau}_{\ge \bullet}^{\le n-1}R$-bimodule is a functor $X\co J_{n-1}^{\op}\rightarrow\Sp$ along with natural transformations 
		\[ X\otimes_{Day} \tilde{\tau}_{\ge \bullet}^{\le n-1}R\rightarrow \tilde{\tau}_{\ge \bullet}^{\le n-1}R, \]
		\[ \tilde{\tau}_{\ge \bullet}^{\le n-1}R\otimes_{Day}X \rightarrow \tilde{\tau}_{\ge \bullet}^{\le n-1}R \]
satisfying the usual associativity and commutativity axioms. Let $F$ and $G$ be constant symmetric $\tilde{\tau}_{\ge \bullet}^{\le n-1}R$-bimodules such that $F(\ell)$ is 
\[  \colim_{i\dot{+}j\ge n} \tilde{\tau}_{\ge i}^{\le n-1}R \wedge \tilde{\tau}_{\ge j}^{\le n-1}R   \] 
for each $0 \le \ell \le n-1$ and $G(\ell)$ is 
\[  \colim_{i\dot{+}j\dot{+}k\ge n} \tilde{\tau}_{\ge i}^{\le n-1}R \wedge \tilde{\tau}_{\ge j}^{\le n-1}R  \wedge \tilde{\tau}_{\ge k}^{\le n-1}R \] 
for each $0 \le \ell \le  n-1$. To make this precise we need to define the action maps, which are natural transformations 
	\[
		\begin{array}{cc}  
			F\otimes_{Day} \tilde{\tau}_{\ge \bullet}^{\le n-1}R \rightarrow F, &
			  \tilde{\tau}_{\ge \bullet}^{\le n-1}R  \otimes_{Day}F \rightarrow F, \\
			\tilde{\tau}_{\ge \bullet}^{\le n-1}R  \otimes_{Day} G \rightarrow G, & 
			G \otimes_{Day} \tilde{\tau}_{\ge \bullet}^{\le n-1}R  \rightarrow G. 
		\end{array}
	\]
By the definition of Day convolution
\[(F\otimes_{Day} \tilde{\tau}_{\ge \bullet}^{\le n-1}R)(\ell)=\colim_{a\dot{+}b\ge \ell} \left ( F(a)   \Smash  \tilde{\tau}_{\ge b}^{\le n-1}R \right )\]
and since $F$ is constant, 
\[
 \begin{array}{l}\colim_{a\dot{+}b\ge \ell} \left ( F(a)   \Smash  \tilde{\tau}_{\ge b}^{\le n-1}R \right ) \cong\\ 
\colim_{i\dot{+}j\ge n} (\tilde{\tau}_{\ge i}^{\le n-1}R \wedge \tilde{\tau}_{\ge j}^{\le n-1}R )  \Smash  \tilde{\tau}_{\ge 0}^{\le n-1}R. \end{array}\]
Since smashing with a cofibrant object commutes with colimits, there is a weak equivalence
\[ 
\begin{array}{l}
\colim_{i\dot{+}j\ge n} (\tilde{\tau}_{\ge i}^{\le n-1}R \wedge \tilde{\tau}_{\ge j}^{\le n-1}R)   \Smash  \tilde{\tau}_{\ge 0}^{\le n-1}R\simeq  \\
\colim_{i\dot{+}j\ge n} (\tilde{\tau}_{\ge i}^{\le n-1}R \wedge \tilde{\tau}_{\ge j}^{\le n-1}R  \Smash  \tilde{\tau}_{\ge 0}^{\le n-1}R)\end{array}
\]
and we can use the maps $\rho_{j,0}:\tilde{\tau}_{\ge j}^{\le n-1}R\wedge \tilde{\tau}_{\ge 0}^{\le n-1}R \rightarrow \tilde{\tau}_{\ge j }^{\le n-1}R$ to define a map 
\[ \colim_{i\dot{+}j\ge n} (\tilde{\tau}_{\ge i}^{\le n-1}R \wedge \tilde{\tau}_{\ge j}^{\le n-1}R  \Smash  \tilde{\tau}_{\ge 0}^{\le n-1}R)\rightarrow  \colim_{i\dot{+}j\ge n} (\tilde{\tau}_{\ge i}^{\le n-1}R \wedge \tilde{\tau}_{\ge j}^{\le n-1}R  )\]
which provides a natural transformation $F\otimes_{Day} \tilde{\tau}_{\ge \bullet}^{\le n-1}R \rightarrow F$ in the evident way. 
The remaining action natural transformation for $F$ is defined in an analogous way and the these two maps satisfy the commutativity and associativity axioms by the inductive hypothesis.
We then define the two action natural transformations of $G$ in the same way and they also satisfy the commutativity and associativity axioms by the inductive hypothesis. 

Here is a sketch of our next steps in this proof: we will now define $\tilde{\tau}_{\ge n}^{\le n}R$ in such a way that it is automatically equipped with associative, commutative, and compatible maps 
\[ \rho_{i,j}:\tilde{\tau}_{\ge i}^{\le n}R \wedge \tilde{\tau}_{\ge j}^{\le n}R \rightarrow  \tilde{{\tau}}_{\ge n}^{\le n} R\] 
for $i+j\ge n$, where  $\tilde{\tau}_{\ge j}^{\le n}R$ is defined to be $\tilde{\tau}_{\ge j}^{\le n-1}R$ for $0\le j\le n-1$. 
To build in the associativity of the structure maps $\rho_{i,j}\co\tilde{\tau}_{\ge i}^{\le n}R \wedge \tilde{\tau}_{\ge j}^{\le n}R \to  \tilde{{\tau}}_{\ge n}^{\le n} R$ for $0<i,j<n$ and $i+j\ge n$, we will encode all of the necessary associativity diagrams into one pushout, where each object in the pushout is a colimit of a truncated directed cube diagram. We will then show that the map from this pushout to a certain given object is nulhomotopic and therefore factors through a single contractible object. We then define $\tilde{\tau}_{\ge n}^{\le n}R$ as the pullback of a diagram involving this contractible object. This is arranged so that we only need to make one choice of contractible object at each stage of the induction. We also work in the category of symmetric $\tilde{\tau}_{\ge \bullet}^{\le n-1}R$-bimodules throughout this process in order to encode the remaining associativity, commutativity, compatibility, and unitality diagrams.

Note that each map 
\[ \tilde{\tau}_{\ge i}^{\le n-1}R  \wedge \tilde{\tau}_{\ge j}^{\le n-1}R \rightarrow \tau_{\ge m}^{\le n-1} R  \rightarrow \tau_{\ge n-1}^{\le n-1}(\gamma_{\le n-1}R) \]
is nulhomotopic for $i+j\ge n$ because $\pi_\ell( \tilde{\tau}_{\ge i}^{\le n-1}R  \wedge \tilde{\tau}_{\ge j}^{\le n-1}R)\cong 0$ for $\ell<n$ and $\pi_\ell(\tau_{\ge m}(\gamma_{\le n-1}R))\cong 0$ for $\ell>n-1$. 
We claim that the composite map from the Bousfield-Kan homotopy colimit (in the sense of our Appendix B) of the diagram
\begin{equation}\label{eg assoc}
\xymatrix{
\colim_{i\dot{+}j\dot{+}k\ge n} \tilde{\tau}_{\ge i}^{\le n-1}R \wedge \tilde{\tau}_{\ge j}^{\le n-1}R  \wedge \tilde{\tau}_{\ge k}^{\le n-1}R  \ar[r] \ar[d] & \colim_{i\dot{+}j\dot{+}k\ge n} \tilde{\tau}_{\ge i\dot{+}j}^{\le n-1}R \wedge \tilde{\tau}_{\ge k}^{\le n-1}R  \\
\colim_{i\dot{+}j\dot{+}k\ge n} \tilde{\tau}_{\ge i}^{\le n-1}R \wedge \tilde{\tau}_{\ge j\dot{+}k}^{\le n-1}R &    }
\end{equation}
to $\tilde{\tau}_{\ge m}^{\le n-1} R$, followed by the map  $\tilde{\tau}_{\ge m}^{\le n-1} R\rightarrow \tilde{\tau}_{\ge m}^{\le n-1}(\gamma_{\le n-1}R)$, is nulhomotopic as well. (Note that this will follow if the map from the diagram \eqref{eg assoc} to $\tilde{\tau}_{\ge n-1}^{\le n-1} R$ followed by the map  $\tilde{\tau}_{\ge n-1}^{\le n-1} R\rightarrow \tilde{\tau}_{\ge n-1}^{\le n-1}(\gamma_{\le n-1}R)$ is nulhomotopic since the composite of a nulhomotopic map with any other map is always nulhomotopic.)

To prove the claim above, we use the Bousfield-Kan spectral sequence of Theorem \ref{bkss existence thm}, which is discussed further in our self-contained Appendix~\ref{BK}:
\[ (R^s_E\lim_{d\in\mathcal{D}})\left( Z^{t}A(d)\right) \Rightarrow  Z^{s+t} \hocolim A .\] 
We will apply this spectral sequence when $\mathcal{D}$ is one of the small categories
\[\mathcal{P}_n^{3}=\{ (i,j,k)\in \mathbb{N}^{3} : i\dot{+}j\dot{+}k\ge n \text{ and }0<  i,j,k<n \}^{\op}\subseteq (\mathbb{N}^{3})^{\op},\]
 \[ \mathcal{P}_n^{2}=\{(\ell,m)\in \mathbb{N}^{2}:\ell\dot{+}m\ge n\text{ and } 0<\ell,m<n \}^{\op}\subseteq (\mathbb{N}^{2})^{\op},\] 
 or the pushout category $\mathcal{PO}$, defined in \eqref{def of PO}. We will take $Z$ to be $\tau_{\ge n-1}(\gamma_{\le n-1}R)$. 

Let $H\co \mathcal{P}_n^2\rightarrow \mathcal{C}$ be the functor given by $H(i,j) =  \tilde{\tau}_{\ge i}^{\le n-1}R \wedge \tilde{\tau}_{\ge j}^{\le n-1}R$. 
A colimit of a functor on $\mathcal{P}_n^{2}$ can be written as an iterated pushout, so we can prove a vanishing result for $[R_{E}^{s}\lim_{d\in(\mathcal{P}_n^{2})^{\op}}](Z^tH(d))$ for $s-t=0$ by proving an appropriate vanishing result for $R_E^s\lim_{d\in \mathcal{PO}^{\op}}(Z^t(\iota_*H)(d))$, where $\mathcal{PO}$ is the category indexing pushouts, as in~\eqref{def of PO}, and $\iota_*H$ is the restriction of $H$ along one of the inclusions $\iota\co\mathcal{PO}\hookrightarrow\mathcal{P}_n^{2}$. 
By Example~17.10 of \cite{dugco} and our Theorem~\ref{answer to question}, given a functor $\mathcal{F}:\mathcal{PO} \rightarrow \mathcal{C}$ and an object $Y$ of $\mathcal{C}$, the map $[\hocolim \mathcal{F}, Y]\rightarrow \lim [\mathcal{F},Y]$ is an isomorphism if $R^1_E\lim_{d\in\mathcal{PO}^{\op}} [\Sigma \mathcal{F}(d),Y] \cong [\Sigma \mathcal{F}(1^{\prime}),Y]/\left( [\Sigma \mathcal{F}(0),Y] + [\Sigma \mathcal{F}(1^{\prime}),Y]\right)$ vanishes. 

So we carry out an induction: let $\mathcal{I}_0: \mathcal{PO} \rightarrow \mathcal{C}$ be the functor given by 
\begin{align*}
 \mathcal{I}_0(1^{\prime}) &= \tilde{\tau}_{\ge 2}^{\leq n-1}R \wedge \tilde{\tau}_{\ge n-1}^{\leq n-1}R,\\
 \mathcal{I}_0(0) &= \tilde{\tau}_{\ge 1}^{\leq n-1}R \wedge \tilde{\tau}_{\ge n-1}^{\leq n-1}R,\\
 \mathcal{I}_0(1) &= \tilde{\tau}_{\ge 2}^{\leq n-1}R \wedge \tilde{\tau}_{\ge n-2}^{\leq n-1}R.
\end{align*}
Vanishing of $[\Sigma \mathcal{I}_0(1^{\prime}), \tilde{\tau}_{\ge n-1}(\gamma_{\le n-1}R)]$ and of $[ \mathcal{I}_0(d), \tilde{\tau}_{\ge n-1}(\gamma_{\le n-1}R)]$ for $d=0,1$ is standard (by properties of maps from sufficiently connective spectra to sufficiently coconnective spectra), so $[\hocolim \mathcal{I}_0, Z]$ vanishes.

That was the initial step. For the inductive step, let $\mathcal{I}_j: \mathcal{PO} \rightarrow \mathcal{C}$ be the functor given by 
\begin{align*}
 \mathcal{I}_j(1^{\prime}) &= \tilde{\tau}_{\ge j+2}^{\leq n-1}R \wedge \tilde{\tau}_{\ge n-j-1}^{\leq n-1}R,\\
 \mathcal{I}_j(0) &= \hocolim \mathcal{I}_{j-1},\\
 \mathcal{I}_j(1) &= \tilde{\tau}_{\ge j+2}^{\leq n-1}R \wedge \tilde{\tau}_{\ge n-j-2}^{\leq n-1}R.
\end{align*}
and suppose we have already shown that $[\hocolim\mathcal{I}_{j-1},Z]$ vanishes. Again, 
\[ [\Sigma \mathcal{I}_j(1^{\prime}), \tilde{\tau}_{\ge n-1}(\gamma_{\le n-1}R)] \cong [ \mathcal{I}_j(1), \tilde{\tau}_{\ge n-1}(\gamma_{\le n-1}R)]\cong 0\] 
due to standard properties of maps from sufficiently connective spectra to sufficiently coconnective spectra, so $[\hocolim \mathcal{I}_j,Z]$ vanishes. The case $j=n-2$ completes the induction, since the natural map 
\[ \hocolim \mathcal{I}_{n-2} \rightarrow \hocolim H\]  is an equivalence.

Therefore, $\hocolim H\rightarrow \tilde{\tau}_{\ge n-1}^{\le n-1}(\gamma_{\le n-1} R)$ is nulhomotopic, and hence the 
composite \[ \hocolim H\rightarrow \tilde{\tau}_{\ge n-1}^{\le n-1}(\gamma_{\le n-1} R)\rightarrow \tilde{\tau}_{\ge n-1}^{\le n-1}(\gamma_{\le n-1} R)\] is nulhomotopic. 

Now we claim that the diagram $A\co \mathcal{P}_n^{3}\rightarrow \Sp$ given by 
\[ (i,j,k)\mapsto \tilde{\tau}_{\ge i}^{\le n-1}R \wedge \tilde{\tau}_{\ge j}^{\le n-1}R \wedge \tilde{\tau}_{\ge k}^{\le n-1} R \] 
is actually Reedy cofibrant and hence, since $\mathcal{P}_n^{3}$ is a direct category (as in Remark~\ref{projective coincides with reedy}), it is projectively cofibrant. 
To see this, it is sufficient to check that each object in the diagram is cofibrant, each map is a cofibration, and for each inclusion of a square-shaped diagram the map from the pushout of the upper left horn to the terminal vertex is cofibration by an elementary check of Reedy's conditions \cite{reedy}. 
Each object in the diagram is cofibrant and each map in the diagram is a cofibration by definition. 
Also, one can easily check that for each inclusion of a square-shaped diagram the map from the pushout of the upper left horn to the terminal vertex is a cofibration by iterated use of the pushout product axiom (cf Lemma \ref{connectivity of truncated cube colims} where essentially the same result is proven in more detail). 
The fact that the canonical map from $\left| \sr(\mathcal{F})\right| \rightarrow \colim\mathcal{F}$ is a weak equivalence for projectively cofibrant $A$ is Theorem \ref{thm from gambino}, which is a consequence of Theorems 3.2 and 3.3 in \cite{MR2586997}. 

The colimit of the diagram $A\co \mathcal{P}_n^{3}\rightarrow \Sp$ described above 
can be written as an iterated colimit over directed cubes with terminal vertex removed. 
So, by an induction totally analogous to the induction we just carried out using $\mathcal{I}_0,\mathcal{I}_1, \dots$, 
the question of whether the map $\left|\sr(A) \right| \rightarrow \tilde{\tau}_{\ge m}(\gamma_{\le n-1}R)$ is nulhomotopic, as in Question~\ref{question on nulhomotopies}, reduces to the case of Question~\ref{question on nulhomotopies} for diagrams indexed by directed cubes with terminal vertex removed, which we handle in Lemma \ref{main lemma}.
Let $\mathcal{D}$ be a subcategory of $\mathcal{P}_n^{3}$ isomorphic to a directed cube with terminal vertex removed.
By Lemma \ref{main lemma}, it suffices to show that $[R_{E}^{s}\lim_{d\in\mathcal{D}}](Z^{-t}A(d)$ vanishes for $s-t=0$ and $s=0,1,2$ because $(R^s_E\lim_{d\in\mathcal{D}})\left( Z^{-t}A(d)\right)\cong 0$ for $s>2$. 
The condition
\begin{equation}\label{vanishing condition 3208} [\Sigma^{t} \tilde{\tau}_{\ge i}^{\le n-1}R \wedge \tilde{\tau}_{\ge j}^{\le n-1}R  \wedge \tilde{\tau}_{\ge k}^{\le n-1}R, \tilde{\tau}_{\ge m}^{\le n-1}(\gamma_{\le n-1}R)]\cong 0  \end{equation}
 for $t>-i-j-k+n-1$ is sufficient to show that $(R^s_E\lim_{d\in\mathcal{D}})\left( Z^{-t}A(d)\right)$ vanishes for $s-t=0,1$ and $s=0,1,2$ when $A$ is any of the truncated cubes in the iterative process and hence the maps
\[ \colim_{i\dot{+}j\dot{+}k\ge n-1} \tilde{\tau}_{\ge i}^{\le n-1}R \wedge \tilde{\tau}_{\ge j}^{\le n-1}R  \wedge \tilde{\tau}_{\ge k}^{\le n-1}R \to \tilde{\tau}_{\ge m}^{\le n-1} R\to \tilde{\tau}_{\ge m}^{\le n-1}(\gamma_{\le n-1}R)\]
are nulhomotopic for $0\le m\le n-1$. 
Condition~\ref{vanishing condition 3208} is indeed satisfied, by standard properties of maps from sufficiently connected spectra to sufficiently co-connected spectra.

Finally, we conclude that the map from the diagram \eqref{eg assoc} to $\tilde{\tau}_{\ge m}^{\le n-1} R$ followed by the map $\tilde{\tau}_{\ge m}^{\le n-1} R\to \tilde{\tau}_{\ge m}^{\le n-1}(\gamma_{\le n-1}R)$ is nulhomotopic, again using Example 17.10 in Dugger \cite{dugco}, along with the vanishing of 
\[ Z^{1}(\hocolim_{i\dot{+}j\dot{+}k\ge m+1} \tilde{\tau}_{\ge i}^{\le n-1}R \wedge \tilde{\tau}_{\ge j}^{\le n-1}R  \wedge \tilde{\tau}_{\ge k}^{\le n-1}R ),\] 
\[Z^{0}(\hocolim_{i\dot{+}j\dot{+}k\ge m+1} \tilde{\tau}_{\ge i}^{\le n-1}R \wedge \tilde{\tau}_{\ge j\dot{+}k}^{\le n-1}), \text{ and }\]
\[Z^{0}(\hocolim_{i\dot{+}j\dot{+}k\ge m+1} \tilde{\tau}_{\ge i\dot{+}j}^{\le n-1}R \wedge \tilde{\tau}_{\ge k}^{\le n-1}R).\] We may also consider diagram \eqref{eg assoc} as a diagram in symmetric $\tilde{\tau}_{\ge \bullet}^{\le n-1}R$-bimodules, and what we have shown is that the map of symmetric $\tilde{\tau}_{\ge \bullet}^{\le n-1}R$-bimodules from the homotopy colimit of \eqref{eg assoc} to $\tilde{\tau}_{\ge \bullet}^{\le n-1}R$ followed by the map $\tilde{\tau}_{\ge \bullet}^{\le n-1}R\rightarrow \tilde{\tau}_{\ge \bullet}^{\le n-1} (\gamma_{\le n-1} R)$ is nulhomotopic in the category of $\tilde{\tau}_{\ge \bullet}^{\le n-1}R$-modules. The point is that we can therefore factor this map through a contractible $\tilde{\tau}_{\ge \bullet}^{\le n-1}R$-module, which we denote $C(n)_{\bullet}$. 

The pullback of the diagram 
\begin{equation}\label{key pb}
\xymatrix{ 
& C(n)_{\bullet} \ar[d] \\
\tilde{\tau}_{\ge \bullet }^{\le n-1}(R) \ar[r] & \tilde{\tau}_{\ge \bullet}^{\le n-1}(\gamma_{\le n-1}R) 
}
\end{equation}
in the category of symmetric $\tilde{\tau}_{\ge \bullet }^{\le n-1}(R)$-bimodules is a functor $B_{\bullet}:J_{n-1}^{\op}\to\Sp$ along with structure maps 
\[ B_{\bullet}\otimes_{Day} \tilde{\tau}_{\ge \bullet }^{\le n-1}(R)  \rightarrow B_{\bullet} \text{ and }\]
\[  \tilde{\tau}_{\ge \bullet }^{\le n-1}(R) \otimes_{Day} B_{\bullet} \rightarrow B_{\bullet}.\]
(For the sake of consistency with our notation $\tilde{\tau}_{\ge \bullet }^{\le n-1}(R)$, the subscripted bullet in $B_{\bullet}$ indicates that $B_{\bullet}$ is a finite sequence $B_{n-1}\rightarrow B_{n-2} \rightarrow \dots \rightarrow B_0$.)
We observe that $B_{n-1}\simeq B_i$ for all $0 \le i \le n-1$, and there is a map $B_{n-1}\rightarrow \tilde{\tau}_{\ge 0}^{\le n-1}R$ inducing an isomorphism on $\pi_k$ for $k\ge n$ and $\pi_k(B_{n-1})\cong 0$ for $k<n$.  We therefore define $\tilde{\tau}_{\ge n}^{\le n}R$ to be $B_{n-1}$.

Next we need to check that the structure maps $\rho_{i,j} :\tilde{\tau}_{\ge i}^{\le n }R \Smash \tilde{\tau}_{\ge j}^{\le n }R\to \tilde{\tau}_{\ge n}^{\le n }R$ for $i+j\ge n$ where $\tilde{\tau}_{\ge i}^{\le n }R$ is defined to be $\tilde{\tau}_{\ge i}^{\le n-1}R$ when $0\le i<n$, are unital, associative, and commutative. 
We use the symmetric $\tilde{\tau}_{\ge \bullet}^{\le n-1}R$ bimodule structure to produce associative and commutative structure maps $\rho_{i,n} :\tilde{\tau}_{\ge i}^{\le n-1}R \wedge \tilde{\tau}_{\ge n}^{\le n}R\rightarrow \tilde{\tau}_{\ge n}^{\le n}R$ and $\rho_{n,i}\co\tilde{\tau}_{\ge n}^{\le n-1}R \wedge \tilde{\tau}_{\ge i}^{\le n}R\rightarrow \tilde{\tau}_{\ge n}^{\le n}R$ for $0\le i<n$ as follows: We have the commutative diagram of natural transformations
\[ 
	\xymatrix{
		\tilde{\tau}_{\ge \bullet}^{\le n-1}R \otimes_{Day}  B \ar[r] \ar[d] & B \\
		B \otimes_{Day} \tilde{\tau}_{\ge \bullet}^{\le n-1}R  \ar[ur] &
	 }
\]
so by definition of Day convolution, there is, by evaluating at $n-1$, a commutative diagram 
\[ 
	\xymatrix{
		\colim_{i\dot{+}j\ge n-1} \tilde{\tau}_{\ge i}^{\le n-1}R \Smash B_j \ar[r] \ar[d] & B_{n-1} \\
		\colim_{\ell\dot{+}m\ge n-1} B_{\ell} \Smash \tilde{\tau}_{\ge m}^{\le n-1}R  \ar[ur] &
	 }
\]
in $\Sp$. Since the map of colimits can equivalently be defined as the colimit of factor swap maps $\chi\co\tilde{\tau}_{\ge i}^{\le n-1}R \Smash  B_j\rightarrow B_j\Smash  \tilde{\tau}_{\ge i}^{\le n-1}R$, the commutativity of the diagram above implies the diagram 
\[ 
	\xymatrix{
		\tilde{\tau}_{\ge i}^{\le n-1}R \Smash B_{n-1} \ar[r] \ar[d] & B_{n-1} \\
		B_{n-1} \Smash \tilde{\tau}_{\ge i}^{\le n-1}R  \ar[ur] &
	 }
\]
commutes for all $i$ such that $0\le i<n$. 

We prove that $\rho_{n,i}$ and $\rho_{i,n}$ satisfy associativity with respect to the maps $\rho_{i,j}$ for $i,j<n$ by the same method. We will just describe one example of this, since the remaining examples are proven in the same way. By definition of $B$, there is commutative diagram of natural transformations
\[ 
	\xymatrix{
		\tilde{\tau}_{\ge \bullet}^{\le n-1}R \otimes_{Day} \tilde{\tau}_{\ge \bullet}^{\le n-1}R \otimes_{Day}  B  \ar[r] \ar[d] &\tilde{\tau}_{\ge \bullet}^{\le n-1}R \otimes_{Day}  B \ar[d]  \\
		\tilde{\tau}_{\ge \bullet}^{\le n-1}R \otimes_{Day} B   \ar[r] & B 
	 }
\]
so by definition of Day convolution, there is, by evaluating at $n-1$, a commutative diagram 
\[ 
	\xymatrix{
		\colim_{i\dot{+}j\dot{+}k\ge n-1} \tilde{\tau}_{\ge i}^{\le n-1}R \Smash \tilde{\tau}_{\ge j}^{\le n-1}R \Smash B_k  \ar[r] \ar[d] &\colim_{i\dot{+}j\dot{+}k\ge n} \tilde{\tau}_{\ge i\dot{+}j}^{\le n-1}R \Smash B_k \ar[d] \\
		\colim_{i\dot{+}j\dot{+}k\ge n-1} \tilde{\tau}_{\ge i}^{\le n-1}R \Smash B_{j\dot{+}k}   \ar[r] & B_{n-1}.
	 }
\]
Therefore, in particular, there is a commutative diagram 
\[ 
	\xymatrix{
		 \tilde{\tau}_{\ge i}^{\le n-1}R \Smash \tilde{\tau}_{\ge j}^{\le n-1}R \Smash B_{n-1}  \ar[r] \ar[d] & \tilde{\tau}_{\ge i\dot{+}j}^{\le n-1}R \Smash B_{n-1} \ar[d] \\
		 \tilde{\tau}_{\ge i}^{\le n-1}R \Smash B_{n-1}   \ar[r] & B_{n-1}.
	 }
\]
which proves the desired associativity diagram for $B_{n-1}$. 

To define $\rho_{n,n}:\tilde{\tau}_{\ge n}^{\le n}R \Smash \tilde{\tau}_{\ge n}^{\le n}R \rightarrow \tilde{\tau}_{\ge n}^{\le n}R$, we use the commutative diagram of symmetric $\tilde{\tau}_{\ge \bullet}^{\le n-1}R$-bimodules 
\begin{equation}\label{rho n n}
	\xymatrix{ 
		B\otimes_{Day} B \ar[r] \ar[d]  &  B\otimes \tilde{\tau}_{\ge \bullet}^{\le n-1}R  \ar[d] \\
		\tilde{\tau}_{\ge \bullet}^{\le n-1}R \otimes_{Day} B \ar[r] & B  
	}
\end{equation}
which implies that there is a map
\[ \colim_{i\dot{+}j\ge n-1} B_i\Smash B_j \rightarrow  B_{n-1} \]
by evaluating at $n-1$ and therefore there is a map 
\[ B_{n-1}\Smash B_{n-1}\rightarrow B_{n-1} \]
which we then define to be $\rho_{n,n}$. This map is commutative by definition and it is associative with respect to all the other maps $\rho_{i,j}$ by the same argument used for associativity above because the diagram \ref{rho n n} is a diagram of symmetric $\tilde{\tau}_{\ge \bullet}^{\le n-1}R$-bimodules. 

The maps $\rho_{i,j}\co \tilde{\tau}_{\ge i}^{\le n-1}R \Smash \tilde{\tau}_{\ge j}^{\le n-1}R \rightarrow \tilde{\tau}_{\ge n}^{\le n}R$ for $i+j\ge n$ are compatible with all the maps $\rho_{i,j}$ for $i+j<n$ by definition of $\tilde{\tau}_{\ge n}^{\le n}R$ as the pullback in diagram \eqref{key pb}. (Note that the forgetful functor from symmetric $\tilde{\tau}_{\ge \bullet}^{\le n-1}R$-modules to $\Sp^{J_{n-1}^{\op}}$ composed with the $n$-th evaluation functor to $\Sp$ is a right adjoint and therefore it preserves limits.)

The maps $\rho_{i,n}\co \tilde{\tau}_{\ge i}^{\le n-1}R \Smash \tilde{\tau}_{\ge n}^{\le n-1}R \rightarrow \tilde{\tau}_{\ge n}^{\le n}R$ and $\rho_{n,i}\co \tilde{\tau}_{\ge n}^{\le n-1}R \Smash \tilde{\tau}_{\ge i}^{\le n-1}R \rightarrow \tilde{\tau}_{\ge n}^{\le n}R$ are unital by definition of $B$ as a symmetric $\tilde{\tau}_{\ge \bullet}^{\le n-1}R$ bimodule by the same kind of argument as given above for commutativity and associativity. We will say why the remaining maps 
$\rho_{i,j}\co \tilde{\tau}_{\ge i}^{\le n-1}R \Smash \tilde{\tau}_{\ge j}^{\le n-1}R \rightarrow \tilde{\tau}_{\ge n}^{\le n}R$ for $i+j\ge n$ and $i,j<n$ are unital later. 

We now prove associativity of the maps $\rho_{i,j}:\tilde{\tau}_{\ge i}^{\le n-1}R \Smash \tilde{\tau}_{\ge j}^{\le n-1}R \rightarrow \tilde{\tau}_n^{\le n-1}R$ where $0<i,j<n$ and $i+j\ge n$. The diagrams 
\[ 
\xymatrix{ 
\tilde{\tau}_{\ge i}^{\le n}R \wedge \tilde{\tau}_{\ge j}^{\le n}R \wedge \tilde{\tau}_{\ge k}^{\le n}R \ar[r] \ar[d] & \tilde{\tau}_{\ge i\dot{+}j}^{\le n}R \wedge \tilde{\tau}_{\ge k}^{\le n}R  \ar[d]\\
 \tilde{\tau}_{\ge i}^{\le n}R \wedge \tilde{\tau}_{\ge j\dot{+}k}^{\le n}R \ar[r] & \tilde{\tau}_{\ge n}^{\le n}R }\]
commute for $i,j,k<n$ and $i+j+k\ge n$ by the definition of $\tilde{\tau}_{\ge n}^{\le n}R$ as the pullback \eqref{key pb}. 
The maps $\rho_{i,j}: \tilde{\tau}_{\ge i}^{\le n}R \wedge \tilde{\tau}_{\ge j}^{\le n}R \rightarrow  \tilde{\tau}_{\ge n}^{\le n}R$ are also commutative because the diagrams 
\[ 
	\xymatrix{
		 \tilde{\tau}_{\ge i}^{\le n-1}R \wedge \tilde{\tau}_{\ge j}^{\le n-1}R \ar[dr] \ar[r] & \tilde{\tau}_{\ge j}^{\le n-1}R \wedge \tilde{\tau}_{\ge i}^{\le n-1}R \ar[d] &\\
		  & \tilde{\tau}_{\ge n-1}^{\le n-1}R   \ar[r] &  \tilde{\tau}_{\ge n-1}^{\le n-1} (\gamma_{\le n-1} R) 
		   }
\]
commute for all $i$, $j$, and the diagrams
\[ 
	\xymatrix{
		 \tilde{\tau}_{\ge i}^{\le n-1}R \wedge \tilde{\tau}_{\ge j}^{\le n-1}R \ar[dr] \ar[r] & \tilde{\tau}_{\ge j}^{\le n-1}R \wedge \tilde{\tau}_{\ge i}^{\le n-1}R \ar[d] &\\
		  & C(n)_n   \ar[r] & \tilde{\tau}_{\ge n-1}^{\le n-1} \gamma_{\le n-1} R\\	
		   }
\]
commute for all $i,j$, so by the definition of $B$ as the pullback of the diagram \eqref{key pb}, the diagram
\[ 
	\xymatrix{
		 \tilde{\tau}_{\ge i}^{\le n-1}R \wedge \tilde{\tau}_{\ge j}^{\le n-1}R \ar[dr] \ar[r] & \tilde{\tau}_{\ge j}^{\le n-1}R \wedge \tilde{\tau}_{\ge i}^{\le n-1}R \ar[d] \\
		&  \tilde{\tau}_{\ge n}^{\le n} R 
		}
\]
commutes for $i,j<n$ and $i+j\ge n$. 

The unitality of the maps $\rho_{i,j}$ for $i,j<n$ and $i+j\ge n$ follows by the induction hypothesis and a similar argument to the one above. 
 
We have therefore produced an object in $\Comm \Sp^{J_n^{\op}}$. We then cofibrantly replace this object in $\Comm \Sp^{J_n^{\op}}$ to produce an object $\tau_{\ge \bullet}^{\le n}R$. By induction, we can therefore produce an object in $\Comm \Sp^{\mathcal{N}^{\op}}$ and then cofibrantly replace it to produce (by Remark~\ref{remark on day conv}) a cofibrant decreasingly filtered commutative monoid in $\Sp$, denoted $\tau_{\ge \bullet}R$, as desired. 
\end{proof} 

\begin{remark}\label{gabes remark on cofibrancy of degree 0 quotient}
Note that what we really produce in the proof above is a cofibrant commutative monoid in the category of functors $\Comm \Sp^{\mathcal{N}^{\op}}$, or by the equivalence of categories discussed in Remark \ref{equiv of cats Day}, a lax symmetric monoidal functor in $\Sp^{\mathcal{N}^{\op}}$. Due to the triviality of our $\Sp$-enrichment on $\mathcal{N}$, this is the same data as a lax symmetric monoidal functor in $\Sp^{\mathbb{N}^{\op}}$, and the cofibrancy condition produces a cofibrant decreasingly filtered commutative monoid in $\Sp$ as discussed in Definition \ref{def of dec filt comm mon}. 

We need to check that $\tau_{\ge \bullet}R$ satisfies the ``Cofibrancy of degree 0 quotient'' condition from Definition \ref{def of dec filt comm mon}. In order to see that the map $S\to \tau_{\ge 0}R/\tau_{\ge 1}R$ is a level-wise cofibration we observe, first of all, that the construction of the map $R\to H\pi_0R$ is a cofibration because it is formed by attaching $E_{\infty}$-cells to kill higher homotopy. This is proven in \cite[IV.3.1]{MR1417719} in the associative setting, but the same proof works in the commutative setting. We then observe that $\tau_{\ge 0}R$ is constructed so that the composite $S\to \tau_{\ge 0}R\to R\to H\pi_0R$ is a cofibration of commutative monoids in $\Sp$, hence also a level-wise cofibration in $\Sp$. 
The map $\tau_{\ge 0}R\to H\pi_0R$ factors through the projection $\tau_{\ge 0}R\to \tau_{\ge 0}R/\tau_{\ge 1}R$, so since level-wise cofibrations are retractile, the map $S\rightarrow \tau_{\ge 0}R/\tau_{\ge 1}R$ is also a level-wise cofibration. This proves the ``Cofibrancy of degree 0 quotient'' condition, so $\tau_{\ge \bullet}R$ is a decreasingly filtered commutative monoid in symmetric spectra.
\end{remark}

\begin{example} ~\label{exm J} Assume a prime $p>2$ is fixed. Let $j$ be a cofibrant replacement in $\Comm\Sp$, for the commutative ring spectrum $\hat{K}(\mathbb{F}_q)_p$, where $q$ is a topological generator of $\mathbb{Z}_p^{\times}$. Then by Theorem~\ref{post filt}, we produce a cofibrant decreasingly filtered commutative monoid in $\Sp$, denoted $j_{\ge \bullet}$. The associated graded commutative monoid $E_0^*j_{\ge \bullet}$ is 
\[ H\pi_0j\vee\Sigma^{2p-3}H\pi_{2p-3}j\vee \Sigma^{4p-5}H\pi_{4p-5}j\vee ... \] 
after forgetting the commutative monoid structure. We therefore denote the associated graded commutative monoid object $H\pi_*j$. There is an isomorphism of graded rings $\pi_*(H\pi_*j)\cong \pi_*(j)$, but as we have seen $H\pi_*j$ is a generalized Eilenberg-Mac Lane spectrum so by taking the associated graded of this filtration of $j$ we have effectively killed off all the $k$-invariants of $j$. Note that by $H\pi_n j$ we mean an explicit model for the Eilenberg-Mac Lane spectrum constructed as the cofiber of a cofibration $j_{\ge n+1}\rightarrow j_{\ge n}$. This example is used to compute $V(1)_*THH(j)$ in a paper by the first author \cite{161200548}. 
\end{example}
\begin{example} \label{exm 2}
Let $R$ be a commutative ring spectrum with homotopy groups $\pi_k(R)\cong \hat{\mathbb{Z}}_p$ for $k=0,n$ and $\pi_k(R)\cong 0$ otherwise. Then  one can build 
\[ 0 \longrightarrow \Sigma^n H\hat{\mathbb{Z}}_p \longrightarrow R  \]
as a cofibrant decreasingly filtered commutative ring spectrum using Theorem~\ref{post filt}, were $\Sigma^n H\hat{\mathbb{Z}}_p$ is again an explicit model for the Eilenberg-Mac Lane spectrum constructed in Theorem~\ref{post filt}. Since one can construct a Postnikov truncation of a commutative ring spectrum as a commutative ring spectrum \cite{MR1732625}, we can produce an example of this type by considering the truncation of the connective $p$-complete complex K-theory spectrum 
\[ \Sigma^2H\pi_2ku_p \longrightarrow ku_p^{\le 2} \longrightarrow H\pi_0ku_p\] 
where the map $ku_p^{\le 2} \longrightarrow H\pi_0ku_p$ is constructed as in \cite[Thm. 8.1]{MR1732625}. 
\end{example} 

\section{Applications.}
Let $S/p$ be the mod $p$ Moore spectrum and let $R$ be a connective commutative ring spectrum. We now present two calculations: first, we calculate $(S/p)_*THH(R)$ when $R$ has the property that $\pi_*(R)\cong \hat{\mathbb{Z}}_p[x]/x^2$ where $|x|>0$; second, we provide a bound on topological Hochschild homology of $R$ in terms of $THH(H\pi_*(R))$ and we give an explicit bound in the case $\pi_*(R)\cong \mathbb{Z}_{(p)}[x]$ where $|x|=2n$ for $n>0$. 
\subsection{Topological Hochschild homology of Postnikov truncations.} 
Let $R$ be a commutative ring spectrum with the property that $\pi_*(R)\cong \hat{\mathbb{Z}}_p[x]/x^2$ with $|x|>0$. We will consider the THH-May spectral sequence 
\[ (S/p)_*(THH(H\hat{\mathbb{Z}}_p \ltimes \Sigma^{n} H\hat{\mathbb{Z}}_p)) \Rightarrow (S/p)_*(THH(R)) \] 
produced using the short filtration of a commutative ring spectrum $R$ given in Example \ref{exm 2}. First, we compute the input of the $S/p$-THH-May spectral sequence for this example. 
\begin{prop} Let $p$ be an odd prime, then 
\[ (S/p)_*(THH(H\hat{\mathbb{Z}}_p\ltimes \Sigma^{n} H\hat{\mathbb{Z}}_p)) \cong E(\lambda_1)\otimes_{\mathbb{F}_p} P(\mu_1)\otimes_{\mathbb{F}_p} HH_*(E(x)) \]
where $|x|=n$. The grading of $HH_*(E(x))$ is given by the sum of the internal and homological gradings. 
\end{prop} 
\begin{proof} 
Due to B\"okstedt ~\cite{bok}, there is an isomorphism
\[ \pi_*(S/p\Smash THH(H\hat{\mathbb{Z}}_p))\cong E(\lambda_1)\otimes_{\mathbb{F}_p} P(\mu_1). \]  
Let $S\ltimes \Sigma^{n}S$ be the trivial split square-zero extension of $S$ by $\Sigma^nS$. Then $H\mathbb{Z}$ and $S\ltimes \Sigma^{n}S$ are commutative $S$-algebras and there is an equivalence of commutative $S$-algebras $H\hat{\mathbb{Z}}_p \ltimes \Sigma^{n} H\hat{\mathbb{Z}}_p\simeq H\hat{\mathbb{Z}}_p \Smash \left( S\ltimes \Sigma^{n}S\right)$. Since the functor $S^1_{\bullet}\otimes(-)$ commutes with coproducts in $\Comm\Sp$ by \cite{MR1473888}, there are equivalences
\[ \begin{array}{rcl} THH(H\hat{\mathbb{Z}}_p\ltimes \Sigma^{n} H\hat{\mathbb{Z}}_p)) &\simeq &   THH(H\hat{\mathbb{Z}}_p \Smash(S\ltimes \Sigma^{n}S)) \\
& \simeq &   THH(H\hat{\mathbb{Z}}_p )\Smash THH(S\ltimes \Sigma^n S )  \\ \end{array}  \] 
of commutative ring spectra. 
Since $S/p\Smash H\hat{\mathbb{Z}}_p\simeq H\mathbb{F}_p $
 and the spectrum $THH(H\hat{\mathbb{Z}}_p )$ is a $H\hat{\mathbb{Z}}_p$-algebra, the spectrum $S/p\Smash THH(H\hat{\mathbb{Z}}_p)$ naturally has the structure of a $H\mathbb{F}_p$-module. Hence, there are isomorphisms
\[ \begin{array}{rc} \pi_*(S/p\Smash THH(H\hat{\mathbb{Z}}_p )\Smash THH(S\ltimes \Sigma^nS )) &\cong  \\
\pi_*\left(\left(S/p\Smash THH(H\hat{\mathbb{Z}}_p)\right)\Smash_{H\mathbb{F}_p} \left(H\mathbb{F}_p\Smash THH(S\ltimes \Sigma^nS )\right)\right)  &\cong \\ 
 \pi_*(S/p\Smash THH(H\hat{\mathbb{Z}}_p ))\otimes_{\mathbb{F}_p} {H\mathbb{F}_p}_*(THH(S\ltimes \Sigma^{n}S)). &\\ \end{array} \]  
Now, we apply the B\"okstedt spectral sequence
\[ HH_*( {H\mathbb{F}_p}_*(S\ltimes \Sigma^n S) ) \Rightarrow H_*(THH(S\ltimes \Sigma^nS ); \mathbb{F}_p) \]
whose input is isomorphic to $HH_*(E(x)).$ If $|x|$ is odd, then $HH_*(E(x))\cong E(x)\otimes_{\mathbb{F}_p} \Gamma(\sigma x)$, which follows from Theorem 6.1 in Chapter 10 of \cite{MR1731415} and the standard fact that $\Tor_*^{E(x)}(\mathbb{F}_p,\mathbb{F}_p)\cong \Gamma(\sigma x)$ (see \cite{MR1209233} for details). If $|x|$ is even, then one easily computes
\[ HH_n(E(x)) \cong \left \{ \begin{array}{ll} E(x)  & *=0  \\  \Sigma^{|x|(2i-1)} \mathbb{F}_p\{1\} & n=2i-1 \\ \Sigma^{|x|(2i+1)} \mathbb{F}_p\{x\}&  n=2i \end{array} \right.\]  
for $i\ge 1$. There is an isomorphism of bigraded rings
\[ HH_{*,*}(E(x))\cong E(x)[x_i,y_j\co i\ge 1, j\ge 0 ]/\sim\]
where the degrees are given by $|x_i|=(2i,2|x|i+|x|)$ and $|y_j|=(2j+1,2j|x|+|x|)$, and the equivalence relation is the one that makes all products zero. (See \cite[Prop. 3.3]{MR2183525} for the more general calculation of $HH_*(\mathbb{F}_p[x]/x^h)$ when $|x|=2n$, $n>0$ and $(p,h)=1$.)
The representatives in the cyclic bar complex for $x_i$ and $y_j$ are $x^{\otimes 2i+1}$ and $1\otimes x^{\otimes 2j+1}$ respectively. Whether $|x|$ is even or odd, the B\"okstedt spectral sequence collapses, with no hidden multiplicative extensions, for bidegree reasons. \end{proof} 
\begin{corollary} (Rigidity of $THH$ mod $p$ for Postnikov truncations ) \label{rigid square-zero}
Let $R$ be a connective $E_{\infty}$-ring spectrum with $\pi_*(R)\cong \hat{\mathbb{Z}}_p[x]/x^2$, $\pi_i(R)\cong 0$ for $i\neq0,k$. Suppose that 
\[ p \not \equiv k+1 \text{ mod } 2k+1. \]
Then the graded abelian group $\pi_*(S/p\wedge THH(R))$ depends only on $\pi_*(R^{\le 2k})$; i.e only on $p$ and $k$. 
\end{corollary} 

\begin{proof} The THH-May spectral sequence 
\[ (S/p)_{*,*}(THH(H\hat{\mathbb{Z}}_p\ltimes \Sigma^{2k} H\hat{\mathbb{Z}}_p)) \Rightarrow (S/p)_*(THH(R)) \]
collapses since there are no possible differentials for bidegree reasons under the assumptions on $k$ with respect to $p$. 
\end{proof} 
 \begin{remark} 
 Corollary~\ref{rigid square-zero} can be considered a rigidity theorem in the sense that $S/p\wedge THH$ does not see the first Postnikov $k$-invariant in the cases given by the congruences above. 
  \end{remark} 
\begin{corollary} Let $p$ be a prime such that $p\not \equiv 2 \text{ mod }3$, then 
\begin{equation}\label{iso 230845} \pi_*(S/p\Smash THH(ku_p^{\le 2}))\cong E(\lambda_1)\otimes_{\mathbb{F}_p} P(\mu_1)\otimes_{\mathbb{F}_p} HH_*(E(x)) \end{equation}
up to multiplicative extensions in the THH-May $E^{\infty}$-term (so, in particular, \eqref{iso 230845} is an isomorphism of graded abelian groups), where $|x|=2$ and the degree of $HH_*(E(x))$ in $\pi_*$ is given by the sum of the internal and homological degree.  
\end{corollary} 
\subsection{Upper bounds on the size of $THH$.}

Many explicit computations are possible using the $THH$-May spectral sequence; for example, the first author's computations of topological Hochschild homology of the algebraic $K$-theory of finite fields, in~\cite{161200548}. 
These computations are sufficiently lengthy that they merit their own separate paper. 

In the present paper, in lieu of explicit computations using the $THH$-May spectral sequence, we point out that the mere existence of the $THH$-May spectral sequence implies an upper bound on the size of the topological Hochschild homology groups of a ring spectrum: namely, if $R$ is a graded-commutative ring and $X_{\bullet}$ is a simplicial finite set and $E_*$ is a generalized homology theory, then for any $E_{\infty}$-ring spectrum $A$ such that $\pi_*(A) \cong R$, $E_*(X_{\bullet}\otimes A)$ is a subquotient of $E_*(X_{\bullet} \otimes HR)$. Here $HR$ is the generalized Eilenberg-Mac Lane spectrum of the graded ring $R$.

In particular:
\begin{theorem} \label{upper bound thm}
For all integers $n$ and all connective $E_{\infty}$-ring spectra $A$,
the cardinality of $THH_n(A)$ is always less than or equal to the cardinality of $THH_n(H\pi_*(A))$. 
\end{theorem}

Below are more details in a more restricted class of examples, namely, the $E_{\infty}$ ring spectra $A$ such that $\pi_*(A) \cong \mathbb{Z}_{(p)}[x]$.

\begin{definition}\label{partial order on power series}
We put a partial ordering on power series with integer coefficients as follows:
given $f,g\in \mathbb{Z}[[t]]$, we write $f\leq g$ if and only if, for all nonnegative integers $n$, the coefficient of $t^n$ in $f$ is less than or equal to the coefficient of $t^n$ in $g$.
\end{definition}

\begin{definition}\label{def of finite-type}
Let $A$ be a graded ring. 
We will say that a graded $A$-module $M$ is {\em finite-type and free} if there exists a function $c: \mathbb{Z} \rightarrow \mathbb{N}$
and an isomorphism of graded $A$-modules 
 \[ \coprod_{n\in \mathbb{Z}} (\Sigma^n A)^{\oplus c(n)} \stackrel{\cong}{\longrightarrow} M.\]
We will say that $M$ is {\em finite-type} 
if there exists an exact sequence of graded $A$-modules of the form
$F_1 \rightarrow F_0 \rightarrow M \rightarrow 0$
with $F_0,F_1$ both finite-type and free.
\end{definition}

\begin{lemma}\label{finite type closure lemma}
Let $A$ be a Noetherian connective commutative graded ring. Then the collection of bounded-below finite-type graded $A$-modules is closed under taking kernels, cokernels, extensions, 
and tensor products over $A$. Consequently, every bounded-below finite-type graded $A$-module admits a resolution by bounded-below finite-type free graded $A$-modules, and if $M,N$ are bounded-below finite-type graded $A$-modules, then so is $\Tor^A_{n,*}(M,N)$ for each nonnegative integer $n$, and furthermore, $\Tor^A_{s,t}(M,N)$ vanishes for all $s>t$.
\end{lemma}
\begin{proof}
This is a bit of elementary algebra and we leave the proof as an exercise.
\end{proof}

Lemma~\ref{finite generation of THH lemma} is surely not a new result:
\begin{lemma}\label{finite generation of THH lemma}
Suppose that $A$ is a connective $E_{\infty}$-ring spectrum such that the ring $\pi_*(A)$ is commutative and Noetherian, and suppose that the graded $\pi_*(A)$-module $\pi_*(A\Smash A)$ is finite-type. 
Suppose that $X_{\bullet}$ is a simplicial finite set. Then $\pi_n(X_{\bullet}\otimes A)$ is a finitely generated $\pi_0(A)$-module for all $n$. 
\end{lemma}
\begin{proof} 
First, a quick induction: if we have already shown that the graded $\pi_*(A)$-module $\pi_*(A^{\Smash m})$ is connective and finite-type, then by Lemma~\ref{finite type closure lemma}, the input for the 
the K\"{u}nneth spectral sequence 
\begin{align}\label{kss 13049} E^2_{s,t} \cong \Tor^{\pi_*(A)}_{s,t}(\pi_*(A\Smash A), \pi_*(A^{\Smash m})) &\Rightarrow \pi_{s+t}(A^{\Smash m+1}) \\ 
\nonumber d^r: E^r_{s,t} & \rightarrow E^r_{s-r,t+r-1}
\end{align}
is first-quadrant and hence strongly convergent, and consists of a finite-type graded $\pi_*(A)$-module on each $s$-line. The differentials in spectral sequence~\eqref{kss 13049} are $\pi_*(A)$-linear (see Theorem~IV.4.1 of~\cite{MR1417719}, which is stated in terms of commutative $S$-algebras, but the linearity of the spectral sequence differentials is formal and works for any model of commutative ring spectra), so by Lemma~\ref{finite type closure lemma}, for all integers $r\geq 2$ we have that each $s$-line in the $E^r$-page of spectral sequence~\eqref{kss 13049} is a finite-type graded $\pi_*(A)$-module. Any differential of length $>s$ supported on the $s$-line in spectral sequence~\ref{kss 13049} must be zero, but it is not impossible that the $s$-line could be hit by differentials of arbitrarily long length. Let $E^r_{s,*}$ denote the graded $\pi_*(A)$-module which is the the $s$-line in the $E^r$-page of spectral sequence~\ref{kss 13049}. Then, for each pair $(s,t)$, there exists some $N$ such that $E^r_{s,t} \cong E^{r+1}_{s,t}$ for all $r\geq N$. Consequently $E^{\infty}_{s,*}$ is the colimit of a sequence of graded $\pi_*(A)$-module surjections $E^{s+1}_{s,*} \rightarrow E^{s+2}_{s,*} \rightarrow\dots $ with the property that, for each integer $t$, the sequence of $\pi_0(A)$-modules 
$E^{s+1}_{s,t} \rightarrow E^{s+2}_{s,t} \rightarrow \dots$ is eventually constant. So the graded $\pi_*(A)$-module $E^{\infty}_{s,*}$ is finite-type.

Now $\pi_*(A^{\Smash m+1})$ admits a filtration whose filtration quotients are the 
rows in the $E^{\infty}$-page of~\eqref{kss 13049}. For any fixed choice of integer $N$, $\pi_*(A^{\Smash m+1})$ agrees in grading degrees $\leq N$ with a graded $\pi_*(A)$-module given by finitely many of the filtration quotients (eg the first $N+1$ rows in the $E^{\infty}$-page), due to the vanishing property in Lemma~\ref{finite type closure lemma}. So by Lemma~\ref{finite type closure lemma}, the graded $\pi_*(A)$-module $\pi_*(A^{\Smash m+1})$ is finite-type (and clearly connective), and we are ready to return to the inductive step.

So $\pi_*(A^{\Smash m})$ is a finite-type connective graded $\pi_*(A)$-module for each $m$. In particular, $\pi_n(A^{\Smash m})$ is a finitely generated $\pi_0(A)$-module for each $n$.
Consequently in the Bousfield-Kan-type spectral sequence
\begin{align*} 
 E^1_{s,t} \cong \pi_t(A^{\Smash \#(X_s)}) 
  & \Rightarrow \pi_{s+t}(X_{\bullet}\otimes A) \\
 d^r\co E^r_{s,t} 
  & \rightarrow E^r_{s-r,t+r-1}
\end{align*}
obtained by applying $\pi_*$ to the simplicial ring spectrum $X_{\bullet}\tilde{\otimes} A$ (here we are using the pretensor product, of Definition~\ref{def of tensoring}), 
each bidegree is a finitely generated $\pi_0(A)$-module, and the spectral sequence is half-plane with exiting differentials, hence also strongly convergent by Theorem~6.1 of~\cite{MR1718076}. 
Connectivity of each tensor power $A^{\Smash \#(X_s)}$ implies that only finitely many bidegrees in the $E^{\infty}$-page contribute to each total degree in $\pi_*(X_{\bullet}\otimes A)$.
Consequently $\pi_n(X_{\bullet}\otimes A)$ is a finitely generated $\pi_0(A)$-module for each integer $n$.
\end{proof}

\begin{theorem}\label{polynomial case 120}
Let $n$ be a positive integer, $p$ a prime number, and 
let $E$ be a $E_{\infty}$-ring spectrum such that either $\pi_*(E) \cong \hat{\mathbb{Z}}_p[x]$ or $\pi_*(E) \cong \mathbb{Z}_{(p)}[x]$, with $x$ in grading degree $2n$.
Then the Poincar\'{e} series of the mod $p$ topological Hochschild homology $(S/p)_*(THH(E))$ satisfies the inequality
\[ \sum_{i\geq 0} \left( \dim_{\mathbb{F}_p} (S/p)_*(THH(E))\right) t^i \leq
 \frac{(1 + (2p-1)t)(1 + (2n+1)t)}{(1 - 2nt)(1 - 2pt)}.\]
\end{theorem}
\begin{proof}
It is a classical computation of B\"{o}kstedt (see \cite{bok}) that \[(S/p)_*(THH(H\hat{\mathbb{Z}}_p)) \cong (S/p)_*(THH(H\mathbb{Z}_{(p)})) \cong E(\lambda_1)\otimes_{\mathbb{F}_p} P(\mu_1),\] with 
$\lambda_1$ and $\mu_1$ in grading degrees $2p-1$ and $2p$ respectively. 

Now we use the splitting theorem of Schw\"anzl, Vogt, and Waldhausen, Lemma 3.1 
of \cite{MR1783629}: if $K$ is a commutative ring, 
and $W$ is a $q$-cofibrant $S$-algebra (ie, up to equivalence, an $A_{\infty}$-ring spectrum), then there exists a weak equivalence of $S$-modules (not necessarily a weak equivalence of $S$-algebras!):
\[ THH(W\Smash HK)\simeq THH(W)\Smash THH(HK) \simeq (THH(W)\Smash HK)\Smash_{HK} THH(HK). \]

In our case, $W$ is the free $A_{\infty}$-algebra on a single $2n$-cell, and $K = \hat{\mathbb{Z}}_p$ or $\mathbb{Z}_{(p)}$. Hence 
$THH(W)\Smash HK$
satisfies 
\[ (S/p)_*(THH(W)\Smash HK) \cong (H\mathbb{F}_p)_*(THH(W)) \cong P(x)\otimes_{\mathbb{F}_p} E(\sigma x),\]
by collapse of the B\"okstedt spectral sequence for bidegree reasons. Hence 
\[ (S/p)_*(THH(H\hat{\mathbb{Z}}_p[x])) \cong (S/p)_*(THH(H\mathbb{Z}_{(p)}[x])) \cong E(\lambda_1,\sigma x)\otimes_{\mathbb{F}_p} P(\mu_1,x), \] 
 as a graded $\mathbb{F}_p$-vector space (but not necessarily as $\mathbb{F}_p$-algebras!),
which has Poincar\'{e} series $\frac{(1 + (2p-1)t)(1 + (2n+1)t)}{(1 - 2nt)(1 - 2pt)}$.
\end{proof}

Now we give a few amusing consequences of Theorem~\ref{polynomial case 120}.
Recall that a spectrum is said to be {\em finite-type} if it is weakly equivalent to a CW-spectrum with finitely many cells in each dimension; if a spectrum $X$ is connective, then $X$ is finite-type if and only if the graded $\mathbb{Z}$-module $H_*(X; \mathbb{Z})$ is finite-type.
We will say that $X$ is {\em $p$-local finite-type} if $X$ is weakly equivalent to a $p$-local CW-spectrum with finitely many $S_{(p)}$-cells in each dimension. Again, if a $p$-local spectrum $X$ is connective, then $X$ is $p$-local finite-type if and only if the graded $\mathbb{Z}_{(p)}$-module $H_*(X; \mathbb{Z}_{(p)})$ is finite-type.
\begin{corollary} \label{main corollary 0329}
Let $n$ be a positive integer, $p$ a prime number, and 
let $E$ be a $p$-local finite-type $E_{\infty}$-ring spectrum such that $\pi_*(E) \cong \mathbb{Z}_{(p)}[x]$, with $x$ in grading degree $2n>0$.
\begin{itemize}
\item If $p$ does not divide $n$, then $THH_{2i}(E) \cong 0$ for all $i$ congruent to $-p$ modulo $n$ such that $i\leq pn-p-n$, and
$THH_{2i}(E) \cong 0$ for all $i$ congruent to $-n$ modulo $p$ such that $i\leq pn-p-n$.
In particular, $THH_{2(pn-p-n)}(E) \cong 0$.
\item If $p$ divides $n$, then $THH_{i}(E)\cong 0$, unless $i$ is congruent to $-1,0,$ or $1$ modulo $2p$.
\end{itemize}
\end{corollary}
\begin{proof}
Since we have assumed that $E$ is $p$-local finite-type, each homology group $H_s(E; \pi_*(E))$ is a finite-type graded $\pi_*(E)$-module, and Consequently each $s$-line in the Atiyah-Hirzebruch spectral sequence
\begin{align}\label{ahss 02398} E^2_{s,t} \cong H_s(E; \pi_t(E)) & \Rightarrow \pi_{s+t}(E\Smash E) \\
\nonumber d^r: E^r_{s,t} &\rightarrow E^r_{s+r,t-r+1} \end{align}
is a finite-type graded $\pi_*(E)$-module. Connectivity of $E$ ensures strong convergence of the spectral sequence, as in Theorem~12.2 of~\cite{MR1718076}. As the Atiyah-Hirzebruch spectral sequence is the spectral sequence obtained by applying $E$-homology to a CW-decomposition of $E$, an easy analysis of the spectral sequence of a tower of cofiber sequences shows that the spectral sequence differentials are graded $\pi_*(E)$-module morphisms. Consequently Lemma~\ref{finite type closure lemma} implies that each $s$-line in the $E^r$-page of spectral sequence~\eqref{ahss 02398} is a finite-type graded $\pi_*(E)$-module, for each $r\geq 2$. Now an argument exactly like that used in the proof of Lemma~\ref{finite generation of THH lemma} implies that each $s$-line in the $E^{\infty}$-page of~\eqref{ahss 02398} is also a finite-type graded $\pi_*(E)$-module, and that $\pi_*(E\Smash E)$ is a finite-type graded $\pi_*(E)$-module. So we can make use of Lemma~\ref{finite generation of THH lemma} later in this proof.

Now we split the proof into two cases: the case where $p \nmid n$ and the case where $p|n$. 

Case 1: If $p$ does not divide $n$, then the largest integer $i$ such that the graded polynomial algebra $P(\mu_1,x)$ is trivial in grading degree $2i$ is
$2(pn-p-n)$. (This is a standard exercise in elementary number theory. In schools in the United States it is often presented to students in a form like ``What is the largest integer $N$ such that you cannot make exactly $5N$ cents using only dimes and quarters?'') Triviality of $P(\mu_1,x)$ in grading degree $2(pn-p-n)$ also implies
triviality of $P(\mu_1,x)$ in grading degree $2(pn-p-n) - 2(p+n)$, hence the triviality of $E(\lambda_1,\sigma x)\otimes_{\mathbb{F}_p} P(\mu_1,x)$ in grading degree $2(pn-p-n)$, hence (multiplying by powers of $x$ or $\mu_1$) the triviality of
$E(\lambda_1,\sigma x)\otimes_{\mathbb{F}_p} P(\mu_1,x)$ in all grading degrees $\leq 2(pn-p-n)$ which are congruent to $-2p$ modulo $2n$ or congruent to $-2n$ modulo $2p$.

So $(S/p)_{2i}(THH(E))$ vanishes if $i \leq pn-p-n$ and $i\equiv -p$ modulo $n$ or $i \equiv -n$ modulo $p$.
The long exact sequence
\[ \dots \rightarrow (S/p)_{2i+1}(THH(E)) \rightarrow THH_{2i}(E) \stackrel{p}{\longrightarrow} THH_{2i}(E) \rightarrow (S/p)_{2i}(E) \rightarrow \dots \]
then implies that $THH_{2i}(E)$ is $p$-divisible. By Lemma~\ref{finite generation of THH lemma}, $THH_{2i}(E)$ is a finitely generated $\pi_0(E)$-module. 
Since $\pi_0(E)\cong \mathbb{Z}_{(p)}$ is a PID, one knows its finitely generated modules explicitly, and 
the only finitely generated $\mathbb{Z}_{(p)}$-module which is $p$-divisible 
is the trivial module.

Case 2: If $p$ divides $n$, then $E(\lambda_1,\sigma x)\otimes_{\mathbb{F}_p} P(\mu_1,x)$ is concentrated in grading degrees congruent to $-1,0$ and $1$ modulo $2p$. 
An argument exactly as in the previous part of this proof then shows that, if $i$ is not congruent to $-1,0,$ or $1$ modulo $2p$, then $THH_{i}(E)$ must be a $p$-divisible finitely generated 
$\mathbb{Z}_{(p)}$-module, hence is trivial.
\end{proof}
\appendix
\section{The THH-May spectral sequence with filtered coefficients.}\label{THH-May w coeff}
In this appendix, we briefly describe how to generalize the THH-May spectral sequence to include filtered coefficients; ie, given a pointed simplicial finite set $Y_{\bullet}$, a connective generalized homology theory, $E_{*}$, as in Definition \ref{def of gen hom thy}, a decreasingly filtered commutative monoid in $\mathcal{C}$, $I_{\bullet}$, and a decreasingly filtered $I_{\bullet}$-module in $\mathcal{C}$, $M_{\bullet}$ (see Definition \ref{filtred coeff}), there is a spectral sequence
\begin{equation}\label{THH-May ss coeff} E^1_{*,*}=E_{*,*}\left (Y_{\bullet}\otimes (E_0I_{\bullet};E_0M_{\bullet}) \right )\Rightarrow Y_{\bullet}\otimes (I_0;M_0) \end{equation}
where $Y_{\bullet}\otimes (I_0;M_0)$ is is defined in Definition \ref{loday coeff} and $E_0M_{\bullet}$ is defined in Remark \ref{assoc grd module}. 
\subsection{Filtered coefficients.}
Recall from Remark \ref{rem filt comm mon} and Section \ref{convenient model cat for filtered objects} that a cofibrant object in $\Comm \mathcal{C}^{\mathcal{N}^{\op}}$ (with the model structure created by the forgetful functor to $\mathcal{C}^{\mathcal{N}^{\op}}$, where $\mathcal{C}^{\mathcal{N}^{\op}}$ is equipped with the projective model structure) is a cofibrant decreasingly filtered commutative monoid in $\mathcal{C}$ in the sense of Definition \ref{def of dec filt comm mon}. 
\begin{definition}\label{filtred coeff}
Let $I_{\bullet}$ be a cofibrant object in $\Comm \mathcal{C}^{\mathcal{N}^{\op}}$. By a cofibrant decreasingly filtered $I_{\bullet}$-module we mean a cofibrant object in $\mathcal{C}^{\mathcal{N}^{\op}}$ with the structure of a $I_{\bullet}$-module (equivalently a symmetric $I_{\bullet}$-bimodule); ie, there are natural transformations
\[ \psi^{\ell}:I_{\bullet}\otimes_{Day}M_{\bullet}\rightarrow M_{\bullet} \text{ and } \]
\[ \psi^{r}:M_{\bullet}\otimes_{Day}I_{\bullet}\rightarrow M_{\bullet}  \]
satisfying the usual commutative diagrams for a module over a commutative monoid along with the commutativity of the diagram 
\[ 
\xymatrix{  I_{\bullet}\otimes_{Day}M_{\bullet} \ar[rr]^{\tau}  \ar[dr]^{\psi^{\ell}} & &I_{\bullet}\otimes_{Day}M_{\bullet}  \ar[dl]^{\psi^{r}}  \\
 & M_{\bullet} & }\]
where $\tau$ is the factor swap map that we have by the definition of $\mathcal{C}^{\mathcal{N}^{\op}}$ as a symmetric monoidal category with respect to $\otimes_{Day}$. 
\end{definition}
\begin{remark}\label{assoc grd module}
There is an associated graded symmetric $E_0^*I_{\bullet}$-bimodule, $E_0^*M_{\bullet}$, which can be defined in a similar way to $E_0^*I_{\bullet}$; ie as an object in $\mathcal{C}$, it is
\[ E_0^*M_{\bullet}=\bigvee_{i\ge 0} M_i /M_{i+1} \]
and to define it as a symmetric $E_0I_{\bullet}$-bimodule we define maps 
\[ I_i/I_{i+1} \Smash M_j/M_{j+1} \rightarrow M_{i+j}/M_{i+j+1} \]
using the structure maps $\psi^{\ell}$ and $\psi^{r}$ along with the structure maps $g_m:M_m\rightarrow M_{m-1}$ and the structure maps of $I_{\bullet}$ and then extend to a map 
\[ \left( \bigvee_{i\ge 0} I_i/I_{i+1}\right )\Smash \left (\bigvee_{i\ge 0}M_j/M_{j+1} \right) \rightarrow \bigvee_{i\ge 0} M_{i+j}/M_{i+j+1} \]
in the same way as in Definition \ref{assoc graded monoid}. 
\end{remark}
\subsection{Loday construction in the pointed setting.}
\begin{definition}\label{loday coeff}
For a finite pointed set $S$ write $*_S$ for basepoint of $S$ and $S'$ for $S-\{*_S\}$. Given a cofibrant commutative monoid $R$ in a model category $\mathcal{C}$ satisfying Running Assumption \ref{ra:1} define a functor 
\[ fSet_*\times R\text{-mod} \rightarrow R\text{-mod}\]
by sending $(S,N)$ to $ N\Smash \bigsmash_{s\in S} R$
and on morphisms by sending a based map  of finite sets $S\rightarrow T$ to the composite map 
\[ N\Smash \bigsmash_{s\in S} R \simeq  N\Smash \bigsmash_{s\in f^{-1}(*_T)} R\Smash \bigsmash_{s\in S-f^{-1}(*_T)} R \rightarrow X\Smash \bigsmash_{t\in T} R \]
defined as follows: the first map is the factor swap map and the second map is the smash product of the iterate of the (right) action map 
\[ N\Smash \bigsmash_{s\in f^{-1}(*_T)} R \rightarrow N \]
and the map 
\[ \bigsmash_{s\in S-f^{-1}(*_T)} R \rightarrow \bigsmash_{s\in S-f^{-1}(*_T)} R \]
as defined as in Definition \ref{def of tensoring}. 
This functor naturally extends to a functor 
\[(-)\tilde{\otimes}(R;-)\co sfSet_*\times R\text{-mod}  \rightarrow sR\text{-mod}\] 
by sending $(Y_{\bullet},N)$ to the simplicial $R$-module with $n$-simplices 
\[ N\Smash \bigsmash_{y\in Y_n'} R \]
and using the functoriality of the functor $fSet_*\times R\text{-mod} \rightarrow R\text{-mod}$ to define the face and degeneracy maps. We can therefore define
\[ Y_{\bullet}\otimes (R;N)=|Y_{\bullet}\otimes (R;N)|.\]
\end{definition}
\begin{remark}
This construction is sufficiently general that given a cofibrant object $I_{\bullet}$ in $\Comm \mathcal{C}^{\mathcal{N}^{\op}}$ and a cofibrant decreasingly filtered $I_{\bullet}$-module in $\mathcal{C}$, $M_{\bullet}$, we can define an object in $\Comm^{\mathcal{N}^{\op}}$
\[ Y_{\bullet}\otimes (I_{\bullet};M_{\bullet}).\] 
\end{remark} 
\subsection{The fundamental theorem of the May filtration with coefficients.}
Recall that the fundamental theorem of the May filtration may be described using the slogan ``higher order Hochschild homology commutes with passage to the associated graded commutative ring spectrum."
\begin{definition}
Let $\mathcal{C}$ satisfy Running Assumptions \ref{ra:1} and \ref{ra:2}. Given a pointed simplicial finite set $Y_{\bullet}$, a cofibrant object in $\Comm \mathcal{C}^{\mathcal{N}^{\op}}$, $I_{\bullet}$, and a cofibrant decreasingly filtered $I_{\bullet}$-module, $M_{\bullet}$, we define \emph{the May filtration of $M_0\Smash \bigsmash_{y\in Y_0} I_0$} to be the cofibrant decreasingly filtered $I_{\bullet}$-module in $\mathcal{C}$:
$ Y_{\bullet} \otimes (I_{\bullet} ; M_{\bullet}).$
\end{definition}
\begin{remark}
Note that Running Assumption \ref{ra:2} is needed to ensure that $Y_{\bullet} \otimes (I_{\bullet} ; M_{\bullet})$ is actually a \emph{cofibrant} object in $\mathcal{C}^{\mathcal{N}^{\op}}$. 
\end{remark}
\begin{theorem}{\bf (Fundamental theorem of the May filtration with coefficients.)}
There is an equivalence of $E_0^*I_{\bullet}$-modules 
\[ E_0^*\left (Y_{\bullet} \otimes( I_{\bullet} \otimes M_{\bullet})\right )\simeq Y_{\bullet} \otimes (E_0^*I_{\bullet},E_0^*M_{\bullet}). \]
\end{theorem}
\begin{proof}
To prove the theorem requires generalizing all of the definitions and lemmas from Section \ref{fund thm section} to the pointed setting. The proofs of each of these generalizations of the lemmas from Section \ref{fund thm section} follow in an evident way from the proofs that are already given. We therefore do not reprove them. The proof of the fundamental theorem of the May filtration also follows from the evident generalizations of the lemmas in Section \ref{fund thm section} in the same way as the proof of Theorem \ref{thm on fund theorem}. 
\end{proof}
\begin{remark}\label{THH-May coeff}
The construction of the spectral sequence is exactly the same and therefore we do not discuss it here. To prove convergence and strong convergence of the spectral sequence we must prove the lemmas and theorems of Section \ref{construction of ss} in the pointed setting, but these generalizations follow easily. We therefore produce a spectral sequence of the form \ref{THH-May ss coeff} in the same way as before. 
\end{remark}
\section{The Bousfield-Kan spectral sequence.}\label{BK}
\subsection{The BKSS and nulhomotopic maps out of diagrams.} \label{BKss}
\begin{definition}\label{def of sr}
Given a small category $\mathcal{D}$, a category $\mathcal{A}$, and a functor $\mathcal{F}: \mathcal{D} \rightarrow\mathcal{A}$, we adopt these notations:
we will write $\mathcal{D}^0$ for the set of objects of $\mathcal{D}$, we will write $\mathcal{D}^1$ for the set of morphisms in $\mathcal{D}$, and if $n$ is a positive integer, we will write $\mathcal{D}^n$ for the set of composable ordered $n$-tuples of morphisms in $\mathcal{D}$,
given a composable ordered $n$-tuples of morphisms $d = \left( X_0 \stackrel{f_1}{\longrightarrow} \dots \stackrel{f_n}{\longrightarrow} X_n\right)$ in $\mathcal{D}$, we will write $\tilde{d}$ for $X_n$,
we will write $\sr(\mathcal{F})$ for the {\em simplicial replacement of $\mathcal{F}$}, that is, the simplicial object of $\mathcal{A}$ given by 
\[\xymatrix{
\coprod_{d\in \mathcal{D}^0} \mathcal{F}(\tilde{d})
 \ar[r] & 
  \coprod_{d\in \mathcal{D}^1} \mathcal{F}(\tilde{d}) \ar@<1ex>[l]\ar@<-1ex>[l]\ar@<1ex>[r]\ar@<-1ex>[r] &
  \coprod_{d\in \mathcal{D}^2} \mathcal{F}(\tilde{d}) \ar@<2ex>[l]\ar@<-2ex>[l]\ar[l]
                                 \ar@<2ex>[r]\ar@<-2ex>[r]\ar[r] &
   \dots \ar@<3ex>[l]\ar@<-3ex>[l]\ar@<1ex>[l]\ar@<-1ex>[l]   }\]
with face and degeneracy maps induced by composition and inserting of an identity morphism, respectively, as operations on composable ordered tuples of morphisms in $\mathcal{D}$. (This construction is standard; see eg section XI.5.1 of~\cite{MR0365573}.)
\end{definition}

The following theorem seems to have had a long history: it seems as though it was understood in some form, at least in the setting of simplicial sets, to Bousfield and Kan when they wrote~\cite{MR0365573}, and a proof is sketched in Corollary~9.8 and Proposition~9.11 of Dugger's unpublished notes~\cite{dugco}, limited to the setting of topological spaces, but clearly using the technology of~\cite{MR584566} which applies to much more general model categories. Finally, a really clear and general treatment can be derived from Theorem 3.2 and Theorem 3.3 of Gambino's paper \cite{MR2586997}, as explained in Section~4 of \cite{MR2586997}. Gambino's result only requires that $\mathcal{C}$ is a simplicial model category.
\begin{theorem}\label{thm from gambino}
Suppose that $\mathcal{D}$ is a small category such that the projective model structure on $\mathcal{C}^{\mathcal{D}}$ exists. Suppose that $\mathcal{F}:\mathcal{D}\rightarrow\mathcal{C}$ is projectively cofibrant. Then the natural map $\left| \sr(\mathcal{F}) \right| \rightarrow \colim\mathcal{F}$ in $\mathcal{C}$ is a weak equivalence.
\end{theorem}
Theorem \ref{thm from gambino} is our justification for the following notational convention: in this section and the following section, we will write $\hocolim$ for the Bousfield-Kan model for the homotopy colimit, as in \cite{MR0365573}; that is, $\hocolim \mathcal{F} = \left| \sr(\mathcal{F})\right|$.

\begin{definition}\label{def of E}
Let $\mathcal{D}$ be a small category, and write $\Ab$ for the category of abelian groups.
We will write $E$ for the allowable class, in the sense of relative homological algebra (see chapter~IX of~\cite{MR1344215}), on the category of functors $\Ab^{\mathcal{D}}$, given as follows:
a short exact sequence $0 \rightarrow \mathcal{F}^{\prime} \rightarrow \mathcal{F} \rightarrow \mathcal{F}^{\prime\prime} \rightarrow 0$ in $\Ab^{\mathcal{D}}$ is in $E$ if and only if the short exact sequence of abelian groups $0 \rightarrow \mathcal{F}^{\prime}(d) \rightarrow \mathcal{F}(d) \rightarrow \mathcal{F}^{\prime\prime}(d) \rightarrow 0$ is split for all objects $d$ of $\mathcal{D}$.
\end{definition}

In the case where the underlying model category is simplicial sets (not a stable model category!), the spectral sequence of Theorem~\ref{bkss existence thm} was constructed by Bousfield and Kan in chapter~XII of~\cite{MR0365573}. It is a widespread bit of folklore than the construction also works in more general model categories (see eg the discussion preceding Proposition~A.10 in~\cite{hoveyss}), so we give only a sketch of the proof of Theorem~\ref{bkss existence thm}. 

In Theorem~\ref{bkss existence thm}, as everywhere in this paper, we let $\mathcal{C}$ be as in Running Assumption~\ref{ra:1}, but in the proof we sketch, we only use the assumptions that $\mathcal{C}$ is cocomplete, stable, closed and simplicial, and left proper.
\begin{theorem}\label{bkss existence thm} {\bf (The Bousfield-Kan spectral sequence (BKSS) for generalized cohomology of a homotopy colimit.)}
Let $\mathcal{D}$ be a small category and let $F: \mathcal{D} \rightarrow\mathcal{C}$ be a functor such that $F(d)$ is cofibrant for all $d\in\ob\mathcal{D}$. Let $E$ be as in Definition~\ref{def of E}. Let $X$ be a fibrant object of $\mathcal{C}$. 
Then there exists a spectral sequence
\begin{equation}\label{bkss 1} E^2_{s,t} \cong (R^s_E\lim_{d\in\mathcal{D}})\left( [\Sigma^{t} F(d), X]\right)  \Rightarrow \left[\Sigma^{t-s}\hocolim F, X\right]  
\end{equation}
which is strongly convergent if the functor $\lim \co \Ab^{\mathcal{D}} \rightarrow\Ab$ is of finite $E$-injective dimension (ie, if there exists $N\in\mathbb{Z}$ such that $(R^s_E\lim)(\mathcal{F})$ vanishes for all $s>N$ and all $\mathcal{F}\in\ob\Ab^{\mathcal{D}}$).
\end{theorem}
\begin{proof}[Sketch of proof:]
It is routine to show that the cofibrancy assumption on each $F(d)$ is enough to imply that the simplicial object $\sr(F)$ of $\mathcal{C}$ is Reedy-cofibrant, and Consequently that the geometric realization $\left| \sr(F)\right|$ is a homotopy colimit for $F$.
Reedy cofibrancy plays another role here, however: 
the latching comparison maps $L_n\sr(F) \rightarrow \sr(F)_n$ are each cofibrations in $\mathcal{C}$, and Consequently (by Quillen's ``condition SM7(b),'' which holds for any closed simplicial model category by Proposition~2.3 of~\cite{MR0223432}) the natural map $(L_n\sr(F)\otimes \Delta^n) \coprod_{L_n\sr(F)\otimes \delta\Delta^n} (\sr(F)_n\otimes \delta \Delta^n) \rightarrow \sr(F)_n \otimes \Delta^n$ is a cofibration in $\mathcal{C}$. By Lemma~3.1 of~\cite{reedy}, we can build 
$\left| \sr(F)\right|$ as a colimit $\colim_n \left| \sr(F)\right|_n$ of objects $\left| \sr(F)\right|_n$ given inductively by
$\left|\sr(F)\right|_0 = \sr(F)_0$ and for each positive integer $n$ a pushout square 
\begin{equation}\label{reedy square 32099}\xymatrix{
 (L_n\sr(F)\otimes \Delta^n) \coprod_{L_n\sr(F)\otimes \delta\Delta^n} (\sr(F)_n\otimes \delta \Delta^n) \ar[r]\ar[d] & \sr(F)_n \otimes \Delta^n \ar[d] \\
 \left| \sr(F) \right|_{n-1} \ar[r] & \left| \sr(F) \right|_{n}.
}\end{equation}
Now since each $F(d)$ is assumed cofibrant, $\coprod_{d\in \mathcal{D}^0}F(d) = \sr(F)_0 = \left| \sr(F)\right|_0$ is also cofibrant. Consequently, when $n=1$, square~\ref{reedy square 32099} is a pushout square in which the two objects on the left are cofibrant, and the top vertical map is a cofibration, so left properness of $\mathcal{C}$ implies that square~\ref{reedy square 32099} is also a homotopy pushout square, and that $\left| \sr(F) \right|_0 \rightarrow \left| \sr(F) \right|_1$ is a cofibration, hence that $\left| \sr(F) \right|_1$ is cofibrant.

That was the initial step in an induction. The inductive step works as follows: since $0 \otimes \delta\Delta^n \cong 0\otimes \Delta^n \cong 0$, another application of Quillen's condition SM7(b) gives us that, since $L_n\sr(F)$ is cofibrant (as it is a colimit of cofibrant objects), the map $L_n\sr(F)\otimes \delta\Delta^n \rightarrow L_n\sr(F) \otimes\Delta^n$ is a cofibration, and hence $L_n\sr(F)\otimes \Delta^n \coprod_{L_n\sr(F)\otimes \delta\Delta^n} \sr(F)_n\otimes \delta \Delta^n$ is cofibrant, again using left properness of $\mathcal{C}$. So, if we have already shown that $\left| \sr(F) \right|_{n-1}$ is cofibrant, then the bottom horizontal map in~\ref{reedy square 32099} is a pushout of a cofibration, hence is a cofibration, hence $\left| \sr(F)\right|_n$ is cofibrant; and furthermore the square~\ref{reedy square 32099} is a homotopy pushout square, with horizontal cofiber $\Sigma^n \sr(F)_n$ by an argument dual to Proposition~X.6.3 of~\cite{MR0365573}.  By induction, we have a tower of cofiber sequences
\[ \xymatrix{ \left| \sr(F)\right|_0 \ar[r]\ar[d]^{\cong} & \left| \sr(F)\right|_1 \ar[r]\ar[d] & \left| \sr(F)\right|_2 \ar[r]\ar[d] & \dots \\
 \sr(F)_0 & \Sigma \sr(F)_1 & \Sigma^2 \sr(F)_2 & }\]
and, applying the generalized cohomology theory on $\Ho(\mathcal{C})$ represented by $X$, we get spectral sequence~\eqref{bkss 1}. 
\end{proof}

Similarly, we have the following theorem:
\begin{theorem}\label{homological bkss existence thm} {\bf (The Bousfield-Kan spectral sequence (BKSS) for generalized homology of a homotopy colimit.)}
Let $\mathcal{D}$ be a small category and let $F: \mathcal{D} \rightarrow\mathcal{C}$ be a functor such that $F(d)$ is cofibrant for all $d\in\ob\mathcal{D}$. Let $E$ be as in Definition~\ref{def of E}. Let $H_*$ be a connective generalized homology theory on $\mathcal{C}$, 
as defined in Definition~\ref{def of gen hom thy}.
Then there exists a strongly convergent spectral sequence
\begin{equation*}
E^2_{s,t} \cong (L^E_s\underset{d\in\mathcal{D}}{\colim})\left( H_t(F(d))\right)  \Rightarrow H_{s+t}\hocolim F.
\end{equation*}
\end{theorem}
\begin{proof}
Essentially the same argument as in Theorem~\ref{bkss existence thm}, which is only a generalization of the proof of XII.5.7 of~\cite{MR0365573}. Since $\hocolim F \cong \left| \sr(F)\right|$ is the cofiber of the map
\[ \coprod_{n\in \mathbb{N}} \left| \sr(F) \right|_n \stackrel{\id - T}{\longrightarrow} \coprod_{n\in \mathbb{N}} \left| \sr(F) \right|_n,\]
the exactness and additivity conditions in Definition~\ref{def of gen hom thy} ensure that the homology of $\hocolim F$ is the colimit of the homologies of the finite stages $\left| \sr(F)\right|_n$, ie, the spectral sequence computes $H_*(\hocolim F)$ as claimed.
Strong convergence is automatic since this is a half-plane spectral sequence with exiting differentials, as in~\cite{MR1718076}. 
\end{proof}

The rest of this section is occupied with an application of Theorem~\ref{bkss existence thm}. Here is the basic problem to be solved:
\begin{question}\label{question on nulhomotopies}
Suppose we are given a diagram of objects in our model category $\mathcal{C}$, and a map from the diagram to some particular object $X$. Suppose that the map from each of the objects in the diagram to $X$ is nulhomotopic. Do all these maps from objects in the diagram to $X$ factor through some homotopy colimit of the diagram, such that the map from the homotopy colimit to $X$ is also nulhomotopic?
\end{question} 
The answer to this question is not always yes: it depends on the diagram, and on $X$, and on the maps to $X$! The only tool we really need to answer Question~\ref{question on nulhomotopies} is Theorem~\ref{bkss existence thm}, so while Question~\ref{question on nulhomotopies} does not appear in~\cite{MR0365573} or in~\cite{reedy} or in any other reference we know of, nevertheless the question seems natural enough (and, in some situations, unavoidable enough) and its solution, in Theorem~\ref{answer to question}, is a direct enough application of ideas from the 1970s that we do not claim that anything in this section is really ``new.''
\begin{theorem}\label{answer to question}
Let $\mathcal{C}$ be as in Running Assumptions~\ref{ra:1}. Suppose that $\mathcal{D}$ is a small category, $F: \mathcal{D} \rightarrow \mathcal{C}$ a functor such that $F(d)$ is cofibrant for each $d\in\ob\mathcal{D}$, $X$ a fibrant object of $\mathcal{C}$ and $\overline{X}: \mathcal{D} \rightarrow\mathcal{C}$ the constant $X$-valued functor on $\mathcal{C}$, and $\eta: F \rightarrow \overline{X}$ is a natural transformation such that $\eta(d)$ is nulhomotopic for each $d\in\ob\mathcal{D}$. Let $E$ be as in Definition~\ref{def of E}, and suppose that the functor $\lim: \Ab^{\mathcal{D}} \rightarrow \Ab$ has finite $E$-injective dimension. 
If $(R^n_E\lim_{d\in \mathcal{D}})\left( [\Sigma^{n}F(d), X]\right) \cong 0$ for all $n>0$, then the resulting map $\hocolim F \rightarrow X$ is also nulhomotopic.
\end{theorem}
\begin{proof}
The terms in the $E^2$-page of spectral sequence~\eqref{bkss 1} which can contribute to $[\Sigma^0 \hocolim F, X]$ in the $E^{\infty}$-page are the terms of the form $(R^n_E\lim_{d\in\mathcal{D}})\left( [\Sigma^{n} F(d), X]\right)$. Consequently the vanishing hypothesis in the statement of the theorem ensures that the projection map $[\hocolim F, X] \rightarrow (R^0_E\lim_{d\in\mathcal{D}})\left( [F(d),X]\right) \cong \lim_{d\in\mathcal{D}}[F(d),X]$ is an isomorphism, and Consequently that the map induced by $\eta$, which represents zero in $\lim_{d\in\mathcal{D}}[F(d),X]$, also represents zero in $[\hocolim F, X]$.
\end{proof}

\subsection{Application to truncated cubes.}
In this subsection, we prove a few lemmas on a particular special case of Theorem~\ref{answer to question} which arises in \cref{whitehead}. The main result here is Lemma~\ref{main lemma}, which we use in our proof of Theorem~\ref{post filt}.
As in \cref{whitehead}, we will write $\Sp$ for the category of symmetric spectra in pointed simplicial sets with the positive flat stable model structure. We will write $\Ab$ for the category of abelian groups.
\begin{definition}\label{small cat}
Let $\mathcal{D}$ be the small category with all objects and non-identity morphisms drawn below:\\
\begin{center}
\begin{tikzpicture}[baseline= (a).base]
\node[scale=.7] (a) at (0,0){
\begin{tikzcd}
&
(1,1,1) 
\ar{dl}[swap, sloped, near start]{}
\ar{rr}{}
\ar[]{dd}[near end]{}
& & (1,1,0)
\ar{dd}{}
\ar{dl}[swap, sloped, near start]{}
\\
(0,1,1)
\ar[crossing over]{rr}[near start]{}
\ar{dd}[swap]{}
& & (0,1,0)
\\
&
(1,0,1)
\ar[near start]{rr}{}
\ar[sloped, near end]{dl}{}
& & (1,0,0)

\\
(0,0,1).

& & 

\end{tikzcd}
};
\end{tikzpicture}
\end{center}
\end{definition}
\begin{lemma}\label{relative rlim vanishing lemma}
Let $\mathcal{F}: \mathcal{D}^{\op} \rightarrow \Ab$ be a functor. Then the following statements are all true:
\begin{itemize} 
\item $R^0_E\lim\mathcal{F} \cong \lim \mathcal{F}$ is a subgroup of $F(1,0,0)\oplus F(0,1,0)\oplus F(0,0,1)$,
\item $R^1_E\lim\mathcal{F}$ is a subquotient of 
\[ \mathcal{F}(1,1,1)^{\oplus 3} \oplus \mathcal{F}(1,1,0)^{\oplus 2} \oplus \mathcal{F}(1,0,1)^{\oplus 2} \oplus \mathcal{F}(0,1,1)^{\oplus 2} ,\]
\item $R^2_E\lim\mathcal{F}$ is a subquotient of $\mathcal{F}(1,1,1)$,
\item and $R^n_E\lim\mathcal{F}$ vanishes for all $n>2$.
\end{itemize}
\end{lemma}
\begin{proof}
That $R^0_E\lim \mathcal{F}\cong \lim \mathcal{F}$ is a standard consequence of $\lim$ being left exact. For the remaining statements,
we first find all the relative cofree objects; ie the functors $\Ab^{\mathcal{D}^{\op}}$ that are in the image of the cofree functor $C: \Ab^{\times|\mathcal{D}^{\op}|}=\Ab^{\ob \mathcal{D}^{\op}} \rightarrow \Ab^{\mathcal{D}^{\op}}$, which is right adjoint to the functor $(ev_{d})_{d\in \mathcal{D}^{\op}}$. The relative cofree objects are the products of the diagrams of the form $C_i\in \Ab^{\mathcal{D}^{\op}}$: \footnotesize
\[ C_{111}=\left [
\begin{tikzpicture}[baseline= (a).base]
\node[scale=.6] (a) at (0,0){
\begin{tikzcd}
&
X
& & X \ar{ll}{\id}
\\
X \ar{ur}[swap, sloped, near start]{\id}
& & X \ar[crossing over]{ll}[near start]{\id} \ar{ur}[swap, sloped, near start]{\id}
\\
&
X \ar[]{uu}[near end]{\id}
& & X \ar[near start]{ll}{\id} \ar{uu}{\id}
\\
X \ar[sloped, near end]{ur}{\id} \ar{uu}[swap]{\id}
& & 
\end{tikzcd}
};
\end{tikzpicture}
\right ], 
C_{101}=
\left [
\begin{tikzpicture}[baseline= (a).base]
\node[scale=.6] (a) at (0,0){
\begin{tikzcd}
&
0
& & 0 \ar{ll}{\id}
\\
0 \ar{ur}[swap, sloped, near start]{\id}
& & 0 \ar[crossing over]{ll}[near start]{\id} \ar{ur}[swap, sloped, near start]{\id}
\\
&
X \ar[]{uu}[near end]{0}
& & X \ar[near start]{ll}{\id} \ar{uu}{0}
\\
X \ar[sloped, near end]{ur}{\id} \ar{uu}[swap]{0}
& & 
\end{tikzcd}
};
\end{tikzpicture}
\right ], 
C_{011}=
\left [
\begin{tikzpicture}[baseline= (a).base]
\node[scale=.6] (a) at (0,0){
\begin{tikzcd}
&
0
& & 0 \ar{ll}{\id}
\\
X \ar{ur}[swap, sloped, near start]{0}
& & X \ar[crossing over]{ll}[near start]{\id} \ar{ur}[swap, sloped, near start]{0}
\\
&
0 \ar[]{uu}[near end]{\id}
& & 0 \ar[near start]{ll}{\id} \ar{uu}{\id}
\\
X \ar[sloped, near end]{ur}{0} \ar{uu}[swap]{\id}
& & 
\end{tikzcd}
};
\end{tikzpicture}
\right ], \]
\[
C_{110}=
\left [
\begin{tikzpicture}[baseline= (a).base]
\node[scale=.6] (a) at (0,0){
\begin{tikzcd}
&
0
& & X \ar{ll}{\id}
\\
0 \ar{ur}[swap, sloped, near start]{\id}
& & X \ar[crossing over]{ll}[near start]{0} \ar{ur}[swap, sloped, near start]{\id}
\\
&
0 \ar[]{uu}[near end]{\id}
& & X \ar[near start]{ll}{0} \ar{uu}{\id}
\\
0 \ar[sloped, near end]{ur}{\id} \ar{uu}[swap]{\id}
& & 
\end{tikzcd}
};
\end{tikzpicture}
\right ],
C_{100}=
\left [
\begin{tikzpicture}[baseline= (a).base]
\node[scale=.6] (a) at (0,0){
\begin{tikzcd}
&
0
& & 0 \ar{ll}{\id}
\\
0 \ar{ur}[swap, sloped, near start]{\id}
& & 0 \ar[crossing over]{ll}[near start]{\id} \ar{ur}[swap, sloped, near start]{\id}
\\
&
0 \ar[]{uu}[near end]{\id}
& & X \ar[near start]{ll}{0} \ar{uu}{0}
\\
0 \ar[sloped, near end]{ur}{\id} \ar{uu}[swap]{\id}
& & 
\end{tikzcd}
};
\end{tikzpicture}
\right ],
C_{001}=
\left [
\begin{tikzpicture}[baseline= (a).base]
\node[scale=.6] (a) at (0,0){
\begin{tikzcd}
&
0
& & 0 \ar{ll}{\id}
\\
0 \ar{ur}[swap, sloped, near start]{\id}
& & 0 \ar[crossing over]{ll}[near start]{\id} \ar{ur}[swap, sloped, near start]{\id}
\\
&
0 \ar[]{uu}[near end]{\id}
& & 0 \ar[near start]{ll}{\id} \ar{uu}{\id}
\\
X \ar[sloped, near end]{ur}{0} \ar{uu}[swap]{0}
& & 
\end{tikzcd}
};
\end{tikzpicture}
\right ],\]
\[
C_{010}=
\left [
\begin{tikzpicture}[baseline= (a).base]
\node[scale=.6] (a) at (0,0){
\begin{tikzcd}
&
0
& & 0 \ar{ll}{\id}
\\
0 \ar{ur}[swap, sloped, near start]{\id}
& & X \ar[crossing over]{ll}[near start]{0} \ar{ur}[swap, sloped, near start]{0}
\\
&
0 \ar[]{uu}[near end]{\id}
& &  0 \ar[near start]{ll}{\id} \ar{uu}{\id}
\\
0 \ar[sloped, near end]{ur}{\id} \ar{uu}[swap]{\id}
& & 
\end{tikzcd}
};
\end{tikzpicture}
\right ]
\]
\normalsize
which we can think of as functors $C_{ijk}:\Ab\rightarrow \Ab^{\mathcal{D}^{\op}}$.

Now we resolve $\mathcal{F}$ by relative cofree objects:
we have a long exact sequence 
\begin{equation}\label{resolution 21039} \xymatrix{
0 \ar[r] & \mathcal{F} \ar[r] & M_0 \ar[r] & M_1 \ar[r] &M_2 \ar[r] & 0 }\end{equation}
where $M_i$ are defined as follows:
\[ M_0= \begin{array}{l} \coprod_{i,j,k}C_{ijk}(\mathcal{F}(i,j,k)) 
\end{array} \]
\[ M_1= \begin{array}{l} 
C_{011}(\mathcal{F}(1,1,1))\oplus C_{101}(\mathcal{F}(1,1,1))\oplus C_{110}(\mathcal{F}(1,1,1)) \\ \oplus C_{001}(\mathcal{F}(1,0,1))  \oplus C_{100}(\mathcal{F}(1,1,0))\oplus 
C_{010}(\mathcal{F}(0,1,1)) \\ \oplus C_{010}(\mathcal{F}(1,1,0))\oplus C_{100}(\mathcal{F}(0,1,1)) \oplus C_{100}(\mathcal{F}(1,0,1)) 
\end{array} \]
\[ M_2= C_{001}(\mathcal{F}(1,1,1))\oplus \\
C_{010}(\mathcal{F}(1,1,1))\oplus \\
C_{100}(\mathcal{F}(1,1,1))  .\]
Clearly \eqref{resolution 21039} is a resolution of $\mathcal{F}$ by $E$-projective objects in $\Ab^{\mathcal{D}^{\op}}$. For each object $d$ of $\mathcal{D}$, each morphism in \eqref{resolution 21039} has the property that its evaluation at $d$ is the composite, in $\Ab$, of a split epimorphism followed by a split monomorphism. In other words, \eqref{resolution 21039} is an $E$-resolution of $\mathcal{F}$, and we can use it to compute right $E$-derived functors. (See Chapter IX of \cite{MR1344215} for these ideas from relative homological algebra.)

So, to compute $R^*_E\lim \mathcal{F}$, we can omit $\mathcal{F}$ from \eqref{resolution 21039} and apply $\lim$, producing the cochain complex
\begin{equation}\label{resolution 21040} \xymatrix{ 0 \ar[r] & N_0 \ar[r]^{} & N_1
\ar[r]^{}& N_2 \ar[r] & 0 }\end{equation}
where $N_i$ are defined as follows:
\[ N_0=\begin{array}{l} 
\mathcal{F}(1,1,1)\oplus \mathcal{F}(1,1,0)\oplus \mathcal{F}(1,0,1)\oplus \mathcal{F}(0,1,1) \\ \oplus 
\mathcal{F}(1,0,0)\oplus \mathcal{F}(0,1,0)\oplus \mathcal{F}(0,0,1) 
\end{array} \]
\[ N_1= \begin{array}{l}
\mathcal{F}(1,1,1)\oplus \mathcal{F}(1,1,1)\oplus \mathcal{F}(1,1,1) \oplus \mathcal{F}(1,0,1)\oplus \mathcal{F}(1,1,0)\\ \oplus 
\mathcal{F}(0,1,1)\oplus \mathcal{F}(1,1,0)\oplus \mathcal{F}(0,1,1)\oplus \mathcal{F}(1,0,1) 
\end{array} \]
\[ N_2= \mathcal{F}(1,1,1)\oplus \mathcal{F}(1,1,1)\oplus \mathcal{F}(1,1,1).\]
Consequently $R^n_E\lim\mathcal{F}$ is a subquotient as claimed, and vanishes for all $n>2$. (One could also compute the maps in \eqref{resolution 21040} to get an explicit presentation of $R^1_E\lim\mathcal{F}$ and $R^2_E\lim\mathcal{F}$, but while interesting, this is unnecessary for the present paper.)
\end{proof}

\begin{lemma}\label{main lemma}
Suppose we have a commutative diagram in $\Sp$\\
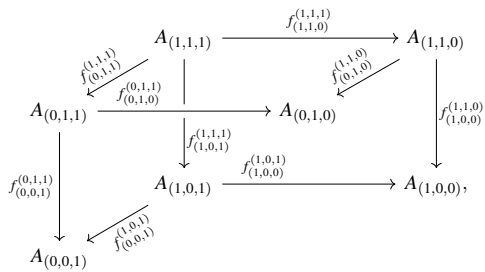
\begin{center}
\begin{tikzpicture}[baseline= (a).base] 
\label{diagram 7832}
\node[scale=.7] (a) at (0,0){
\begin{tikzcd}
&
A_{(1,1,1)} 
\ar{dl}[swap, sloped, near start]{f_{(0,1,1)}^{(1,1,1)}}
\ar{rr}{f_{(1,1,0)}^{(1,1,1)}}
\ar[]{dd}[near end]{f_{(1,0,1)}^{(1,1,1)}}
& & A_{(1,1,0)} 
\ar{dd}{f_{(1,0,0)}^{(1,1,0)}}
\ar{dl}[swap, sloped, near start]{f_{(0,1,0)}^{(1,1,0)}}
\\
A_{(0,1,1)} 
\ar[crossing over]{rr}[near start]{f_{(0,1,0)}^{(0,1,1)}}
\ar{dd}[swap]{f_{(0,0,1)}^{(0,1,1)}}
& & A_{(0,1,0)} 
\\
&
A_{(1,0,1)} 
\ar[near start]{rr}{f_{(1,0,0)}^{(1,0,1)}}
\ar[sloped, near end]{dl}{f_{(0,0,1)}^{(1,0,1)}}
& & A_{(1,0,0)}, \\
A_{(0,0,1)} 
& & 
\end{tikzcd}
};
\end{tikzpicture}
\end{center}
and suppose there is a fibrant symmetric spectrum $Z$ satisfying the conditions\footnote{Recall that $Z^{-m}(A_{(i,j,k)})$ is a standard notation for $[\Sigma^{m}A_{(i,j,k)},Z]$.}
\begin{itemize}
\item $Z^{-1}(A_{(i,j,k)})$ vanishes whenever $i+j+k = 2$, and
\item $Z^{-1}(A_{(1,1,1)})$ and $Z^{-2}(A_{(1,1,1)})$ vanish.
\end{itemize}
Regard \eqref{diagram 7832} as a functor $A: \mathcal{D} \rightarrow \Sp$. 
Then the natural map of abelian groups $[\hocolim A, Z] \rightarrow \lim_{d\in \mathcal{D}^{\op}} [A(d), Z]$ is an isomorphism.

In particular, if we make the additional assumption that $Z^0(A_{(i,j,k)})$ vanishes whenever $i+j+k = 1$, then $[\hocolim A, Z]$ vanishes.
\end{lemma}

\begin{proof}
By Theorem \ref{answer to question}, the obstruction to $[\hocolim A, Z] \rightarrow \lim_{d\in \mathcal{D}^{\op}} [A(d), Z]$ being an isomorphism is the groups 
\[(R^s_E\lim_{d\in\mathcal{D}})\left( Z^{-s}A(d)\right)\] 
for $s>0$. By Lemma~\ref{relative rlim vanishing lemma}, $(R^1_E\lim_{d\in\mathcal{D}})\left( Z^{-1}A(d)\right)$ vanishes as long as $Z^{-1}A_{i,j,k}$ vanishes for all triples $(i,j,k)$ satisfying $i+j+k\geq 2$.
Also by Lemma~\ref{relative rlim vanishing lemma}, $(R^2_E\lim_{d\in\mathcal{D}})\left( Z^{-2}A(d)\right)$ vanishes as long as $Z^{-2}A_{1,1,1}$ vanishes. Finally, the group $(R^s_E\lim_{d\in\mathcal{D}})\left( Z^{-s}A(d)\right)$ vanishes for all $s>2$. This proves the first claim. If we additionally assume that $Z^0(A_{(i,j,k)})$ vanishes whenever $i+j+k = 1$, then $\lim_{d\in \mathcal{D}^{\op}} [A(d), Z]$ vanishes and Consequently $[\hocolim A, Z]$ vanishes as well. \end{proof}

\bibliographystyle{plain}
\bibliography{/Users/gabrielangelini-knoll/Library/texmf/bibtex/bib/salch}{}

\def\cprime{$'$} \def\cprime{$'$}
\begin{thebibliography}{10}

\bibitem{161200548}
G.~{Angelini-Knoll}.
\newblock {On topological Hochschild homology of the $K(1)$-local sphere}.
\newblock {\em ArXiv e-prints}, December 2016.

\bibitem{161106215}
G.~{Angelini-Knoll} and A.~{Salch}.
\newblock {Maps of simplicial spectra whose realizations are cofibrations}.
\newblock {\em ArXiv e-prints}, November 2016.

\bibitem{angeltveitpreprint}
Vigleik Angeltveit.
\newblock On the algebraic {$K$}-theory of {W}itt vectors of finite length.
\newblock {\em preprint, available on arXiv}, 2011.

\bibitem{AHL10}
Vigleik Angeltveit, Michael~A. Hill, and Tyler Lawson.
\newblock Topological {H}ochschild homology of {$\ell$} and {$ko$}.
\newblock {\em Amer. J. Math.}, 132(2):297--330, 2010.

\bibitem{MR2183525}
Christian Ausoni.
\newblock Topological {H}ochschild homology of connective complex {$K$}-theory.
\newblock {\em Amer. J. Math.}, 127(6):1261--1313, 2005.

\bibitem{MR1732625}
M.~Basterra.
\newblock Andr\'e-{Q}uillen cohomology of commutative {$S$}-algebras.
\newblock {\em J. Pure Appl. Algebra}, 144(2):111--143, 1999.

\bibitem{MR3380069}
Marzieh Bayeh, Kathryn Hess, Varvara Karpova, Magdalena Kedziorek, Emily Riehl,
  and Brooke Shipley.
\newblock Left-induced model structures and diagram categories.
\newblock In {\em Women in topology: collaborations in homotopy theory}, volume
  641 of {\em Contemp. Math.}, pages 49--81. Amer. Math. Soc., Providence, RI,
  2015.

\bibitem{MR1718076}
J.~Michael Boardman.
\newblock Conditionally convergent spectral sequences.
\newblock In {\em Homotopy invariant algebraic structures ({B}altimore, {MD},
  1998)}, volume 239 of {\em Contemp. Math.}, pages 49--84. Amer. Math. Soc.,
  Providence, RI, 1999.

\bibitem{bok}
M.~B{\"o}kstedt.
\newblock The topological {H}ochschild homology of $\mathbb{Z}$ and of
  $\mathbb{Z}/p\mathbb{Z}$.
\newblock preprint, 1987.

\bibitem{MR0365573}
A.~K. Bousfield and D.~M. Kan.
\newblock {\em Homotopy limits, completions and localizations}.
\newblock Lecture Notes in Mathematics, Vol. 304. Springer-Verlag, Berlin,
  1972.

\bibitem{MR1750729}
M.~Brun.
\newblock Topological {H}ochschild homology of {$\mathbb{Z}/p^n$}.
\newblock {\em J. Pure Appl. Algebra}, 148(1):29--76, 2000.

\bibitem{MR1823499}
M.~Brun.
\newblock Filtered topological cyclic homology and relative {$K$}-theory of
  nilpotent ideals.
\newblock {\em Algebr. Geom. Topol.}, 1:201--230 (electronic), 2001.

\bibitem{MR1731415}
Henri Cartan and Samuel Eilenberg.
\newblock {\em Homological algebra}.
\newblock Princeton Landmarks in Mathematics. Princeton University Press,
  Princeton, NJ, 1999.
\newblock With an appendix by David A. Buchsbaum, Reprint of the 1956 original.

\bibitem{Day}
Brian~J. Day.
\newblock {\em Construction of biclosed categories}.
\newblock PhD thesis, University of New South Wales, 1970.

\bibitem{2003math......5173D}
D.~{Dugger}.
\newblock {Multiplicative structures on homotopy spectral sequences I}.
\newblock {\em ArXiv Mathematics e-prints}, May 2003.

\bibitem{dugco}
Daniel Dugger.
\newblock A primer on homotopy colimits.
\newblock available online.

\bibitem{MR1870516}
Daniel Dugger.
\newblock Combinatorial model categories have presentations.
\newblock {\em Adv. Math.}, 164(1):177--201, 2001.

\bibitem{MR584566}
W.~G. Dwyer and D.~M. Kan.
\newblock Function complexes in homotopical algebra.
\newblock {\em Topology}, 19(4):427--440, 1980.

\bibitem{MR0012228}
Samuel Eilenberg and Norman~E. Steenrod.
\newblock Axiomatic approach to homology theory.
\newblock {\em Proc. Nat. Acad. Sci. U. S. A.}, 31:117--120, 1945.

\bibitem{MR1417719}
A.~D. Elmendorf, I.~Kriz, M.~A. Mandell, and J.~P. May.
\newblock {\em Rings, modules, and algebras in stable homotopy theory},
  volume~47 of {\em Mathematical Surveys and Monographs}.
\newblock American Mathematical Society, Providence, RI, 1997.
\newblock With an appendix by M. Cole.

\bibitem{MR2586997}
Nicola Gambino.
\newblock Weighted limits in simplicial homotopy theory.
\newblock {\em J. Pure Appl. Algebra}, 214(7):1193--1199, 2010.

\bibitem{160201515}
O.~{Gwilliam} and D.~{Pavlov}.
\newblock {Enhancing the filtered derived category}.
\newblock {\em ArXiv e-prints}, February 2016.

\bibitem{MR2580430}
John~E. Harper.
\newblock Bar constructions and {Q}uillen homology of modules over operads.
\newblock {\em Algebr. Geom. Topol.}, 10(1):87--136, 2010.

\bibitem{MR1650134}
Mark Hovey.
\newblock {\em Model categories}, volume~63 of {\em Mathematical Surveys and
  Monographs}.
\newblock American Mathematical Society, Providence, RI, 1999.

\bibitem{hoveyss}
Mark Hovey.
\newblock Some spectral sequences in {M}orava {$E$}-theory.
\newblock {\em unpublished, available online}, 2004.

\bibitem{MR1695653}
Mark Hovey, Brooke Shipley, and Jeff Smith.
\newblock Symmetric spectra.
\newblock {\em J. Amer. Math. Soc.}, 13(1):149--208, 2000.

\bibitem{MR2713397}
Samuel~Baruch Isaacson.
\newblock {\em Cubical homotopy theory and monoidal model categories}.
\newblock ProQuest LLC, Ann Arbor, MI, 2009.
\newblock Thesis (Ph.D.)--Harvard University.

\bibitem{MR2177301}
G.~M. Kelly.
\newblock Basic concepts of enriched category theory.
\newblock {\em Repr. Theory Appl. Categ.}, (10):vi+137, 2005.
\newblock Reprint of the 1982 original [Cambridge Univ. Press, Cambridge;
  MR0651714].

\bibitem{MR981743}
Jean-Louis Loday.
\newblock Op\'erations sur l'homologie cyclique des alg\`ebres commutatives.
\newblock {\em Invent. Math.}, 96(1):205--230, 1989.

\bibitem{MR1344215}
Saunders Mac~Lane.
\newblock {\em Homology}.
\newblock Classics in Mathematics. Springer-Verlag, Berlin, 1995.
\newblock Reprint of the 1975 edition.

\bibitem{MR1712872}
Saunders Mac~Lane.
\newblock {\em Categories for the working mathematician}, volume~5 of {\em
  Graduate Texts in Mathematics}.
\newblock Springer-Verlag, New York, second edition, 1998.

\bibitem{MR1806878}
M.~A. Mandell, J.~P. May, S.~Schwede, and B.~Shipley.
\newblock Model categories of diagram spectra.
\newblock {\em Proc. London Math. Soc. (3)}, 82(2):441--512, 2001.

\bibitem{MR1867203}
J.~P. May.
\newblock The additivity of traces in triangulated categories.
\newblock {\em Adv. Math.}, 163(1):34--73, 2001.

\bibitem{MR2614527}
J.~Peter May.
\newblock {\em T{HE} {COHOMOLOGY} {OF} {RESTRICTED} {LIE} {ALGEBRAS} {AND} {OF}
  {HOPF} {ALGEBRAS}: {APPLICATION} {TO} {THE} {STEENROD} {ALGEBRA}}.
\newblock ProQuest LLC, Ann Arbor, MI, 1964.
\newblock Thesis (Ph.D.)--Princeton University.

\bibitem{MR1473888}
J.~McClure, R.~Schw{\"a}nzl, and R.~Vogt.
\newblock {$THH(R)\cong R\otimes S^1$} for {$E_\infty$} ring spectra.
\newblock {\em J. Pure Appl. Algebra}, 121(2):137--159, 1997.

\bibitem{MR1209233}
J.~E. McClure and R.~E. Staffeldt.
\newblock On the topological {H}ochschild homology of {$b{\rm u}$}. {I}.
\newblock {\em Amer. J. Math.}, 115(1):1--45, 1993.

\bibitem{MR1755114}
Teimuraz Pirashvili.
\newblock Hodge decomposition for higher order {H}ochschild homology.
\newblock {\em Ann. Sci. \'Ecole Norm. Sup. (4)}, 33(2):151--179, 2000.

\bibitem{MR0223432}
Daniel~G. Quillen.
\newblock {\em Homotopical algebra}.
\newblock Lecture Notes in Mathematics, No. 43. Springer-Verlag, Berlin-New
  York, 1967.

\bibitem{reedy}
Christopher Reedy.
\newblock Homotopy theory of model categories.
\newblock {\em preprint}, 1973.

\bibitem{MR1783629}
R.~Schw{\"a}nzl, R.~M. Vogt, and F.~Waldhausen.
\newblock Topological {H}ochschild homology.
\newblock {\em J. London Math. Soc. (2)}, 62(2):345--356, 2000.

\bibitem{MR1819881}
Stefan Schwede.
\newblock {$S$}-modules and symmetric spectra.
\newblock {\em Math. Ann.}, 319(3):517--532, 2001.

\bibitem{schwedebook}
Stefan Schwede.
\newblock Symmetric spectra book project.
\newblock {\em draft version}, available online.

\bibitem{MR1734325}
Stefan Schwede and Brooke~E. Shipley.
\newblock Algebras and modules in monoidal model categories.
\newblock {\em Proc. London Math. Soc. (3)}, 80(2):491--511, 2000.

\bibitem{MR3666740}
David White.
\newblock Model structures on commutative monoids in general model categories.
\newblock {\em J. Pure Appl. Algebra}, 221(12):3124--3168, 2017.

\end{thebibliography}
\end{document}